\documentclass{article}

\usepackage[utf8]{inputenc}
\usepackage[T1]{fontenc}
\usepackage{microtype}
\usepackage{amsmath,amssymb,mathtools}
\usepackage{graphicx}
\usepackage{latexsym}
\usepackage{pxfonts,kpfonts}
\usepackage{fullpage}
\usepackage{bbm}

\usepackage[style=alphabetic, sorting=nyt, maxnames=10, maxalphanames=4]{biblatex}
\usepackage[inline]{enumitem}
\setlist[itemize]{label={$\blacktriangleright$}}

\usepackage{tikz-cd}

\usepackage[numbered]{bookmark}
\usepackage{hyperref}
\hypersetup{
  pdfencoding=auto,
  colorlinks,
  pdfauthor = {Ricardo Campos, Julien Ducoulombier, and Najib Idrissi},
  pdftitle = {Boardman--Vogt resolutions and bar/cobar constructions of (co)operadic (co)bimodules}
}

\usepackage{amsthm}
\usepackage[all]{hypcap}

\theoremstyle{definition}
\newtheorem{defi}{Definition}[section]

\newtheorem{expl}[defi]{Example}
\newtheorem{notat}[defi]{Notation}

\theoremstyle{plain}
\newtheorem{pro}[defi]{Proposition}
\newtheorem{thm}[defi]{Theorem}
\newtheorem{cor}[defi]{Corollary}

\newtheorem{lmm}[defi]{Lemma}
\newtheorem{thmintro}{Theorem}

\theoremstyle{remark}

\newtheorem{rmk}[defi]{Remark}

\newcommand{\R}{\mathbb{R}}

\newcommand{\K}{\mathbb{K}}
\DeclareMathOperator{\Map}{Map}
\DeclareMathOperator{\Prim}{Prim}
\DeclareMathOperator{\Ind}{Indec}
\newcommand{\Seq}{\mathrm{Seq}}
\newcommand{\Sp}{\mathrm{Spec}}
\newcommand{\sSet}{\mathrm{sSets}}
\newcommand{\Operad}{\mathrm{Operad}}
\newcommand{\Cooperad}{\mathrm{Cooperad}}
\newcommand{\Hopf}{\mathrm{Hopf}}
\newcommand{\Bimod}{\mathrm{Bimod}}
\newcommand{\Cobimod}{\mathrm{Cobimod}}

\newcommand{\CDGA}{\mathrm{CDGA}}
\newcommand{\Ch}{\mathrm{Ch}}
\newcommand{\dg}{\mathrm{dg}}
\newcommand{\PL}{\mathrm{PL}}
\newcommand{\iso}{\mathrm{core}}

\newcommand{\calO}{\mathcal{O}}
\newcommand{\calC}{\mathcal{C}}
\newcommand{\calB}{\mathcal{B}}

\newcommand{\calP}{\mathcal{P}}
\newcommand{\calQ}{\mathcal{Q}}
\newcommand{\calF}{\mathcal{F}}
\newcommand{\calU}{\mathcal{U}}
\newcommand{\calM}{\mathcal{M}}

\newcommand{\calR}{\mathcal{R}}

\newcommand{\bfC}{\mathbf{C}}
\newcommand{\Tree}{\mathbb{T}}
\newcommand{\sTree}{s\mathbb{T}}
\newcommand{\lTree}{\mathbb{L}}
\newcommand{\slTree}{s\mathbb{L}}

\DeclareMathOperator{\lev}{lev}
\DeclareMathOperator{\ev}{ev}
\DeclareMathOperator{\id}{id}
\DeclareMathOperator{\Emb}{Emb}

\newcommand{\xtwoheadrightarrow}[2][]{%
  \xrightarrow[#1]{#2}\mathrel{\mkern-14mu}\rightarrow
}

\title{Boardman--Vogt resolutions and bar/cobar constructions of (co)operadic (co)bimodules}
\author{Ricardo Campos \and Julien Ducoulombier \and Najib Idrissi}
\date{July 3, 2021}

\begin{document}

\maketitle

\begin{abstract}
  We develop the combinatorics of leveled trees in order to construct explicit resolutions of (co)operads and (co)operadic (co)bimodules.
  We build explicit cofibrant resolutions of operads and operadic bimodules in spectra analogous to the ordinary Boardman--Vogt resolutions and we express them as cobar constructions of indecomposable elements.
  Dually, in the context of CDGAs, we perform similar constructions, and we obtain fibrant resolutions of Hopf cooperads and Hopf cooperadic cobimodules.
  We also express them as bar constructions of primitive elements.
\end{abstract}

\tableofcontents

\section{Introduction}

\subsection{Motivation}
\label{sec:motivation}

Given two smooth manifolds $M$ and $N$, the embedding space $\Emb(M,N)$ is the space of smooth embeddings $M \hookrightarrow N$ endowed with the compact-open topology.
Studying this space is a classical problem in topology.
Determining its homotopy type, or even its rational homotopy type, is a difficult problem.
For example, understanding the connected components of $\Emb(S^{1}, S^{3})$ is the object of knot theory.

Our approach towards the understanding of $\Emb(M,N)$ involves two key ingredients: Goodwillie--Weiss calculus and operad theory.
On the one hand, Goodwillie--Weiss calculus~\cite{GoodwillieWeiss1999} gives information about $\Emb(M,N)$ when $\dim N - \dim M \ge 3$.
It allows one to express $\Emb(M,N)$ as the limit of a tower of ``polynomial approximations'' $T_k \Emb(M,N)$ which are easier to compute.
On the other hand, operads are combinatorial objects that encode algebraic structures.
The little disks operads $E_{n}$, that were initially introduced in the study of iterated loop spaces~\cite{BoardmanVogt1968,May1972}, play a central role in this theory.
An element of $E_{n}$ is a configuration of numbered disjoint $n$-disks in the unit $n$-disk.
The operadic structure consists in plugging a configuration of disks inside one of the disks of another configuration.
Algebras over these little disks operads are precisely spaces having the homotopy type of iterated loop spaces (under some technical conditions).

Ways of computing the rational homotopy type of $\Emb(M,N)$ with combinatorial methods have been developed using these two ingredients.
In particular, for $M = \R^{d}$, $N = \R^{n}$ and $n-m\geq 3$, the space of compactly supported embeddings $\Emb_{c}(\R^{d}, \R^{n})$ was shown to be weakly equivalent to the $(d+1)$-iterated loop space of the derived mapping space of operads $\Operad^{h}(E_{d}, E_{n})$~\cite{DucoulombierTurchin2017,BoavidaWeiss2015}.
Using Sullivan's rational homotopy theory and the formality of the little disks operads, this derived mapping space can be expressed in terms of hairy graph complexes~\cite{FresseTurchinWillwacher2017}.
As a consequence, $\pi_{0}\Emb_{c}(\R^{d}, \R^{n})$ has been proved to be a finitely generated group of rank $\leq 1$.

One of the key steps in this computation was the construction of fibrant resolutions of Hopf cooperads (i.e.\ cooperads in commutative differential graded algebras) that can be expressed as cofree objects.
The fibrant resolution is used to compute the derived mapping space, and the fact that it is cofree allows one to reduce this mapping space to a mapping space of symmetric sequences with some differential.
One of the usual ways of providing resolutions of (co)operads is the Boardman--Vogt $W$ construction.
Fresse--Turchin--Willwacher~\cite{FresseTurchinWillwacher2017} managed to identify the Boardman--Vogt construction of a cooperad with the bar construction of an explicit operad (see also Berger--Moerdijk~\cite{BM}), therefore the $W$ resolution is in particular cofree.

Our goal is to provide tools to extend these computations to study the space of \emph{string links}, i.e.\ the space of compactly supported embeddings $\Emb_c(\R^{d_1}\sqcup\dots \sqcup \R^{d_k}, \R^n)$.
In the strategy outlined above, Goodwillie--Weiss calculus is replaced with multivariable Goodwillie--Weiss calculus.
With this multivariable version, the space of string links cannot be expressed as a mapping space of operads but rather as a mapping space of \emph{operadic bimodules}~\cite{Ducoulombier2018}.
If one thinks of an operad as a monoid in a certain (non-symmetric) monoidal category, then an operadic ($\calP$-$\calQ$)-bimodule corresponds to a right module over the monoid $\calQ$ and a left module over the monoid $\calP$.

In this paper, we extend the Boardman--Vogt resolution and the bar-cobar resolution to deal with (co)operadic (co)bimodules.
The main difference with the classical Boardman--Vogt resolution for (co)operads is the following.
Since an operad is equipped with a unit, the operadic structure maps can be expressed in terms of infinitesimal compositions, i.e.\ compositions where elements are composed one at a time instead of all at once.
Free constructions can thus be defined using planar trees, as we can contract edges independently.
However, for bimodules, this is not the case: only total compositions are available.
We are thus led to consider categories of \emph{leveled trees}, i.e., trees whose leaves are all at the same height (and with possibly bivalent vertices), instead of mere planar trees.
The introduction of levels allows us to have well-defined total edge contraction operations on trees, which encodes the combinatorics of bimodules better.
We moreover consider leveled trees \emph{with section}: the level of the section corresponds to the bimodule itself, the levels below the section correspond to the left module structure, and the levels above the section correspond to the right module structure.
As explained in Remark~\ref{rmk:batanin}, we use leveled trees because it is technically convenient for us.
At the cost of more complex definitions, one can adapt all our results to the setting of traditional planar trees.

Before dealing with (co)bimodules, we adapt the usual Boardman--Vogt resolutions and bar-cobar resolutions of (co)operads to use leveled trees.
While leveled trees had already been used to define bar-cobar constructions of (co)operads~\cite{Fresse2004,Livernet2012}, these earlier constructions did not inherit the structure that we need.
To solve this problem, we introduce additional morphisms in our categories of leveled trees.
(Let us note that our terminology differs from the one of~\cite{LodayVallette2012,MarklShniderStasheff2002} who use the name ``leveled trees'' for different kinds of trees.)
These extra morphisms consist in he permutation of some levels, if they satisfy a certain condition.
This flexibility allows us to define the appropriate structures on the Boardman--Vogt and bar-cobar resolutions.
We prove that, in the (co)operadic case, the bar and cobar constructions that we define are isomorphic to the usual bar and cobar constructions.
These adapted constructions on (co)operads are compatible with our resolutions for (co)bimodules.
Beyond bimodules, our results could extend to other situations in which unleveled trees do not work, e.g.\ the derived triple product $B(L,P,M,Q,R)$ considered by Arone--Ching~\cite{AroneChing2011}.

A particular case of the main result of~\cite{Ducoulombier2018} states that if $d_1=\dots =d_k=d$, then one can express the space of string links $\Emb_c(\R^d \sqcup \dots \sqcup \R^{d}, \R^n )$ in terms of the $d$-fold loop space of the bimodule derived mapping space $\Bimod^h_{E_{d,k}}(E_d \times \dots \times E_d, E_n)$, where $E_{d,k}$ is a certain colored operad obtained from the little disks operads.
In order to compute the rational homotopy type of this derived mapping space, it would be necessary to find appropriate resolutions of the operadic (co)bimodules involved.
In future work, we plan to use these resolutions and the formality of the little disks operads to express the rational homotopy type of the derived mapping space of bimodules above in terms of colored hairy graph complexes.

\subsection{Summary of results}
\label{sec:summary-results}

In Section~\ref{sec:model-categ-struct}, we recall background on operads, cooperads, operadic bimodules, and cooperadic cobimodules.
In each case, we describe the projective model category structures on the associated categories.
We also consider $\Lambda$ versions of these objects, i.e.\ we allow constants in our (co)operads (which gives extra structure on the (co)bimodules), and we describe the Reedy model category structures on the corresponding categories. The rest of the paper is split into $5$ sections.

\paragraph{Categories of trees}\label{para:categ-trees}
In Section~\ref{sec:invent-categ-trees}, we introduce the categories of trees that we consider in this paper:
\begin{itemize}
  \item[$\Tree$:]
        The category $\Tree[n]$ of planar trees with $n$ leaves whose morphisms are generated by isomorphisms of planar trees and contractions of two consecutive vertices.
        The family of categories $\Tree=\{\Tree[n],\, n\geq 1\}$ inherits an operadic structure given by compositions of trees.
        We also consider the subcategory $\Tree^{\geq 2}[n]$ composed of planar trees whose internal vertices all have at least two incoming edges.

  \item[$\lTree$:]
        The category $\lTree[n]$ of leveled trees with $n$ leaves whose morphisms are generated by isomorphisms of trees, contractions of consecutive levels and permutations of two consecutive ``permutable'' levels.
        The family $\lTree=\{\lTree[n],\, n\geq 1\}$ is equipped with a kind of operadic structure satisfying the operadic axioms up to contractions and permutations of permutable levels.
  \item[$\slTree$:]
        The category $\slTree[n]$ of leveled trees with section having $n$ leaves and whose morphisms are generated by isomorphisms of trees, contractions of consecutive levels and permutations of two consecutive permutable levels.
        The family $\slTree=\{\slTree[n],\,n\geq 1\}$ is equipped with a kind of ($\lTree$-$\lTree$)-bimodule structure satisfying the bimodule's axioms up to contractions and permutations of permutable levels.
\end{itemize}

It is well-know that operads can be constructed as algebras over a certain monad of planar trees, which gives a simple description of the free operad.
However, operads are also monoids with respect to the (non-symmetric) monoidal product on symmetric sequences $\circ$.
The iterated products $\mathcal{P}^{\circ n}$ produce naturally leveled trees.

Compared with earlier works that dealt with leveled trees, the key point in this paper is the introduction of a new type of morphisms.
These new morphisms consist in contracting or permuting consecutive levels called ``permutable'' and satisfying some conditions.
With the addition of these morphisms, the two categories $\Tree^{\geq 2}$ and $\lTree$ become almost the same.
We will show that all the constructions based on the category $\Tree^{\geq 2}$ (such as (co)free (co)operad functors, (co)bar constructions, $W$-constructions, etc.) can be extended to the category $\lTree$.
For each construction, both versions are isomorphic.
This is a consequence of the following statement:

\begin{thmintro}[Theorem~\ref{pro:now first theorem in intro}]\label{thm0}
  The functor $\alpha:\lTree[n] \rightarrow \Tree^{\geq 2}[n]$ which forgets the leveled

  structure and the bivalent vertices is full and surjective on objects.
  It admits an explicit right inverse $\beta:\Tree^{\geq 2}[n]\rightarrow \lTree[n]$ which is faithful and injective on objects.
\end{thmintro}

\paragraph{Resolutions of operads in spectra}

Ching~\cite{Chi,Chi2} studied Boardman--Vogt $W$ resolutions for operads in spectra (or more generally any category enriched in pointed simplicial sets).
For such an operad $\calO$ satisfying $\calO(0) = \varnothing$ and $\calO(1) = *$ (the singleton), he proved that $W\calO$ is weakly equivalent to the usual bar-cobar resolution $\Omega B \calO$, and that the usual bar construction $B\calO$ is weakly equivalent to the suspension of the cooperad of indecomposable elements $\Sigma \Ind(\calO)$.
Dual statements were also proved in the setting of commutative differential graded algebras (CDGAs) by Fresse~\cite{Fresse2004}.

In Section~\ref{sec:cofibr-resol-lambda}, we adapt the Boardman--Vogt resolution as well as the bar and cobar constructions to the setting of leveled trees.
We do these constructions for reduced operads, also denoted by $\Lambda$-operads, which are operads defined in arity $\ge 1$ equipped with  extra structure~\cite{Fresse2017}.
Our operads still satisfy $\calO(1) = *$, i.e.\ they are 1-reduced.
Thanks to the comparison Theorem~\ref{thm0}, we prove that these leveled constructions are isomorphic to the usual ones.
We can summarize this by:

\begin{thmintro}\label{thmA}
  Let $\calO$ be a 1-reduced $\Lambda$-operad in spectra.
  \begin{enumerate}[label={(\alph*)}, nosep]
    \item
          The leveled Boardman--Vogt resolution $W_{l}\calO$ is isomorphic to the usual Boardman--Vogt resolution $W\calO$.
          It is therefore a cofibrant resolution of $\calO$ in the Reedy model category structure of 1-reduced $\Lambda$-operads  (Proposition~\ref{prop:w-o-resolution}).
    \item
          The leveled bar construction $\calB_{l}\calO$ is isomorphic to the usual bar construction $\calB\calO$.
          Therefore, the indecomposables of $W_{l}\calO$ define a 1-reduced cooperad $\Ind(W_{l}\calO)$ in spectra and the leveled bar construction $\calB_{l}\calO$ is weakly equivalent to the suspension $\Sigma \Ind(W_{l}\calO)$ (Proposition~\ref{ProOpFr}).
    \item
          The leveled cobar construction $\Omega_{l}\mathcal{C}$ is isomorphic to the usual cobar construction $\Omega\mathcal{C}$ for any cooperad $\mathcal{C}$.
          In particular, the leveled Boardman--Vogt resolution $W_{l}\calO$ is weakly equivalent to the leveled cobar-bar construction $\Omega_{l}\calB_{l}\calO$ as a 1-reduced operad (Proposition~\ref{prop:bw-cobarbar}).
  \end{enumerate}
\end{thmintro}

Even if these results seem redundant with Ching's work, they represent an important step to treat the bimodule case.
Indeed, contrary to the operadic case, the only way to get similar constructions for bimodules is to use leveled trees.
The main purpose of Theorem~\ref{thmA} is to ensure that the techniques we develop for bimodules are compatible with the well-known constructions for operads.

\paragraph{Resolutions of operadic bimodules in spectra}

In Section~\ref{sec:cofibr-resol-lambda-1}, we extend these results to $\Lambda$-bimodules in spectra using leveled trees with section. We define the leveled Boardman--Vogt resolution $W_{l}$ for $\Lambda$-bimodules, and the leveled bar and cobar constructions $\calB_{l}[-,-]$, $\Omega_{l}[-,-]$.
We prove:

\begin{thmintro}\label{thmB}
  Let $\calP$ and $\calQ$ be two 1-reduced $\Lambda$-operads in spectra and let $M$ be a ($\calP$-$\calQ$)-bimodule.
  \begin{enumerate}[label={(\alph*)}, nosep]
    \item
          The leveled Boardman--Vogt resolution $W_{l}M$ defines a cofibrant resolution of $M$ in the Reedy model category structure of ($W_{l}\calP$-$W_{l}\calQ$)-bimodules (Proposition~\ref{pro:w-bimod}).
    \item
          The indecomposables of $W_{l}M$ define a ($\Ind(W_{l}\calP)$-$\Ind(W_{l}\calQ)$)-cobimodule $\Ind(W_{l}M)$ in spectra.
          The leveled two-sided bar construction $\calB_{l}[\calP,\calQ](M)$ is weakly equivalent to the suspension $\Sigma \Ind(W_{l}M)$ (Proposition~\ref{ProModFr}).
    \item
          The leveled Boardman--Vogt resolution $W_{l}M$ is weakly equivalent to the leveled cobar-bar construction as a ($W_{l}\calP$-$W_{l}\calQ$)-bimodule (Proposition~\ref{prop:gamma-bimodule}), i.e.:
          \[ W_{l}M \simeq \Omega_{l}[\calB_{l}\calP, \calB_{l}\calQ]\bigl( \calB_{l}[\calP,\calQ](M) \bigr). \]
  \end{enumerate}
\end{thmintro}

In that context, there is no analogue of our construction in the context of planar trees.
In~\cite{Duc}, the second author introduced resolutions for bimodules using planar trees.
However, this resolution gives rise to an element in the category of ($\calP$-$\calQ$)-bimodules instead of ($W\calP$-$W\calQ$)-bimodules.
The main obstruction comes from the left module operations, which are not linear.

Let us also remark that our constructions also work for pointed topological spaces (or more generally any cofibrantly generated model category enriched in pointed topological spaces with good finiteness assumptions).
The base point is notably crucial for the definition of the operadic module structure of the indecomposables.

\paragraph{Resolutions of Hopf cooperads}

The rational homotopy type of a 1-reduced simplicial $\Lambda$-operad is encoded by a 1-reduced $\Lambda$-cooperad in CDGAs (also known as a Hopf cooperad) thanks to results of Fresse~\cite{Fre} (which extend Sullivan's rational homotopy theory).
Fresse--Turchin--Willwacher~\cite{FresseTurchinWillwacher2017} built a fibrant resolution for 1-reduced Hopf $\Lambda$-cooperads using the Boardman--Vogt $W$ construction.
They identified the underlying dg-cooperad of this construction with the bar construction of the operad formed by the subspace of primitive elements (a key step in computing the rational homotopy type of embedding spaces).

In Section~\ref{sec:fibr-resol-hopf}, we define variants of these constructions using leveled trees.
We show that our constructions are quasi-isomorphic to the usual ones.
Moreover, we also prove that the leveled $W_{l}$ construction is quasi-isomorphic to the bar-cobar construction.
Like in the operadic case, the theorem below only provides a construction isomorphic to the usual one.
However, it is an important step to make our constructions for bimodules compatible with the classical theory.

\begin{thmintro}\label{thmC}
  Let $\calC$ be 1-reduced Hopf $\Lambda$-cooperad.
  \begin{enumerate}[label={(\alph*)}, nosep]
    \item The leveled Boardman--Vogt resolution $W_{l}\calC$ is isomorphic to the usual Boardman--Vogt resolution $W\calC$
          Consequently, $W_{l}\calC$ defines a fibrant resolution of $\calC$ in the Reedy model category structure of 1-reduced Hopf $\Lambda$-cooperads (see Theorem~\ref{ThmQI}).
    \item The leveled bar construction is isomorphic to the usual bar construction.

          Furthermore, the primitive elements of $W_{l}\calC$ define a 1-reduced dg-operad $\Sigma^{-1} \Prim(W_{l}\calC)$.
          Therefore, the underlying dg-cooperad of the Boardman--Vogt construction $W_{l}\calC$ is quasi-isomorphic to the leveled bar construction of the primitive elements $\calB_{l}(\Sigma^{-1} \Prim(W_{l}\calC))$ (Theorem~\ref{thm:W=B for bimods}).
    \item The leveled cobar construction is isomorphic to the usual cobar construction.
          So, the leveled Boardman--Vogt resolution $W_{l}\calC$ is quasi-isomorphic to the leveled bar-cobar construction $\calB_{l}\Omega_{l}\calC$ as a 1-reduced dg-$\Lambda$-cooperad (Theorem~\ref{thm:lambda-compatible}).
  \end{enumerate}
\end{thmintro}

\paragraph{Resolutions of Hopf cooperadic cobimodules}

In Section~\ref{sec:fibr-resol-hopf-2}, we extend the previous results to Hopf $\Lambda$-cobimodules over 1-reduced Hopf $\Lambda$-cooperads.
We define fibrant resolutions for such cobimodules using the Boardman--Vogt construction and leveled trees with section.
We also define leveled two-sided bar and cobar constructions for such cobimodules.

\begin{thmintro}\label{thmD}
  Let $\calP$ and $\calQ$ be two 1-reduced Hopf $\Lambda$-cooperads and let $M$ be a ($\calP$-$\calQ$)-cobimodule.
  \begin{enumerate}[label={(\alph*)}, nosep]
    \item The leveled Boardman--Vogt resolution $W_{l}M$ defines a fibrant resolution of $M$ in the Reedy model category structure of ($W_{l}\calP$-$W_{l}\calQ$)-cobimodules (Theorem~\ref{thm:w-lambda-cobim}).
    \item The primitive elements of $W_{l}M$ define a dg-($\Sigma^{-1}\Prim(W_{l}\calP)$-$\Sigma^{-1}\Prim(W_{l}\calQ)$)-bimodule $\Sigma^{-1}\Prim(W_{l}M)$.
          The underlying dg-cobimodule of the Boardman--Vogt construction $W_{l}M$ is quasi-isomorphic to the leveled two-sided bar construction of the primitive elements $\calB_{l}[\Sigma^{-1}\Prim(W_{l}\calP), \Sigma^{-1}\Prim(W_{l}\calQ)](\Sigma^{-1} \Prim(W_{l}M))$ (Theorem~\ref{thm:wm-bar-cobim}).
    \item The leveled Boardman--Vogt resolution $W_{l}M$ is weakly equivalent to the leveled bar-cobar construction as a ($W_{l}\calP$-$W_{l}\calQ$)-cobimodule (Theorem~\ref{thm:finale}):
          \[ \calB_{l}[\Omega_{l}(\calP), \Omega_{l}(\calQ)]\bigl( \Omega_{l}[\calP,\calQ](M) \bigr). \]
  \end{enumerate}
\end{thmintro}

\paragraph{Notations and conventions}

In this paper, we always work over a field $\K$ of characteristic zero.
All our complexes have a cohomological grading, i.e.\ differentials have degree $+1$.
Whenever we use the adjective ``cofree'', we implicitly mean ``cofree conilpotent''.
We will use the notation $\Sigma$ both for suspension of spectra or cochain complexes, and for symmetric groups (as in $\Sigma$-sequences); the meaning will always be clear from the context.
The category $\Sigma$ has as objects the ordered sets $[n] = \{ 0 < \dots < n \}$ for $n > 0$ and morphisms are bijective (not necessarily increasing) maps.
The category $\Lambda$ has the same objects as $\Sigma$; its morphisms are injective (not necessarily increasing) maps.

\section{Model category structures for (co)operads and (co)bimodules}
\label{sec:model-categ-struct}

In this section, we recall model category structures for operads and bimodules in spectra as well as model category structures for cooperads and cobimodules in commutative differential graded algebras. In both case, the ambient category is symmetric monoidal and equipped with a notion of (co)interval that allows us to build (co)fibrant resolutions.

\subsection{The model category of 1-reduced $\Lambda$-operads in spectra}
\label{SectModSpOp}

\paragraph{The category of spectra}

For concreteness, we take $S$-modules as models for spectra~\cite{ElmendorfKrizMandellMay1997}.
We denote by $\Sp$ the symmetric monoidal category of spectra with respect to the smash product $\wedge$.
The zero object, i.e.\ the constant spectrum on the point, is denoted by $*$.
The monoidal model category structure of $\Sp$ is the one from~\cite[Theorem~VII.4.6]{ElmendorfKrizMandellMay1997}.
This model category is enriched over pointed simplicial sets and it is equipped with a notion of interval introduced by Berger--Moerdijk in \cite{BM2}.
This interval is the pointed set $\Delta[1]_{+}$ obtained from $\Delta[1]$ by adding a base point.
An element in $\Delta[1]_{+}^{n}$ is an $n$-tuple $\underline{t}=(t_{1},\ldots,t_{n})$, with $t_{i}\in \{0,1\}$, or the basepoint $*$.
The associative product is given by
\begin{align*}
  - \wedge - :\Delta[1]_{+}^{p} \wedge \Delta[1]_{+}^{q} & \longrightarrow \Delta[1]_{+}^{p+q}  , \\
  \underline{t}, \underline{t}'
                                                         & \longmapsto
  \begin{cases}
    \hspace{3,5em}*                           & \text{if } \underline{t}=* \text{ or } \underline{t}'=*, \\
    (t_{1},\ldots,t_{p},t_{1}',\ldots,t_{q}') & \text{otherwise}.
  \end{cases}
\end{align*}

\paragraph{The category of 1-reduced operads}

By a symmetric sequence or $\Sigma$-sequence of spectra, we mean a covariant functor $\Sigma \to \Sp$.
Concretely, a symmetric sequence of spectra is a sequence $X = \{ X(n) \}_{n > 0}$ equipped with a right action of $\Sigma_{n}$ on $X(n)$ for all $n > 0$.
A symmetric sequence $X$ is said to be 1-reduced if the arity $1$ component $X(1)$ is the constant spectrum on the point.
We denote by $\Sigma \Seq_{>0}$ and $\Sigma \Seq_{>1}$  the category of symmetric sequences and 1-reduced symmetric sequences, respectively.
A \emph{1-reduced operad} $\calO$ is the data of a 1-reduced symmetric sequence together with operations, called operadic operations, of the form:
\begin{equation}\label{OpOp}
  \gamma:\calO(k)\wedge \calO(n_{1})\wedge \cdots \wedge \calO(n_{k})\longrightarrow \calO(n_{1}+\cdots + n_{k}),
\end{equation}
compatible with the symmetric group action and satisfying associativity and unitality axioms~\cite{Fresse2017}. The category of 1-reduced operads in spectra, denoted by $\Sigma\Operad$, is endowed with an adjunction:
\begin{equation}\label{adjSymOp}
  \mathcal{F}:\Sigma \Seq_{>1} \leftrightarrows \Sigma\Operad : \mathcal{U},
\end{equation}
where $\mathcal{F}$ is the left adjoint of the forgetful functor $\mathcal{U}$. We give an explicit description of the free functor in Section~\ref{SectBarOp}.

\begin{thm}[{\cite[Section 8.2]{Fre}}]
  The category of 1-reduced operads $\Sigma\Operad$ is equipped with a model category structure such that:
  \begin{itemize}
    \item the weak equivalences are morphisms that form a weak equivalence in every arity,
    \item the fibrations are morphisms that form a fibration map in each arity,
    \item cofibrations are characterized by the left lifting property with respect to the class of acyclic fibrations.
  \end{itemize}
  The model structure on $\Sigma \Seq_{>1}$ is defined similarly. Both model structures are called the projective model structures. They make the adjunction~\eqref{adjSymOp} into a Quillen adjunction.
\end{thm}

\begin{rmk}
  An operad $\calO \in \Sigma\Operad$ is said to be \emph{$\Sigma$-cofibrant} if its underlying $\Sigma$-sequence is cofibrant. We will use similar terminology below for 1-reduced operads, Hopf operads, bimodules, etc.
\end{rmk}

\paragraph{The category of 1-reduced $\Lambda$-operads}

Following \cite{Fre}, we call a $\Lambda$-sequence of spectra a covariant functor $\Lambda \to \Sp$ where $\Lambda$ is the category whose objects are finite sets and morphisms are injective maps.
Concretely, a $\Lambda$-sequence is a sequence of spectra $X = \{ X(n) \}_{n > 0}$ equipped with the following structure: for any injection $u : \{ 1, \dots, k \} \to \{ 1, \dots, n \}$, there is a structure map
$$
  u^{*} : X(n) \to X(k),
$$
satisfying the relation $(v\circ u)^{*}=u^{*}\circ v^{*}$ for any pair of composable maps $u : \{ 1, \dots, k \} \to \{ 1, \dots, n \}$ and $v : \{ 1, \dots, n \} \to \{ 1, \dots, m \}$. A $\Lambda$-sequence is said to be 1-reduced if the arity $1$ component $X(1)$ is the constant spectrum on the point. We denote by $\Lambda \Seq_{>0}$ and $\Lambda \Seq_{>1}$  the category of symmetric sequences and 1-reduced $\Lambda$-sequences, respectively.

A 1-reduced $\Lambda$-operad $\calO$ is the data of a 1-reduced $\Lambda$-sequence together with operadic operations \eqref{OpOp} compatible with the $\Lambda$-structure (see \cite{Fre}). The category of 1-reduced $\Lambda$-operads, denoted by $\Lambda\Operad$, is endowed with an adjunction:
\begin{equation}\label{adjLambdaOp}
  \mathcal{F}:\Lambda \Seq_{>1} \leftrightarrows \Lambda\Operad : \mathcal{U},
\end{equation}
where $\mathcal{F}$ is the left adjoint of the forgetful functor $\mathcal{U}$.

Let us remark that, by restricting to bijections, any $\Lambda$-sequence is also a symmetric sequence.
In particular, any $\Lambda$-operad has an underlying $\Sigma$-operad by restricting the action to bijections.

\begin{thm}[{\cite[Section 8.4]{Fre}}]
  The category of 1-reduced $\Lambda$-operads $\Lambda\Operad$ is equipped with a model category structure such that:
  \begin{itemize}
    \item the weak equivalences are morphisms that form a weak equivalence in each arity,
    \item the fibrations are morphisms $\phi:\calO_{1}\rightarrow \calO_{2}$ whose induced maps $\calO_{1}(n)\rightarrow \mathcal{M}(\calO_{1})(n)\wedge_{\mathcal{M}(\calO_{2})(n)} \calO_{2}(n)$ are fibration maps. Here, $\mathcal{M}(\calO)(n)$ is the matching object
          \begin{equation}\label{Matching}
            \mathcal{M}(\calO)(n) \coloneqq \lim_{\substack{f\in \Lambda([r],[n])\\ r<n}}\,\,\calO(r),
          \end{equation}
    \item cofibrations are characterized by the left lifting property with respect to the class of acyclic fibrations.
  \end{itemize}
  The model structure on $\Lambda \Seq_{> 1}$ is defined similarly.
  The model structures so defined are called the Reedy model structures and make the adjunction (\ref{adjLambdaOp}) into a Quillen adjunction.
  An operad in $\Lambda\Operad$ is Reedy cofibrant if and only if its underlying $\Sigma\Operad$ is cofibrant~\cite[Theorem~8.4.12]{Fre}.
\end{thm}

\subsection{The model category of $\Lambda$-bimodules in spectra}

\paragraph{The category of bimodules}

Let $\calP$ and $\calQ$ be 1-reduced operads in spectra. A ($\calP$-$\calQ$)-bimodule is the data of a symmetric sequence $M=\{M(n)\}_{n>0}$ together with operations $\gamma_{L}$ and $\gamma_{R}$, called respectively left and right module operations, of the form
\begin{equation}\label{OpBi}
  \begin{aligned}{rcl}
    \gamma_{L} : \calP(k)\wedge M(n_{1})\wedge \cdots \wedge M(n_{k})     & \longrightarrow M(n_{1}+\cdots+n_{k}), \\
    \gamma_{R} : M(k)\wedge \calQ(n_{1})\wedge \cdots \wedge \calQ(n_{k}) & \longrightarrow M(n_{1}+\cdots+n_{k}),
  \end{aligned}
\end{equation}
satisfying some compatibility relations with the symmetric group action as well as associative, commutative and unit axioms~\cite{AroneTurchin2014}. The category of ($\calP$-$\calQ$)-bimodules, denoted by $\Sigma\Bimod_{\calP,\calQ}$, is endowed with an adjunction
\begin{equation}\label{adjSigmBi}
  \mathcal{F}_{B}:\Sigma \Seq_{>0} \leftrightarrows \Sigma\Bimod_{\calP,\calQ}: \mathcal{U}
\end{equation}
where the free ($\calP$-$\calQ$)-bimodule functor $\mathcal{F}_{B}$ is the left adjoint of the forgetful functor $\mathcal{U}$.

\begin{expl}
  The reader can easily check that an operad $\calO$ is obviously a ($\calO$-$\calO$)-bimodule. Moreover, if $\eta:\calO\rightarrow \calO'$ is a  map of operads, then the map $\eta$ is also a ($\calO$-$\calO$)-bimodule map and the bimodule structure on $\calO'$ is defined as follows:
  $$
    \begin{aligned}
      \gamma_{R}: \calO'(k)\wedge \calO(n_{1})\wedge \cdots\wedge \calO(n_{k})  & \longrightarrow \calO'(n_{1}+\cdots + n_{k}),         \\
      (x\,;\,y_{1},\ldots,y_{k})                                                & \longmapsto \gamma(x,\eta(y_{1}),\ldots,\eta(y_{k})); \\
      \gamma_{L}: \calO(k)\wedge \calO'(n_{1})\wedge \cdots\wedge \calO'(n_{k}) & \longrightarrow \calO'(n_{1}+\cdots +n_{k}),          \\
      (x\,;\,y_{1},\ldots,y_{n})                                                & \longmapsto \gamma(\eta(x),y_{1},\ldots,y_{k}).
    \end{aligned}
  $$
\end{expl}

\begin{thm}[\cite{DucoulombierFresseTurchin2019}]  The category $\Sigma\Bimod_{\calP,\calQ}$ is equipped with a model structure such that:
  \begin{itemize}
    \item the weak equivalences are morphisms that form a weak equivalence in each arity,
    \item the fibrations are morphisms that form a fibration map in each arity,
    \item cofibrations are characterized by the left lifting property with respect to the class of acyclic fibrations.
  \end{itemize}
  The model structure on $\Sigma\Seq_{>0}$ is defined similarly.
  Both model structures are called the projective model structures and make the adjunction (\ref{adjSigmBi}) into a Quillen adjunction.
\end{thm}

\paragraph{The category of $\Lambda$-bimodules}

Let $\calP$ and $\calQ$ be 1-reduced $\Lambda$-operads in spectra. A $\Lambda$-bimodule over the operads $\calP$ and $\calQ$ is the data of a $\Lambda$-sequence together with left and right module operations \eqref{OpBi} compatible with the $\Lambda$-structure (see \cite{DucoulombierFresseTurchin2019}). The category of $\Lambda$-bimodules over $\calP$ and $\calQ$, denoted by $\Lambda\Bimod_{\calP,\calQ}$, is endowed with an adjunction
\begin{equation}\label{adjLambdaBi}
  \mathcal{F}_{B}:\Lambda \Seq_{>0} \leftrightarrows \Lambda\Bimod_{\calP,\calQ}: \mathcal{U}
\end{equation}
where $\mathcal{F}_{B}$ is the left adjoint of the forgetful functor.
Note that, just like in the case of operads, restricting the action of $\Lambda$ to bijections defines an underlying bimodule for any $\Lambda$-bimodule.

\begin{thm}[\cite{DucoulombierFresseTurchin2019}]
  The category $\Lambda\Bimod_{\calP,\calQ}$ is equipped with a model structure such that:
  \begin{itemize}
    \item the weak equivalences are morphisms that form a weak equivalence in each arity,
    \item the fibrations are morphisms $\phi:M_{1}\rightarrow M_{2}$ that induced maps $M_{1}(n)\rightarrow \mathcal{M}(M_{1})(n)\times_{\mathcal{M}(M_{2})(n)} M_{2}(n)$ are fibration maps, where $\mathcal{M}(M)$ is the matching object \eqref{Matching}.
    \item cofibrations are characterized by the left lifting property with respect to the class of acyclic fibrations.
  \end{itemize}
  The model structure on $\Lambda\Seq_{>0}$ is defined similarly. In both cases, the model structures are called the Reedy model structures and make the adjunction (\ref{adjLambdaBi}) into a Quillen adjunction.
  A $\Lambda$-bimodule is Reedy cofibrant if and only if its underlying $\Sigma$-bimodule is cofibrant in the projective model category.
\end{thm}

\subsection{The model category of Hopf $\Lambda$-cooperads}\label{SectModHopf}

\paragraph{The category of (commutative) differential graded algebras}

The main categories considered in this section are chain complexes and commutative differential graded algebras denoted by $\Ch$ and $\CDGA$, respectively. Both are symmetric monoidal categories equipped with a notion of cointerval (dual version of the notion of interval).  This cointerval is given by polynomial forms on the interval, $\Omega^{*}(\Delta^{1}) \coloneqq \K[t,dt]$, with the de Rham differential. We have natural algebra maps $d_{0},d_{1}:\K[t,dt]\rightarrow \K$ by evaluation at the endpoints $t=0$ and $t=1$. Furthermore, one has a coassociative coproduct
$$
  m^{\ast}:\K[t,dt]\rightarrow \K[t,dt]\otimes \K[t,dt],
$$
given by the pullback of the multiplication map
\begin{align*}
  m:[0,1] \times [0, 1] & \longrightarrow [0, 1],   \\
  (s,t)                 & \longmapsto 1-(1-s)(1-t).
\end{align*}
Concretely, $m^{*}(t) \coloneqq t \otimes 1 + 1 \otimes t - t \otimes t$ and $m^{*}(dt) \coloneqq dt \otimes 1 + 1 \otimes dt - dt \otimes t - t \otimes dt$.
The evaluation map at the endpoint $t=0$ defines a counit for the coproduct $m^{\ast}$ so obtained.
On the other hand, evaluation at $t=1$ is a coabsorbing element, i.e. the following diagram commutes:
\[
  \begin{tikzcd}
    \K[t,dt] \otimes \K[t,dt] \ar[ddr, "\id \otimes \ev_{t=1}" swap] & \K[t,dt] \ar[r, "m^{\ast}"] \ar[l, "m^{\ast}" swap] \ar[d, "\ev_{t=1}"] & \K[t,dt] \otimes \K[t,dt] \ar[ddl, "\ev_{t=1} \otimes \id"] \\
    & \K \ar[d, hook] & \\
    & \K[t,dt] &
  \end{tikzcd}
\]

\paragraph{The category of 1-reduced cooperads}

By a symmetric cosequence or $\Sigma$-cosequence of chain complexes, we mean a family of chain complexes $X = \{ X(n) \}_{n > 0}$ equipped with a right coaction of $\Sigma_{n}$ on $X(n)$ for all $n > 0$.
A symmetric cosequence $X$ is said to be 1-reduced if the arity $1$ component $X(1)$ is the one dimension vector space $\K$ concentrated in degree $0$. We denote by $\dg\Sigma \Seq_{>0}^{c}$ and $\dg\Sigma \Seq_{>1}^{c}$  the category of symmetric cosequences and 1-reduced symmetric cosequences, respectively. A 1-reduced cooperad $\calC$ is the data of a 1-reduced symmetric cosequence together with operations, called cooperadic operations, of the form
\begin{equation}\label{OpCoop}
  \gamma^{c}:\calC(n_{1}+\cdots + n_{k})\longrightarrow \calC(k)\otimes \calC(n_{1})\otimes \cdots \otimes \calC(n_{k}) ,
\end{equation}
compatible with the symmetric group coaction and satisfying coassociativity and counitality axioms~\cite{Fresse2017}. The category of 1-reduced cooperads in chain complexes $\Sigma\Cooperad$ is endowed with an adjunction
\begin{equation}\label{adjSym}
  \mathcal{U}:\Sigma\Cooperad  \leftrightarrows  \dg\Sigma \Seq_{>1}^{c}:\mathcal{F}^{c}
\end{equation}
where $\mathcal{F}^{c}$ is the right adjoint of the forgetful functor.

\begin{thm}[{\cite[Section 9.2]{Fre}}]\label{thm:model-struct-cooperads}
  The category of 1-reduced cooperads $\Sigma\Cooperad$ is equipped with a model category structure such that:
  \begin{itemize}
    \item the weak equivalences are morphisms that form a quasi-isomorphism in each arity,
    \item the cofibrations are morphisms that form a surjective map in each arity,
    \item the fibrations are characterized by the right lifting property with respect to the class of acyclic cofibrations.
  \end{itemize}
  The model structure on $ \dg\Sigma\Seq_{>1}^{c}$ is defined similarly.
  Both model structures are called the projective model structures and make the adjunction~\eqref{adjSym} into a Quillen adjunction.
  Any quasi-cofree cooperad $(\calF^{c}(X), \partial)$ with a differential induced by a linear map $\calF^{c}(X) \to X$ vanishing on $X$ is fibrant~\cite[Proposition~9.2.9]{Fre}.
\end{thm}

\paragraph{The category of 1-reduced $\Lambda$-cooperads}

Dually to the previous sections, by a $\Lambda$-cosequence in chain complexes we understand a family of spectra $X = \{ X(n) \}_{n > 0}$ equipped with the following structure: for any injection $u : \{ 1, \dots, k \} \to \{ 1, \dots, n \}$, there is a structure map
$$
  u_{*} : X(k) \to X(n),
$$
satisfying the relation $(v\circ u)_{*}=v_{*}\circ u_{*}$ for any pair of composable maps $u : \{ 1, \dots, k \} \to \{ 1, \dots, n \}$ and $v : \{ 1, \dots, n \} \to \{ 1, \dots, m \}$.
Furthermore, there is a structure map
$
  \epsilon:\K\rightarrow X(n),
$
for all $n > 0$. A $\Lambda$-cosequence is said to be 1-reduced if one has $X(1)=\K$. We denote by $\dg\Lambda \Seq_{>0}^{c}$ and $\dg\Lambda \Seq_{>1}^{c}$  the category of symmetric cosequences and 1-reduced $\Lambda$-cosequences, respectively.  Let us remark that, by restriction to bijection, any  $\Lambda$-cosequence is also a symmetric cosequence.

A 1-reduced $\Lambda$-cooperad $\calC$ is the data of a 1-reduced $\Lambda$-cosequence together with cooperadic operations \eqref{OpCoop} compatible with the $\Lambda$-costructure (see \cite{Fre}). The category of 1-reduced $\Lambda$-cooperads, denoted by $\Lambda\Cooperad$, is endowed with an adjunction:
\begin{equation}\label{adjLambda}
  \mathcal{U}:\Lambda\Cooperad  \leftrightarrows  \dg\Lambda \Seq_{>1}^{c}:\mathcal{F}^{c}
\end{equation}
where $\mathcal{F}^{c}$ is the right adjoint of the forgetful functor.

\begin{thm}[{\cite[Section 11.3]{Fre}}] The category of 1-reduced $\Lambda$-cooperads $\Lambda\Cooperad$ is equipped with a model category structure such that:
  \begin{itemize}
    \item the weak equivalences are morphisms that form a quasi-isomorphism in each arity,
    \item the cofibrations are morphisms $\phi:\calO_{1}\rightarrow \calO_{2}$ that induced maps that form a fibration in the undercategory $Com^{c}\downarrow \Sigma \Operad$ where $Com^{c}$ is the terminal object,
    \item the fibrations are characterized by the right lifting property with respect to the class of acyclic cofibrations.
  \end{itemize}
  The model structure on $ \dg\Lambda\Seq_{>1}^{c}$ is defined similarly.
  Both model structures are called the Reedy model structures and make the adjunction~\eqref{adjLambda} into a Quillen adjunction.
  A $\Lambda$-cooperad is fibrant if and only if its underlying $\Sigma$-cooperad is fibrant in the projective model category (compare with~\cite[Proposition~4.3]{FresseTurchinWillwacher2017}).
\end{thm}

\paragraph{The category of Hopf 1-reduced $\Lambda$-cooperads}

Let $\dg \Hopf\Lambda \Seq_{>0}^{c}$ (resp.\ $\dg \Hopf\Lambda \Seq_{>1}^{c}$) be the category of $\Lambda$-cosequences (resp.\ 1-reduced $\Lambda$-cosequences) in commutative differential graded algebras. A Hopf 1-reduced $\Lambda$-cooperad is a cosequence in $\dg \Hopf\Lambda \Seq_{>1}^{c}$ equipped with cooperadic operations \eqref{OpCoop} compatible with the $\Lambda$-costructure and the Hopf structure. The category of Hopf 1-reduced $\Lambda$-cooperads, denoted by $\Hopf\Lambda\Cooperad$, is endowed with an adjunction:
\begin{equation}\label{adjHopf}
  \mathcal{U}:\Hopf\Lambda\Cooperad  \leftrightarrows  \dg\Hopf\Lambda \Seq_{>1}^{c}:\mathcal{F}^{c}
\end{equation}
where $\mathcal{F}^{c}$ is the right adjoint of the forgetful functor.

\begin{thm}[{\cite[Section 11.4]{Fre}}]
  The category of 1-reduced  cooperads $\Hopf\Lambda\Cooperad$ is equipped with a model category structure such that:
  \begin{itemize}
    \item the weak equivalences are morphisms that form a quasi-isomorphism in each arity,
    \item the fibrations are morphisms that form a fibration in $\Lambda\Cooperad$ when we forget about the Hopf structure,
    \item the cofibrations are characterized by the left lifting property with respect to the class of acyclic cofibrations.
  \end{itemize}
  The model structure on $ \dg\Hopf\Lambda\Seq_{>1}^{c}$ is defined similarly.  Both model structures are called the Reedy--Hopf model structures and make the adjunction\eqref{adjHopf} into a Quillen adjunction.
  A Hopf $\Lambda$-cooperad is fibrant if and only if its underlying Hopf $\Sigma$-cooperad is fibrant in the projective model category~\cite[Proposition~4.3]{FresseTurchinWillwacher2017}.
\end{thm}

\subsection{The model category of Hopf $\Lambda$-cobimodules}

\paragraph{The category of cobimodules}

Let $\calP$, $\calQ$ be 1-reduced cooperads in chain complexes.
A ($\calP$-$\calQ$)-cobimodule is the data of a symmetric cosequence $M=\{M(n)\}_{n>0}$ together with operations $\gamma_{L}^{c}$ and $\gamma_{R}^{c}$, called respectively left and right comodule operations, of the form
\begin{equation}\label{OpCobi}
  \begin{aligned}
    \gamma_{L}^{c}: M(n_{1}+\cdots+n_{k}) & \longrightarrow\calP(k)\otimes M(n_{1})\otimes \cdots \otimes M(n_{k}) ;     \\
    \gamma_{R}^{c}:M(n_{1}+\cdots+n_{k})  & \longrightarrow M(k)\otimes \calQ(n_{1})\otimes \cdots \otimes \calQ(n_{k}),
  \end{aligned}
\end{equation}
satisfying some compatibility relations with the symmetric group coaction as well as coassociativity, cocommutativity and counitality axioms~\cite{FresseWillwacher2019}.
The category of ($\calP$-$\calQ$)-cobimodules, denoted by $\Sigma\Cobimod_{\calP,\calQ}$, is endowed with an adjunction
\begin{equation}\label{adjSigmB}
  \mathcal{U}:\Sigma\Cobimod_{\calP,\calQ} \leftrightarrows  \dg\Sigma \Seq_{>0}^{c}:\mathcal{F}^{c}_{B}
\end{equation}
where $\mathcal{F}^{c}_{B}$ is the right adjoint of the forgetful functor.

\begin{thm}[{\cite{FresseWillwacher2019}}]
  The category $\Sigma\Cobimod_{\calP,\calQ}$ is equipped with a model structure such that:
  \begin{itemize}
    \item the weak equivalences are morphisms that form a quasi-isomorphism in each arity,
    \item the cofibrations are morphisms that form an injective map in each arity,
    \item the fibrations are characterized by the right lifting property with respect to the class of acyclic cofibrations.
  \end{itemize}
  The model category structure on $\dg\Sigma \Seq_{>0}^{c}$ is defined similarly.
  Both model structures are called the projective model structures and make the adjunction \eqref{adjSigmB} into a Quillen adjunction.
  A quasi-cofree Hopf $\Sigma$-cobimodule is fibrant.
\end{thm}

\paragraph{The category of $\Lambda$-cobimodules}

Let $\calP$ and $\calQ$ be 1-reduced $\Lambda$-operads in chain complexes. A $\Lambda$-bimodule over the operads $\calP$ and $\calQ$ is the data of a $\Lambda$-cosequence together with left and right comodule operations \eqref{OpCobi} compatible with the $\Lambda$-costructure. The category of $\Lambda$-cobimodules over $\calP$ and $\calQ$, denoted by $\Lambda\Cobimod_{\calP,\calQ}$, is endowed with an adjunction
\begin{equation}\label{adjLambdaB}
  \mathcal{U}:\Lambda\Cobimod_{\calP,\calQ} \leftrightarrows  \dg\Lambda \Seq_{>0}^{c}:\mathcal{F}^{c}_{B}
\end{equation}
where $\mathcal{F}^{c}_{B}$ is the right adjoint of the forgetful functor.

\begin{thm}[{\cite{FresseWillwacher2019}}]
  The category $\Lambda\Cobimod_{\calP,\calQ}$ is equipped with a model structure such that:
  \begin{itemize}
    \item the weak equivalences are morphisms that form a quasi-isomorphism in each arity,
    \item the cofibrations are morphisms that form a cofibration in the undercategory $Com^{c}\downarrow \Sigma\Bimod_{\calP,\calQ}$.
    \item the fibrations are characterized by the right lifting property with respect to the class of acyclic cofibrations.
  \end{itemize}
  The model category structure on $\dg\Lambda \Seq_{>0}^{c}$ is defined similarly.
  Both model structures are called the Reedy model structures and makes the adjunction \eqref{adjLambdaB} into a Quillen adjunction.
  An object is fibrant if and only if its underlying $\Sigma$-cobimodule is fibrant.
\end{thm}

\paragraph{The model category of Hopf $\Lambda$-cobimodules}

Let $\calP$ and $\calQ$ be Hopf 1-reduced $\Lambda$-operads. A Hopf $\Lambda$-bimodule over the operads $\calP$ and $\calQ$ is the data of a $\Lambda$-sequence in $\dg\Hopf\Lambda \Seq_{>0}^{c}$ together with left and right comodule operations \eqref{OpCobi} compatible with the $\Lambda$-costructure and the Hopf structure. The category of Hopf $\Lambda$-cobimodules over $\calP$ and $\calQ$, denoted by $\Hopf\Lambda\Cobimod_{\calP,\calQ}$, is endowed with an adjunction
\begin{equation}\label{adjHopfB}
  \mathcal{U}:\Hopf\Lambda\Cobimod_{\calP,\calQ}  \leftrightarrows  \dg\Hopf\Lambda \Seq_{>0}^{c}:\mathcal{F}^{c}_{B}
\end{equation}
where $\mathcal{F}^{c}_{B}$ is the right adjoint of the forgetful functor.

\begin{thm}[{\cite{FresseWillwacher2019}}]
  The category $\Hopf\Lambda\Cobimod_{\calP,\calQ}$ is equipped with a model structure such that:
  \begin{itemize}
    \item the weak equivalences are morphisms that form a quasi-isomorphism in each arity,
    \item the fibrations are morphisms  that form  fibrations in $\Lambda\Cobimod_{\calP,\calQ}$ when we forget about the Hopf structure,
    \item the cofibrations are characterized by the left lifting property with respect to the class of acyclic fibrations.
  \end{itemize}
  The model structure on $\dg\Hopf\Lambda \Seq_{>0}^{c}$ is defined similarly.
  Both model structures are called Reedy--Hopf model category structures and make \eqref{adjHopfB} into a Quillen adjunction.
  An object is fibrant if and only if its underlying non-Hopf object is fibrant.
\end{thm}

\section{Inventory of categories of trees}
\label{sec:invent-categ-trees}

In this section, we define the various categories of trees that we will need in the rest in the paper.
As a convention, we will use the letter $\Tree$ for general planar trees and $\lTree$ for leveled trees, and we will prefix categories with $s$ to indicate the choice of a section.

We first introduce the category of planar $n$-trees $\Tree[n]$ as well as the category of leveled $n$-trees $\lTree[n]$.
We show that the family $\Tree=\{\Tree[n], n\geq 2\}$ gives rise to an operad while $\lTree=\{\lTree[n], n\geq 2\}$ is endowed with a kind of operadic structure satisfying the operadic axioms up to some permutation conditions.
We then introduce a category $\sTree$ of planar trees with sections and we show that it defines an operadic bimodule over $\Tree$.
In the same way, we define the family $\slTree=\{\slTree[n], n\geq 2\}$ equipped with operations which look like a bimodule structure over $\lTree$.
These operations will play an important role in the next sections in order to define (co)operadic and (co)bimodule structures on Boardman--Vogt resolutions and alternative versions of bar/cobar constructions.

\begin{rmk}\label{rmk:planar}
  In what follows, we work exclusively with planar trees, i.e.\ the incoming edges of a vertex are ordered.
  However, note that the (co)operads and (co)bimodules are symmetric.
  We will thus need to consider the action of the symmetric group on incoming edges in order for our definitions to make sense in subsequent sections.
\end{rmk}

\subsection{The categories of planar trees and leveled trees}

In the following, we introduce the two categories of planar trees $\Tree_{\iso}[n]$ and $\Tree[n]$ having the same set of objects and which differ in their morphisms. Contrary to $\Tree_{\iso}[n]$, the category $\Tree[n]$ includes morphisms contracting consecutive vertices. The latter one is used to build resolutions while $\Tree_{\iso}[n]$ is often used to construct free objects. Thereafter, we define the categories of leveled trees $\lTree_{\iso}[n]$ and $\lTree[n]$ together with some kind of operadic compositions. The last paragraph is devoted to the comparison between the categories of trees and their leveled versions.

\subsubsection{The category of planar trees}
\label{sec:categ-plan-trees}

Let us first give an informal definition of planar trees. A planar tree $T$ is a finite planar acyclic graph with one output edge at the bottom and input edges, called leaves, at the top. The output and input edges are considered to be half-open, i.e.\ connected only to one vertex in the body of the graph. The vertex connected to the output edge, called the root of $T$, is denoted by $r$. Each edge in the tree is oriented from top to bottom. For an integer $n\geq 2$, a planar $n$-tree is a planar tree with leaves labeled by the set $[n]=\{1,\ldots,n\}$. Formally, we will define planar $n$-trees as follows:

\begin{defi}\label{def:planar tree}
  For any integer $n\geq 2$, a \emph{planar $n$-tree} $T$ is the data of a set $V(T)$, a total order on $V(T) \cup [n]$ (that does not necessarily restrict to the natural order on $[n]$), a non-decreasing map $t:V(T)\cup [n]\rightarrow V(T)$, called the \textit{target map}, and a marked element $r \in V(T)$, satisfying the two conditions:
  \begin{itemize}
    \item $t(r) = r$;
    \item $\forall v \in V(T) \cup [n], \exists k \ge 0 \text{ s.t. } t^{k}(v) = r$.
  \end{itemize}
\end{defi}

In the definition above, the sets $[n]$ and $V(T)$ represent the leaves and the vertices of the tree $T$, respectively.
The total order encodes the planarity (if $v \le v'$ then $v$ is on the left or in a lower level than the one of $v'$) and moreover gives an indexing of the leaves by $[n]$.
The element $r$ is the root of the tree, and given $v \in V(T) \cup [n]$, the vertex $t(v)$ is the target of the only outgoing edge leaving $v$ (except for the root which simply satisfies $t(r) = r$).
An example of planar $n$-tree is represented in Figure~\ref{fig:example-planar-tree}.
Furthermore, each planar tree $T$ is equipped with a level map $\lev : V(T) \to \mathbb{N}$ satisfying $\lev(r) =0$ and $\lev(t(v)) = \lev(v) - 1$ for all $v \in V(T) \setminus \{ r\}$.
The height of the tree $T$, denoted by $h(T)$ is the highest level of the vertices of the tree: $h(T)=max\{\lev(v), v\in V(T)\}$.

\begin{notat}
  Let $T$ be a planar $n$-tree for some integer $n\geq 2$.
  \begin{itemize}
    \item The set of \emph{edges} of $T$ is $E(T) \coloneqq \{ (v, t(v)) \mid v \in [n]\cup V(T)\setminus \{r\} \}$.
          The set of \emph{inner edges} $E^{in}(T) \coloneqq \{ (v, v') \in E(T) \mid v \in V(T)\setminus \{r\} \}$ is formed by the edges connecting two vertices.
          Each edge or vertex is joined to the root by a unique path composed of edges.
    \item The set of \emph{incoming edges} of a vertex $v \in V(T)$ is given by $in(v) \coloneqq \{ (w,w')\in E(T) \mid w' = v \}$.
          It inherits a total order from the total order on $V(T) \cup S$ (pictorially, from left to right).
    \item The number of incoming edges at a vertex $v \in V(T)$, denoted by $|v| \coloneqq \# in(v)$, is called the \emph{arity} of $v$.
          The total number of adjacent edges at $v$ will be called the \emph{valence} of $v$ and is equal to the arity plus one.
  \end{itemize}
\end{notat}

\begin{figure}[htbp]
  \centering
  \includegraphics[scale=0.55]{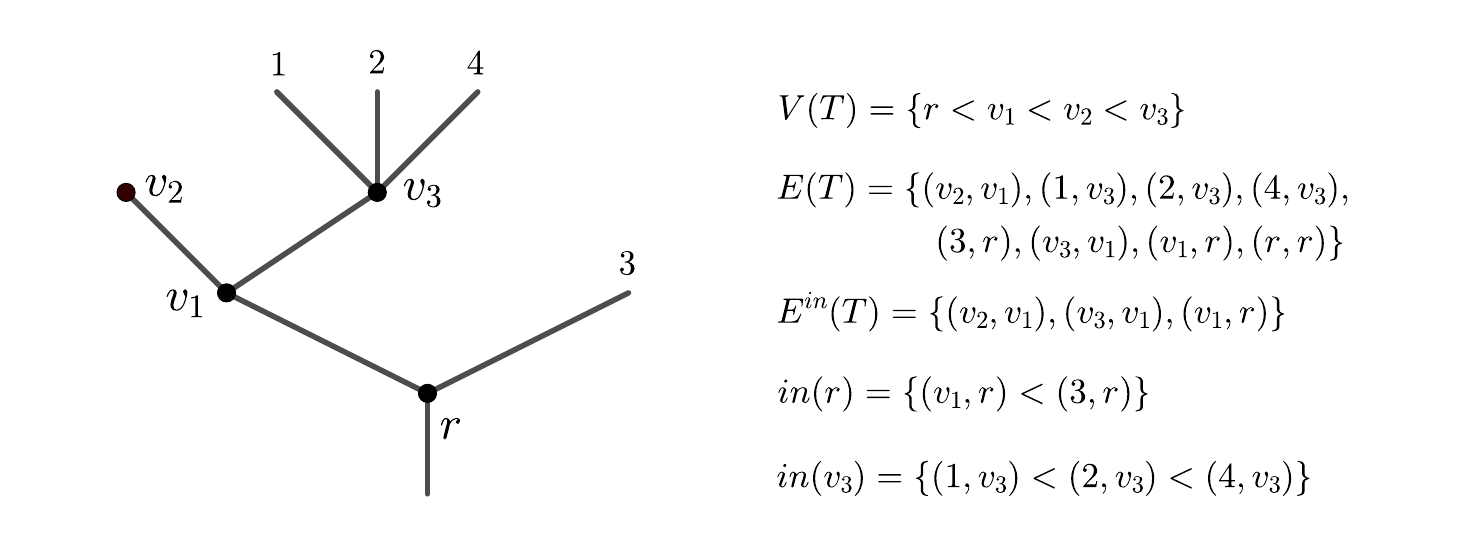}
  \caption{Example of a planar $4$-tree.}
  \label{fig:example-planar-tree}
\end{figure}

\begin{defi}[The categories $\Tree$ and $\Tree_{\iso}$]\label{def:trees}
  An \emph{isomorphism} of planar $n$-trees is a bijection between vertices preserving the root as well as the total order, and commuting with the target map. We also consider morphisms contracting consecutive vertices. There is a \textit{contracting morphism} from a planar tree $T$ to another tree $T'$ if there is a subset of inner edges $E'\subset E^{in}(T)$ such that the tree $T'$ is obtained from $T$ by removing the edges corresponding to $E'$ and by identifying the consecutive vertices $v$ and $v'$ for any $(v,v')\in E'$.

  The category $\Tree_{\iso}[n]$ consists of planar $n$-trees and isomorphisms between them while $\Tree[n]$ is the category with the same set of objects and whose  morphisms are composed of isomorphisms and contracting morphisms. We also consider the subcategories $\Tree^{\geq 2}[n]$ and $\Tree^{\geq 2}_{\iso}[n]$, of planar $n$-trees whose vertices have at least $2$ antecedent (i.e. $|t^{-1}(v)|\geq 2$ for any $v\in V(T)$).
\end{defi}

\begin{rmk}\label{rem:tree isomorphisms}
  Note that planar trees do not admit nontrivial automorphisms, but some isomorphisms change the order of leaves, as we can see in the following example:
  \begin{center}
    \includegraphics[scale=0.4]{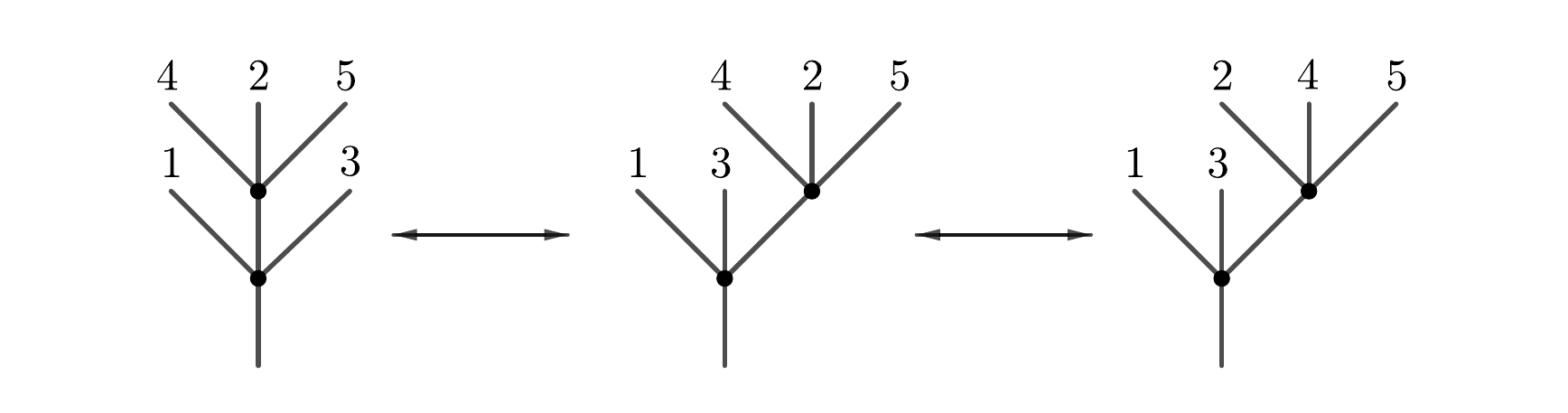}
  \end{center}
\end{rmk}

\subsubsection{The category of leveled trees}\label{SectTree}

Like in the previous section, we first give an informal definition of the category of leveled trees $\lTree[n]$. A leveled $n$-tree is a planar $n$-tree $T$ without univalent vertices for which the level map $\lev : V(T) \to \mathbb{N}$ satisfies the additional condition: $\lev(v) = h(T)$ if the vertex $v$ is of the form $v = t(s)$ for some leaf $s \in [n]$. Furthermore, we assume that each level has at least one vertex of valence $\geq 2$. From now, we give a formal definition of the categories $\lTree[n]$ and $\lTree_{\iso}[n]$ having the same sets of objects.

\begin{defi}
  A \emph{leveled $n$-tree}, with $n > 0$, is the data of a permutation $\sigma\in \Sigma_{n}$, and a sequence of ordered sets together with increasing surjections
  \begin{equation}\label{nondecsurj}
    [n] \xtwoheadrightarrow{t_{h(T)}} V_{h(T)}(T) \xtwoheadrightarrow{t_{h(T)-1}} \cdots \xtwoheadrightarrow{t_{0}} V_{0}(T) = \{r\}
  \end{equation}
  such that $n> |V_{h(T)}(T)|$ and $|V_{i+1}(T)|>|V_{i}(T)|$. If there is no ambiguity about the permutation $\sigma$ and the sequence of non-decreasing surjections, then we will just denote by $T$ a leveled $n$-tree.
  The integer $h(T)$, also denoted by $h$ if there is no ambiguity, is the \textit{height} of the tree $T$. We also make the following definitions:
  \begin{itemize}
    \item The \emph{vertices} of $T$ are given by the set $V(T) \coloneqq \bigsqcup_{i=0}^{h} V_{i}(T)$.

    \item The set of \emph{inner edges} $E^{in}(T)$ and the set of \emph{edges} $E(T)$ are given by
          $$
            E^{in}(T) \coloneqq \big\{\,\, (v, t_{i}(v)) \mid v \in V_{i+1}(T)\,\,\big\}
            \hspace{15pt}\text{and} \hspace{15pt}
            E(T) \coloneqq E^{in}(T) \cup \big\{ (i, t_{h}(i)) \mid i \in [n]\big\}\cup \{(r,r)\}.
          $$

    \item For a vertex $v \in V_{i}(T)$, its \emph{incoming edges} are $in(v) \coloneqq \{ (w,v) \mid t_{i}(w) = v \}$. The set $in(v)$ inherits a total order from $V_{i+1}(T)$. The \emph{arity} $|v|$ is the cardinality of $in(v)$. Note that $|v| \ge 1$ for all $v$, because we require all the $t_{i}$ to be surjections.
  \end{itemize}

  \begin{figure}[htbp]
    \centering
    \includegraphics[scale=0.5]{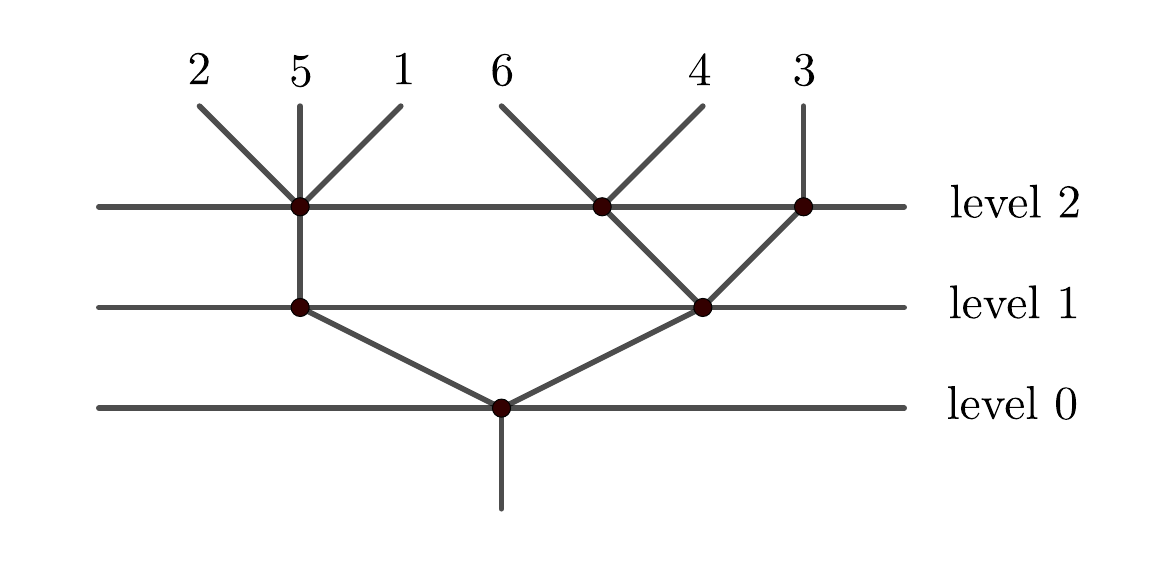}\vspace{-20pt}
    \caption{A leveled $6$-tree.}\label{leveledtree}
  \end{figure}
\end{defi}

\begin{defi}[{The categories of leveled trees $\lTree[n]$ and $\lTree_{\iso}[n]$}]\label{def:ltree}
  In the following we introduce three kinds of elementary morphisms between leveled trees. The categories $\lTree[n]$ and $\lTree_{\iso}[n]$ of leveled $n$-trees have the same set of objects.
  Morphisms in $\lTree[n]$ consist of
  \begin{enumerate*}[label={(\arabic*)}]
    \item isomorphisms,
    \item contractions of consecutive levels, and
    \item permutations of permutable levels.
  \end{enumerate*}
  Morphisms in $\lTree_{\iso}[n]$ only consist of
  \begin{enumerate*}[label={(\arabic*)}]
    \item isomorphisms,
    \item contractions of consecutive \emph{permutable} levels, and
    \item permutations of permutable levels.
  \end{enumerate*}

  \begin{enumerate}
    \item The first kind of elementary morphisms are isomorphisms of planar trees preserving the levels.
    \item The second ones consist in contracting consecutive levels.
          Let $T \in \lTree[n]$ and $N \subset \{ 1, \dots, h(T) \}$ be an interval. We define the tree $T/N$ by forgetting the level $V_{i}(T)$, for $i \in N$, and by composing the corresponding decorations $t_{i}$. We denote the contraction morphism by $\delta_{N}:T \rightarrow T/N$.
          \begin{figure}[htbp]
            \centering
            \hspace{20pt}\includegraphics[scale=0.28]{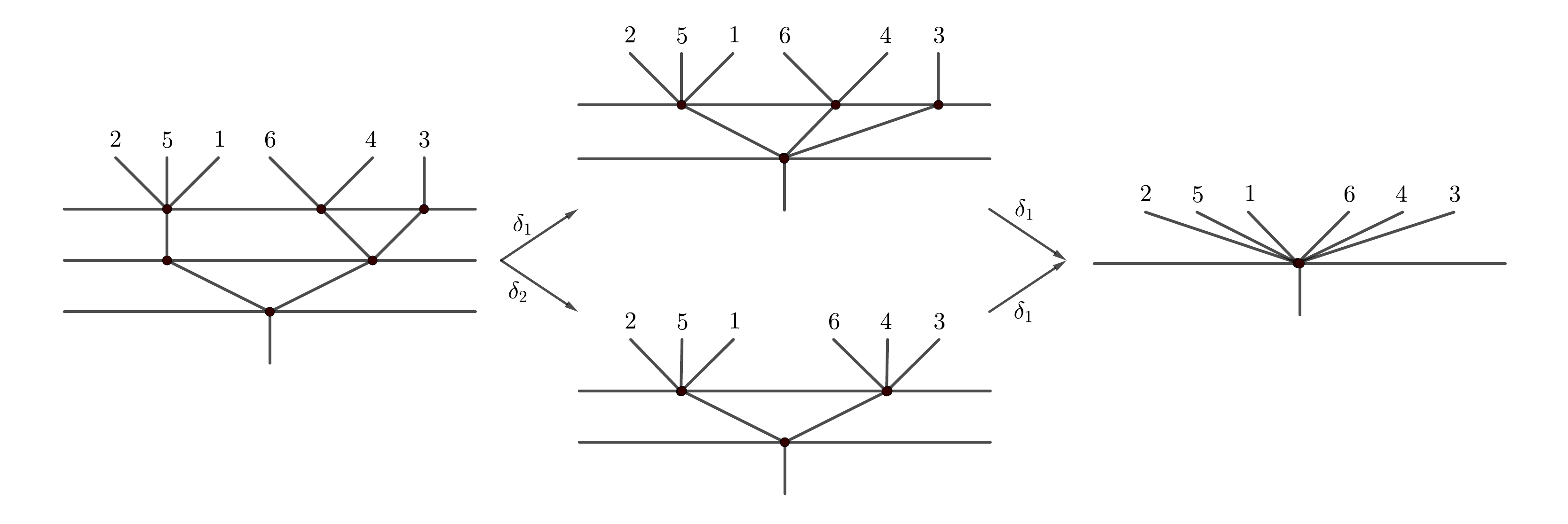}\vspace{-17pt}
            \caption{Contraction morphisms in the category $\lTree[6]$.}
            \label{ContTree}
          \end{figure}
    \item The third ones consist in permuting two consecutive levels. Given a tree $T$, we say that two consecutive levels $i, i+1$ are \emph{permutable} if all the edges between the two have either a bivalent source or a bivalent target. In that case, we denote by $\sigma_{i}(T)$ the tree obtained as follows.  We ``move up'' (see below) all its vertices on level $i$ which have valence $\ge 3$, and we ``move down'' all its vertices on level $i+1$ which have valence $\ge 3$.  See Figure~\ref{PermTree} for an illustration.
          \begin{itemize}
            \item Suppose that $v$ is on level $i$ and that all its children are bivalent. Then we ``move up'' $v$ to the level $i+1$ by collapsing all its children to a single child $v'$. More precisely, suppose we are given a tree $T$ such as in Equation~\eqref{nondecsurj} (with the same notation). Let $v \in V_{i}(T)$ be a vertex such that all its children $c \in t_{i}^{-1}(v)$ satisfy $|c| = 1$. Then we define $\sigma_{v}(T)$ to be the following tree:
                  \[ [n] \xtwoheadrightarrow{t_{h(T)}}  \cdots \xtwoheadrightarrow{t_{i+2}} V_{i+2}(T)
                    \xtwoheadrightarrow{\tilde{t}_{i+1}} \tilde{V}_{i+1}(T) \xtwoheadrightarrow{\tilde{t}_{i}} V_{i}(T) \xtwoheadrightarrow{t_{i-1}} V_{i-1}(T) \xtwoheadrightarrow{t_{i-2}} \cdots \xtwoheadrightarrow{t_{0}} V_{0}(T) \]
                  where $\tilde{V}_{i+1}(T) = V_{i+1}(T) / (t_{i}^{-1}(v))$, i.e.\ we identify all the children of $v$ to a single vertex. We define $\tilde{t}_{i+1}$ and $\tilde{t}_{i}$ to be the induced maps on the quotient.

            \item The reverse operation is moving down a vertex. If $v$ is a vertex of level $i+1$ is the only child of its parent, then we can move $v$ down to level $i$. We replace $v$ by several new vertices, one for each incoming edge at $v$. All of these new vertices have the same parent as $v$.
          \end{itemize}
  \end{enumerate}
\end{defi}

  \begin{figure}[!h]
    \hspace{-25pt}\includegraphics[scale=0.33]{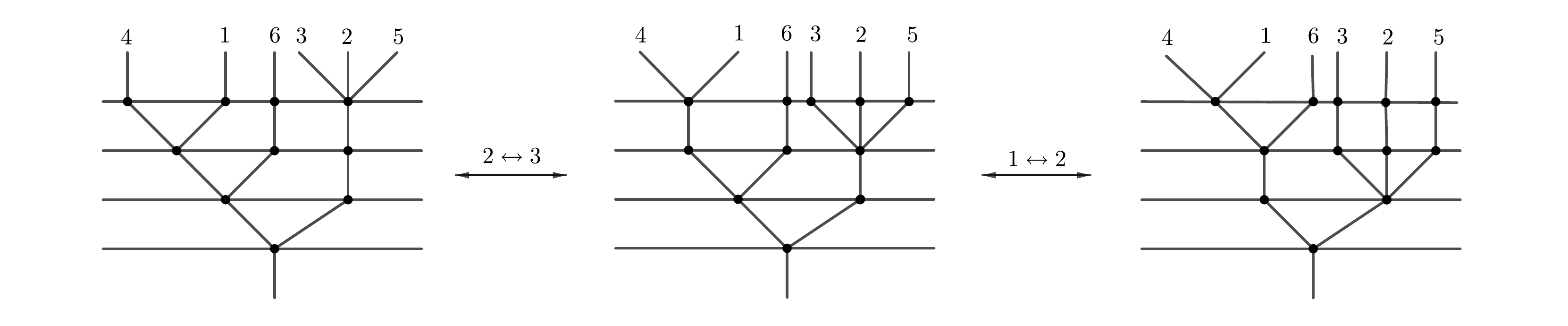}\vspace{-15pt}
    \caption{Permutations in the category $\lTree[6]$.}
    \label{PermTree}
  \end{figure}

\subsubsection{On ``operadic'' structures for (possibly leveled) trees}\label{sec operadic quote}

Let us introduce the operations needed to define the (co)operadic structures on Boardman--Vogt resolutions of (co)operads.
They will also be used to define the (co)operadic structure on an alternative version of the (co)bar construction.
First, we recall the well-known operadic structure on the set of planar trees $\gamma : \Tree[k] \times \Tree[n_{1}] \times \cdots \times \Tree[n_{k}]\longrightarrow \Tree[n_{1}+\cdots+n_{k}]$.
These operations are defined as follows: for any family of planar trees $(T_{0}; T_1, \ldots, T_{k})\in \Tree[k] \times \Tree[n_{1}] \times \cdots \times \Tree[n_{k}]$, the operadic composition is obtained by grafting each tree $T_{i}$, with $1\leq i \leq k$, into the $i$-th leaf of the tree $T_{0}$.
This structure cannot be extended
 to leveled trees.
Nevertheless, we define structure maps on $\lTree=\{\lTree[n],\, n\geq 0\}$ similar to operadic composition maps:
\begin{equation*}
  \gamma : \lTree[k] \times \lTree[n_{1}] \times \cdots \times \lTree[n_{k}]\longrightarrow \lTree[n_{1}+\cdots+n_{k}]
\end{equation*}
in Equation~\eqref{formulatreecomp}.
Fix leveled trees
\begin{align*}
  T_{0}
   & = [k] \xtwoheadrightarrow{t^{0}_{h_{0}}} V_{h_{0}}(T_{0}) \xtwoheadrightarrow{t^{0}_{h_{0}-1}} \cdots \xtwoheadrightarrow{t^{0}_{0}} \{ r \},
   &
  T_{i}
   & = [n_{i}] \xtwoheadrightarrow{t^{i}_{h_{i}}} V_{h_{i}}(T_{i}) \xtwoheadrightarrow{t^{i}_{h_{i}-1}} \cdots \xtwoheadrightarrow{t^{i}_{0}} \{ r \}.
\end{align*}
We will illustrate our constructions with the example of leveled trees from Figure~\ref{famlevtree}.

\begin{figure}[htbp]
  \vspace{-10pt}\hspace{-15pt}\includegraphics[scale=0.4]{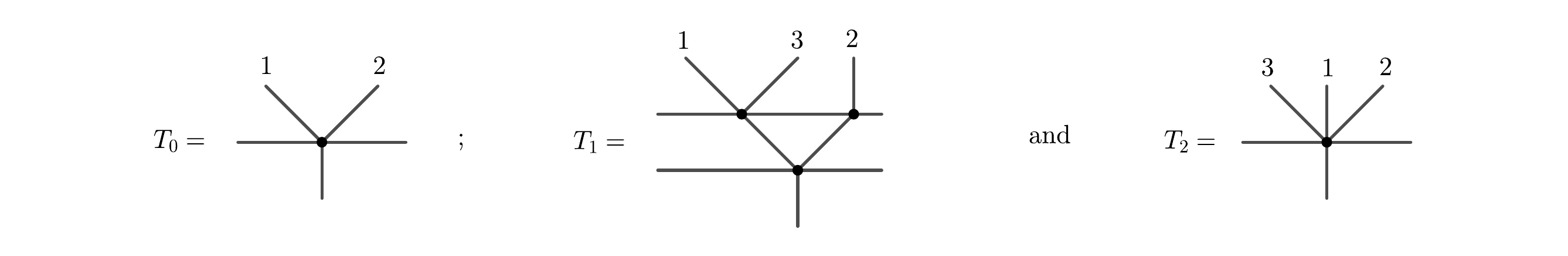}\vspace{-20pt}
  \caption{Leveled trees.}\label{famlevtree}
\end{figure}

To define $\gamma$, we first introduce partial compositions of leveled trees:
\[ \circ_{i} : \lTree[k] \times \lTree[n_{i}] \to \lTree[n_{i}+k-1]. \]
The leveled tree $T_{0}\circ_{i}T_{i}$ is defined by grafting the leveled tree $T_{i}$ into the $i$-th leaf of the leveled tree $T_{0}$ according to the permutation. We then complete the tree using bivalent vertices in order to get a leveled tree. Formally,  $T_{0} \circ_{i} T_{i}$ is given by the sequence of surjective maps
\[
  [n_{i}+k-1] \to  V_{h_{i}}(T_{i}) \sqcup \big( [k]\setminus \{i\} \big) \xrightarrow{t^{i}_{h_{i}-1}\sqcup \id} \dots \xrightarrow{t^{i}_{0} \sqcup \id} \{r\} \sqcup ( [k]\setminus \{i\} \cong [k] \xrightarrow{t^{0}_{h_{0}}} V_{h_{0}}(T_{0}) \xrightarrow{t^{0}_{h_{0}-1}} \dots \xrightarrow{t^{0}_{0}} \{ r \}.
\]
The order on $V_{j}(T_{i}) \sqcup \big( [k]\setminus \{i\} \big)$ is inherited from $V_{j}(T_{i})$ and $[k]$. Furthermore, for any $v\in V_{j}(T_{i})$ and $l \in [k]\setminus \{i\}$, one has $v \ge l$ iff we have $i \ge j$. Finally, the first map in the sequence of surjections is given by
$$
  [n_{i}+k-1]\longrightarrow  V_{h_{i}}(T_{i}) \sqcup \big( [k]\setminus \{i\} \big),
  \quad
  l \longmapsto
  \begin{cases}
    l                    & \text{if } j\leq l,             \\
    t_{h_{i}}^{i}(j-i+1) & \text{if } i\leq l\leq n_{i}+i, \\
    j-n_{i}              & \text{if } j\geq n{i}+i+1.
  \end{cases}
$$
Note that by construction, we have $h(T_{0}\circ_{i}T_{i})=h(T_{0})+h(T_{i})+1$.
See Figure~\ref{famlevtree2} for examples.

\begin{figure}[htbp]
  \hspace{-10pt}\includegraphics[scale=0.4]{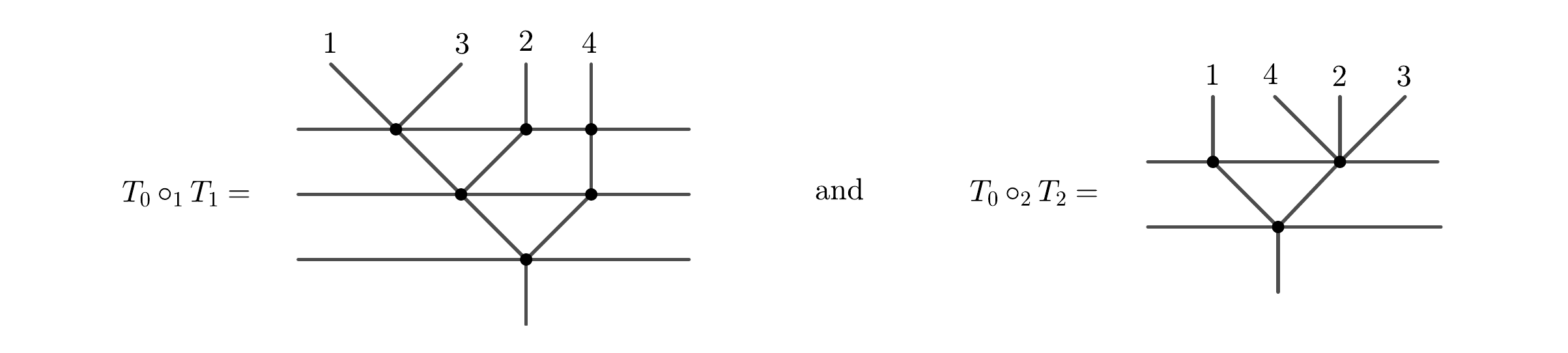}\vspace{-5pt}
  \caption{Partial compositions of the family represented in Figure~\ref{famlevtree}.}
  \label{famlevtree2}
\end{figure}\vspace{-2pt}

\begin{defi}
  The total composition is the leveled tree given by
  \begin{equation}\label{formulatreecomp}
    \gamma(T_{0}, \{T_{i}\}) \coloneqq \bigl(\cdots \bigl(\bigl(T_{0}\circ_{1}T_{1}\big) \circ_{n_{1}+1}T_{2}\bigr)\cdots \bigr) \circ_{n_{1}+\cdots + n_{k-1}+1}T_{k},
  \end{equation}
  which satisfies $h(\gamma(T_{0}, \{T_{i}\}))=h(T_{0})+h(T_{i_{1}})\cdots + h(T_{i_{n}})+n$.
  See Figure~\ref{famlevtree3} for an example.
\end{defi}

The operations $\gamma$ so obtained do not provide an operadic structure on the family $\lTree=\{\lTree[n]\}$.
Indeed, the associativity axiom is only satisfied up to permutation of levels.
Nevertheless, this structure will be enough to define a (co)operad at the level of (co)fibrant resolution or (co)bar construction.

\begin{figure}[htbp]
  \vspace{-10pt}\hspace{80pt}\includegraphics[scale=0.45]{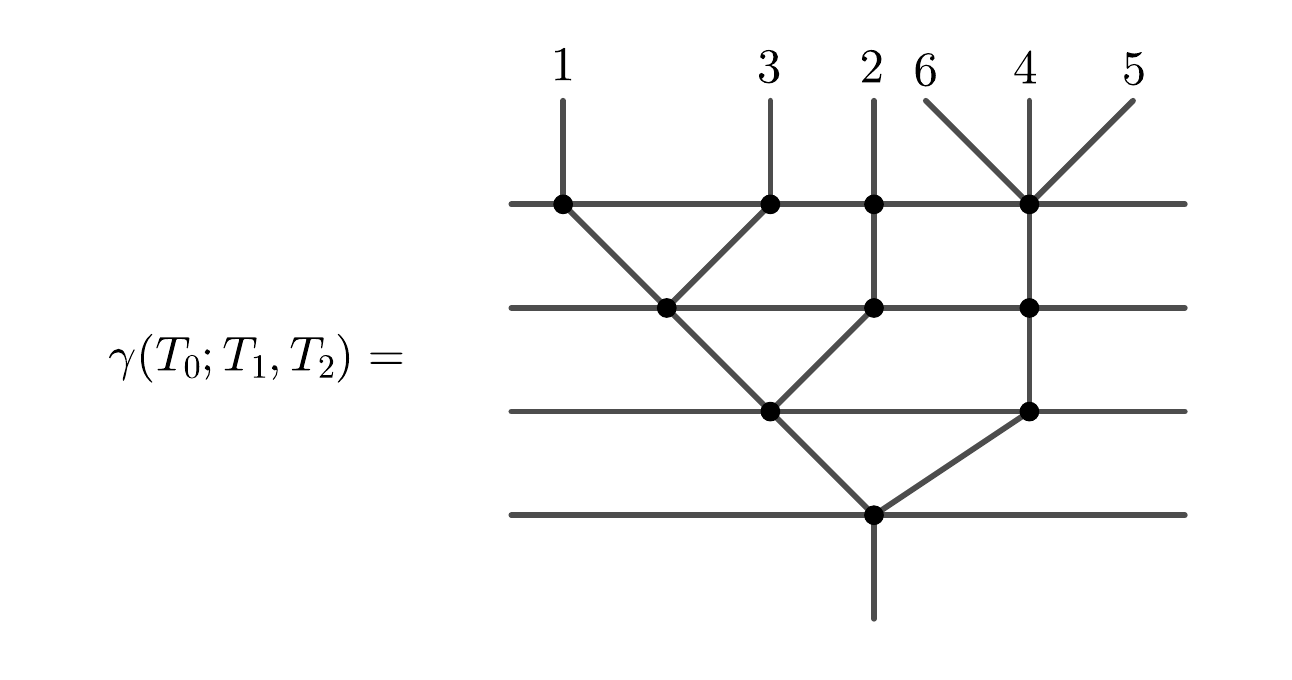}\vspace{-5pt}
  \caption{Total composition $\gamma(T_{0}; T_{1}, T_{2})$ of the family represented in Figure~\ref{famlevtree}}
  \label{famlevtree3}
\end{figure}

\begin{rmk}\label{rem:operadic categories}
  The operations so obtained have their origin in the theory of operadic categories.
  We refer the reader to \cite{BataninMarkl2015} for more details.
  In particular, the functor from the operadic category of leveled trees to the operad of planar trees is similar to our functor $\alpha$ connecting the usual and the levelled constructions.
\end{rmk}

\subsubsection{Comparison between planar trees and leveled trees}
\label{CompTree}

There are two functors $\alpha:\lTree[n] \rightarrow \Tree^{\geq 2}[n]$ and $\alpha_{\iso}:\lTree_{\iso}[n] \rightarrow \Tree_{\iso}^{\geq 2}[n]$ sending a leveled $n$-tree to the planar $n$-tree obtained by removing the bivalent vertices and taking the underlying level map.
In particular, contractions and permutations of permutable levels are sent to identity morphisms, so $\alpha$ and $\alpha_{\iso}$ are not faithful.
Moreover, for a given planar tree there are several ways of adding bivalent vertices to level it, so these functors are not injective on objects.

\begin{figure}[htbp]
  \begin{center}
    \includegraphics[scale=0.32]{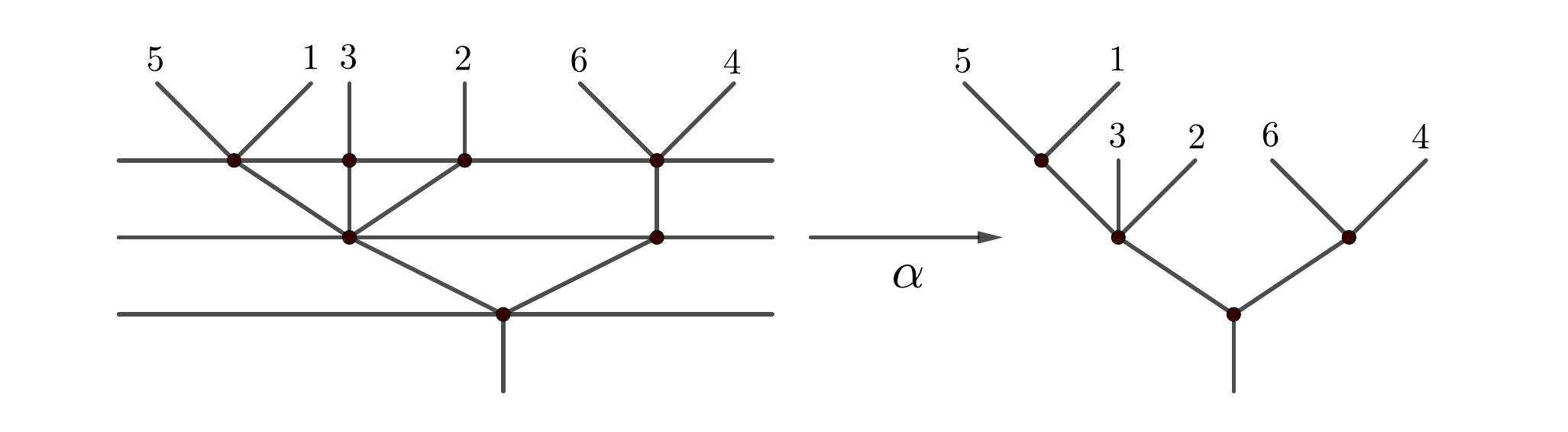}
    \includegraphics[scale=0.32]{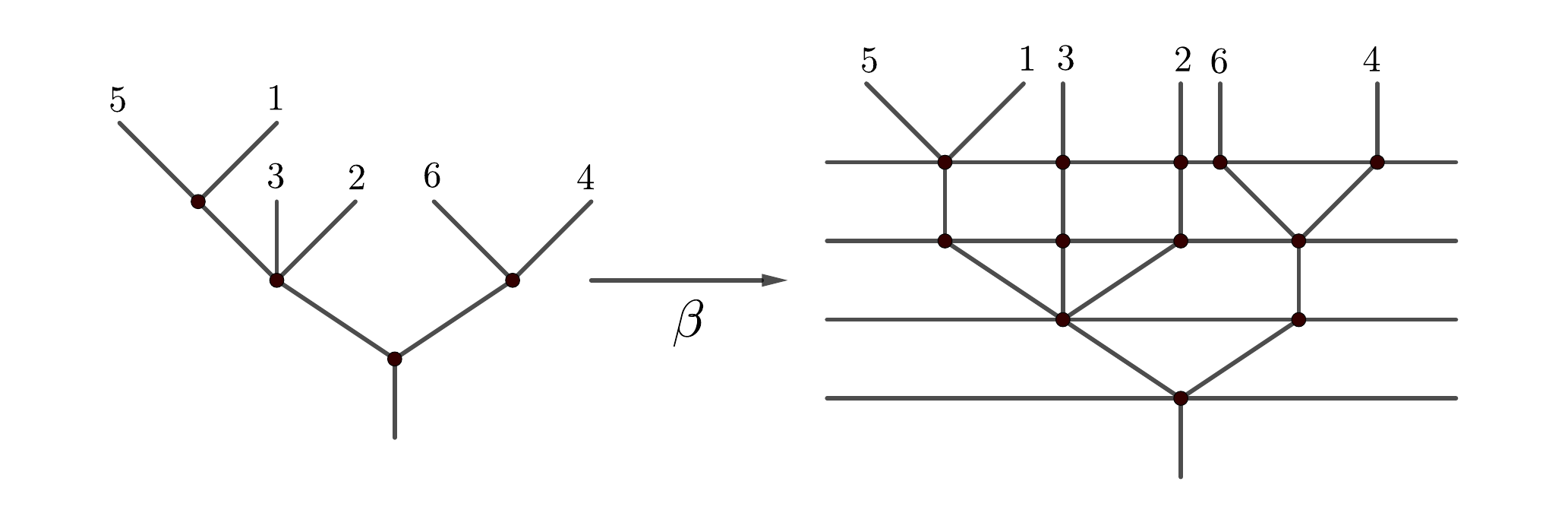}
    \caption{The functors $\alpha$ and $\beta$.}
    \label{MorphAlpha}
  \end{center}
\end{figure}

The functors $\alpha$ and $\alpha_{\iso}$ thus cannot be equivalences of categories.
Nevertheless, they are surjective on objects.
Indeed, for each rooted planar tree $T$, we fix $T_{l}$ to be the unique leveled tree for which  each level has exactly one non-bivalent vertex appearing from bottom to top according to the total order on the set of vertices $V(T)$.
We set $\beta:\Tree^{\geq 2}[n] \rightarrow \lTree[n]$ and $\beta_{\iso}:\Tree_{\iso}^{\geq 2}[n] \rightarrow \lTree_{\iso}[n]$ the two functors sending a planar $n$-tree $T$ to $T_{l}$.
These two functors are faithful and injective on the sets of objects.
However, they are not full, nor are they  surjective on objects or essentially surjective since, contrary to permutations of permutable levels, contractions of permutable levels are not isomorphisms.

The functors $\beta$ and $\beta_{\iso}$ so defined give rise to sections of the functors $\alpha$ and $\alpha_{\iso}$ in the sense that, for any planar $n$-tree $T$, one has $\alpha\circ\beta(T)=T$ and $\alpha_{\iso}\circ\beta_{\iso}(T)=T$. However, for any leveled tree $T$, $\beta\circ \alpha(T)$ and $\beta_{\iso}\circ \alpha_{\iso}(T)$ coincide with $T$ only up to contractions and permutations of permutable levels. Furthermore, all the functors considered are compatible with the operadic operations in the sense that the diagram
$$
  \begin{tikzcd}
    \Tree^{\geq 2}[k] \times \prod_{i=1}^k \Tree^{\geq 2}[n_{i}] \ar[r, "\gamma"] \ar[d, shift left, "\prod \beta"]
    & \Tree^{\geq 2}[n_{1}+\cdots +n_{k}] \ar[d, shift left, "\beta"]
    \\
    \lTree[k] \times \prod_{i=1}^k \lTree[n_{i}] \ar[r, "\gamma"] \ar[u, shift left, "\prod\alpha"]
    & \lTree[n_{1}+\cdots +n_{k}] \ar[u, shift left, "\alpha"]
  \end{tikzcd}
$$
commutes strictly when we restrict to $\alpha$ and it commutes up to contractions and permutations of permutable levels when we restrict to $\beta$. The same is true for the subcategories $\Tree_{\iso}$ and $\lTree_{\iso}$. We resume the above properties in the following theorem:

\begin{thm}\label{pro:now first theorem in intro}
  The functors $\alpha:\lTree[n] \rightarrow \Tree^{\geq 2}[n]$ and $\alpha_{\iso}:\lTree_{\iso}[n] \rightarrow \Tree_{\iso}^{\geq 2}[n]$, obtained removing the bivalent vertices, are full functors and surjective on the sets of objects. They admit right inverses $\beta:\Tree^{\geq 2}[n] \rightarrow \lTree[n]$ and $\beta_{\iso}:\Tree_{\iso}^{\geq 2}[n] \rightarrow \lTree_{\iso}[n]$, respectively, which are faithful and injective on the set of objects.
\end{thm}

\begin{rmk}\label{rmk:bivalent}
  Note that these two functors are well-defined because we use bivalent vertices as ``markers''.
  This is the main reason that we work with trees without levels consisting exclusively of bivalent vertices and why our constructions only work for $1$-reduced objects.
\end{rmk}

\subsection{The categories of planar and leveled trees with section}\label{SectTreeLev}

Similarly to the previous subsection, we introduce the two categories of planar trees with section $\sTree_{\iso}[n]$ and $\sTree[n]$ having the same set of objects and which differ in their morphisms.
Unlike $\sTree_{\iso}[n]$, the category $\sTree[n]$ includes morphisms contracting consecutive vertices.
The latter one is used to build resolutions while $\Tree_{\iso}[n]$ is often used to construct free bimodule objects.
Afterwards, we define the categories of leveled trees with section $\slTree_{\iso}[n]$ and $\slTree[n]$ together with some kind of bimodule structures over $\lTree_{\iso}[n]$ and $\lTree[n]$, respectively.
The last paragraph is devoted to the comparison between the categories of trees with section and their leveled versions.

\subsubsection{The category of planar trees with section}

A planar $n$-tree with section is a pair $(T,V_{\iota}(T))$ where $T$ is a planar $n$-tree and $V_{\iota}(T)$ is a subset of vertices, called pearls, satisfying the following condition: each path joining a leaf
to the root passes through a unique pearl. The pearls form a section cutting the tree into two parts. We usually denote by $V_{u}(T)$ and $V_{d}(T)$ the set of vertices above and below the section, respectively. Let us notice that the sets $V_{u}(T)$ and $V_{d}(T)$ inherit total orders from $V(T)$. We assume that vertices have at least $1$ incoming edge.

\begin{figure}[!h]
  \begin{center}
    \includegraphics[scale=0.33]{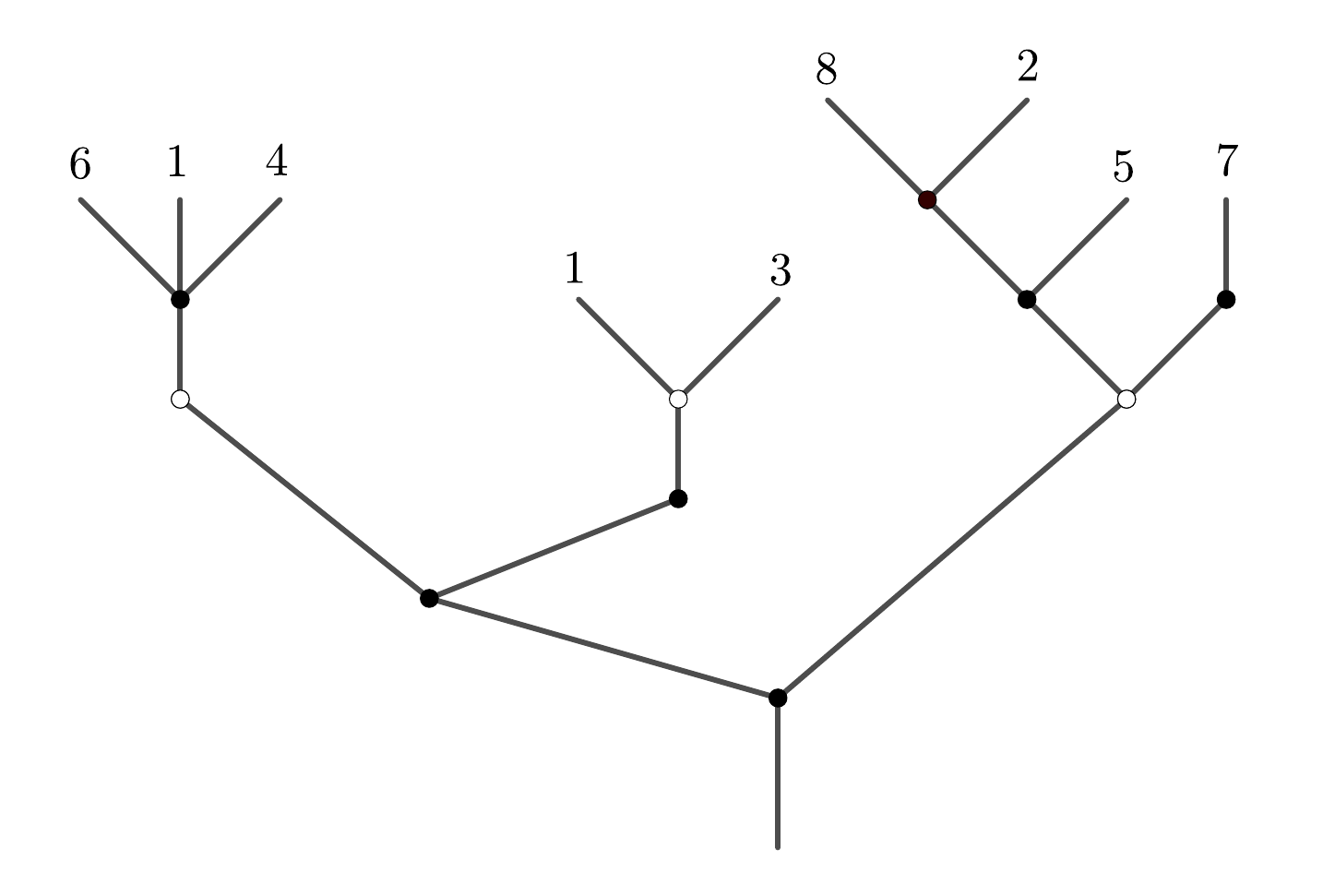}
    \caption{Illustration of a planar tree with section.}
  \end{center}
\end{figure}

\begin{defi}[The categories $\sTree$ and $\sTree_{\iso}$]\label{def:stree}
  An \emph{isomorphism} of planar $n$-trees with section is a bijection between vertices preserving the root, the pearls as well as the total order, and commuting with the target map. We also consider morphisms contracting consecutive vertices. There is an \textit{contracting morphism} from a planar tree $T$ to another tree $T'$ if there is a subset of inner edges $E'\subset E^{in}(T)$ such that the tree $T'$ is obtained from $T$ by removing the edges corresponding to $E'$ and by identifying the consecutive vertices $v$ and $v'$ for any $(v,v')\in E'$. Furthermore, we assume that the subset $E'$ satisfies the condition: if there is $(v,v')\in E'$ with $v\in V_{\iota}(T)$, then all the edges connecting $v'$ to a pearl are contained into $E'$.

  The category $\sTree_{\iso}[n]$ consists of planar $n$-trees with section and isomorphisms between them while $\sTree[n]$ is the category with the same set of objects and whose  morphisms are composed of isomorphisms and contracting morphisms. We also consider the subcategories $\sTree^{\geq 2}[n]$  and $\sTree^{\geq 2}_{\iso}[n]$, of planar $n$-trees whose vertices other than the pearls have at least $2$ antecedents (i.e. $|t^{-1}(v)|\geq 2$ for any $v\in V(T)$).
\end{defi}

\subsubsection{The category of leveled trees with section}\label{TreeSect}

A \emph{leveled $n$-tree with section}, with $n > 0$ is a pair $(T, \iota)$ where $T$ is a sequence of non-decreasing surjections as in \eqref{nondecsurj} and $0 \le \iota \le h(T)$ is an integer such that the surjective maps $t_{j}$ are not bijective for $j\neq \iota$.
The level corresponding to $\iota$ is called the \textit{main section} and can be composed of bivalent vertices.
In particular, if $(T, \iota)$ is a leveled $n$-tree with section, then $T$ is not necessarily a leveled $n$-tree, as $t_{\iota}$ may be bijective.
In pictures, we represent the main section by a dotted line.
We respectively denote by $V_{\iota}(T)$, $V_{u}(T)$, and $V_{d}(T)$ the sets of vertices on the main section, above the main section, and below the main section.
Such a tree will be denoted by $T$ if there is no ambiguity about the main section.

\begin{figure}[htbp]
  \centering
  \includegraphics[scale=0.4]{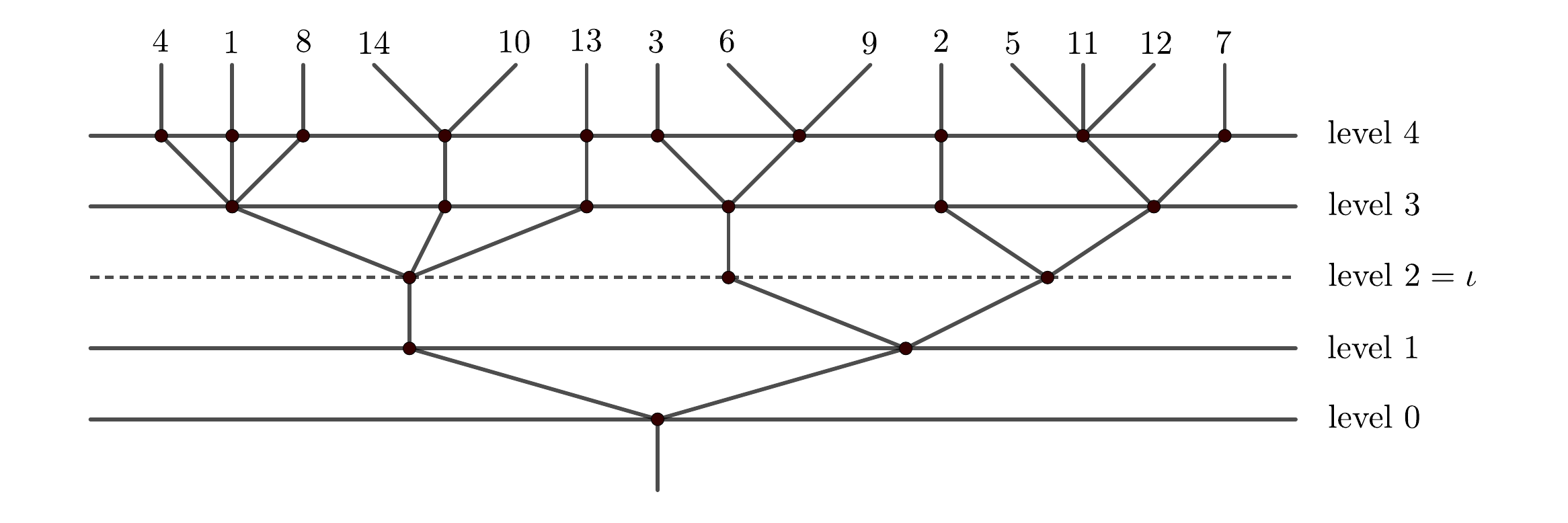}
  \caption{A leveled $14$-tree with section.}
  \label{fig:example-tree-section}
\end{figure}

\begin{defi}[{The categories $\slTree[n]$ and $\slTree_{\iso}[n]$}]\label{def:sltree}
  In the following we introduce three kinds of elementary morphisms between leveled trees. The categories $\slTree[n]$ and $\slTree_{\iso}[n]$ of leveled $n$-trees have the same set of objects. Morphisms in $\slTree_{\iso}[n]$ are generated by isomorphisms of leveled trees preserving the main section, contractions of permutable levels and permutations $\sigma_{i}$ of permutable levels such that, in both cases, neither $i$ nor $i+1$ are the main section $\iota$ (i.e.\ $\iota \not\in \{i,i+1\}$).

  On the other hand, morphisms in  $\slTree[n]$ are generated by isomorphisms of leveled trees preserving the main section, contractions of consecutive levels, and permutations $\sigma_{i}$ of permutable levels such that neither $i$ nor $i+1$ are the main section $\iota$.
\end{defi}

\subsubsection{On ``bimodule'' structures for (possibly leveled) trees with section}

We now introduce the operations needed in order to define (co)bimodule structures on Boardman--Vogt resolutions of (co)bimodules, and which are compatible with the operations introduced in the previous sections. We will also use them to define (co)bimodule structures on alternative versions of two-sided (co)bar constructions. We build the following ``right'' and ``left'' operations (see Equations~\eqref{eq:def-gamma-r} and~\eqref{eq:def-gamma-l})
\begin{align*}
  \gamma_{R} : \slTree[k] \times \lTree[n_{1}] \times \cdots \times \lTree[n_{k}] & \longrightarrow \slTree[n_{1}+\cdots + n_{k}]; \\
  \gamma_{L} : \lTree[k] \times \slTree[n_{1}]\times \cdots \times \slTree[n_{k}] & \longrightarrow \slTree[n_{1}+\cdots + n_{k}].
\end{align*}

The right operation $\gamma_{R}$ is defined as follows. Consider trees $(T_{0},\iota) \in \slTree[k]$ and $T_{i}\in \lTree[n_{i}]$ for $i\leq k$. The right module operation $\gamma_{R}(T_{0}; \{T_{i}\})$ is given by the following formula in which $\gamma$ is the total composition introduced in Section \ref{SectTree}:
\begin{equation}\label{eq:def-gamma-r}
  \gamma_{R}(T_{0}; \{T_{i}\}) = \bigl( \gamma(T_{0}; \{T_{i}\}), \iota \bigr).
\end{equation}

Let us now define the left module operation $\gamma_{L}$. Let $T_{0} \in \lTree[k]$ and $T_{i} =(T_{i}, \iota_{i}) \in \slTree[n_{i}]$ for $i\leq k$. For the sake of example, we will depict the operation $\gamma_{L}$ when applied to the family of trees from Figure~\ref{FigOpLeft}.

\begin{figure}[htbp]
  \hspace{-45pt}\includegraphics[scale=0.4]{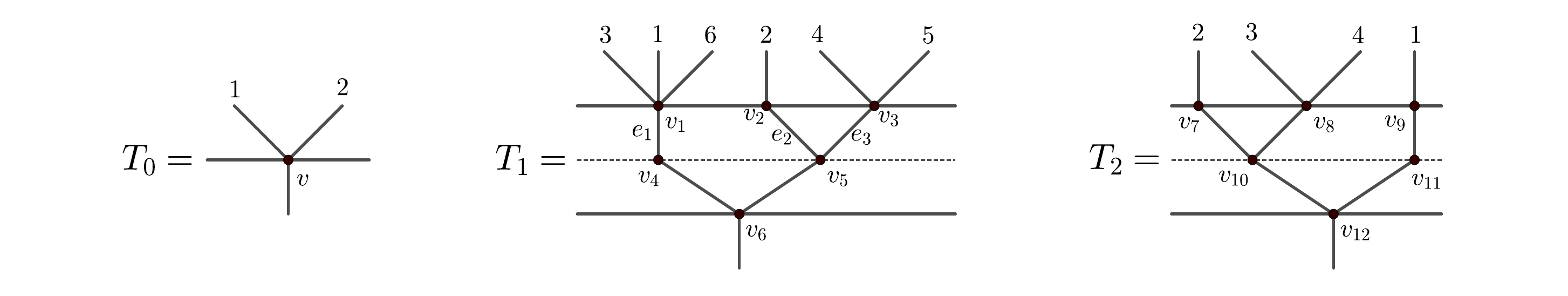}
  \caption{Example of family of leveled trees with $k=2$.}
  \label{FigOpLeft}
\end{figure}

For any $i\leq k$, we denote by $T_{i}^{<}$ the leveled sub-tree of $T_{i}$ composed of the vertices and edges strictly below the main section.
Similarly, for any $i\leq k$, $v \in V_{\iota}(T_{i})$, and $e\in in(v)$ an incoming edge of $v$, we denote by $T^{>e}_{i}$ the leveled sub-tree of $T_{i}$ consisting of all vertices and edges above $e$, having $e$ as the trunk.
Formally, for $\iota < j \le h(T_{i})$), we define the set of vertices of level $j$ by:
\[ V_{j}(T^{>e}_{i}) \coloneqq \{ w \in V_{j}(T) \mid (t^{k-1}(w),t^k(w)) = e \text{ for some } k>0  \} \]
formed by vertices above the edge $e$.
We also denote the leaves of $T_{i}^{>e}$ by $[n]^{>e} = \{ s \in [n] \mid \exists k > 0 \text{ s.t. } (t^{k-1}(s), t^{k}(s)) = e \}$, which we identify with $[n^{>e}]$ for some $n_{e} > 0$.
Then $T^{>e}_{i}$ is the leveled tree given by the sequence of non-decreasing surjections:
\[
  [n^{>e}] \xtwoheadrightarrow{t_{h(T)|[n]^{>e}}} V_{h(T),e}(T_{i}) \xtwoheadrightarrow{t_{h(T)-1|V_{h(T),e}(T_{i})}} \cdots \xtwoheadrightarrow{t_{\iota+1|V_{\iota+2,e}(T_{i})}} V_{\iota+1,e}(T_{i}).
\]

\begin{figure}[htbp]
  \hspace{-110pt}\includegraphics[scale=0.5]{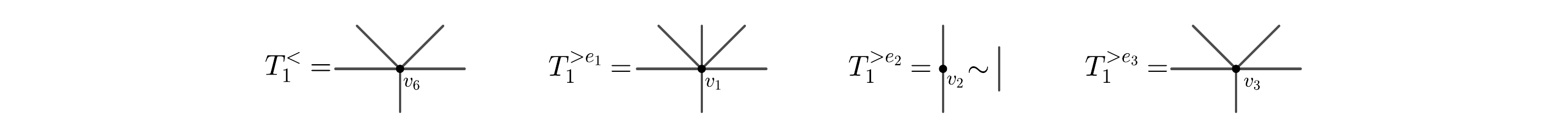}\vspace{-10pt}
  \caption{Sub-trees associated to $T_{1}$ represented in Figure \ref{FigOpLeft}.}
  \label{fig:subtrees}
\end{figure}

First, we consider the leveled tree with section $\Delta(T_{0},\{T_{i}\})$ obtained by grafting into the leaves of $\gamma(T_{0}\,;\,\{T_{i}^{d}\})$ the corresponding vertices in $V_{\iota}(T_{i})$, with $i\leq k$. Furthermore, we remove the sections composed of only bivalent vertices. The main section of this leveled tree so obtained is the top level denoted by $\Delta(\{\iota_{i}\}) = h(T_{0}) + \sum_{i\in I}\iota_{i}$. See Figure~\ref{fig:leveled-delta} for an example.

\begin{figure}[htbp]
  \centering
  \includegraphics[scale=0.4]{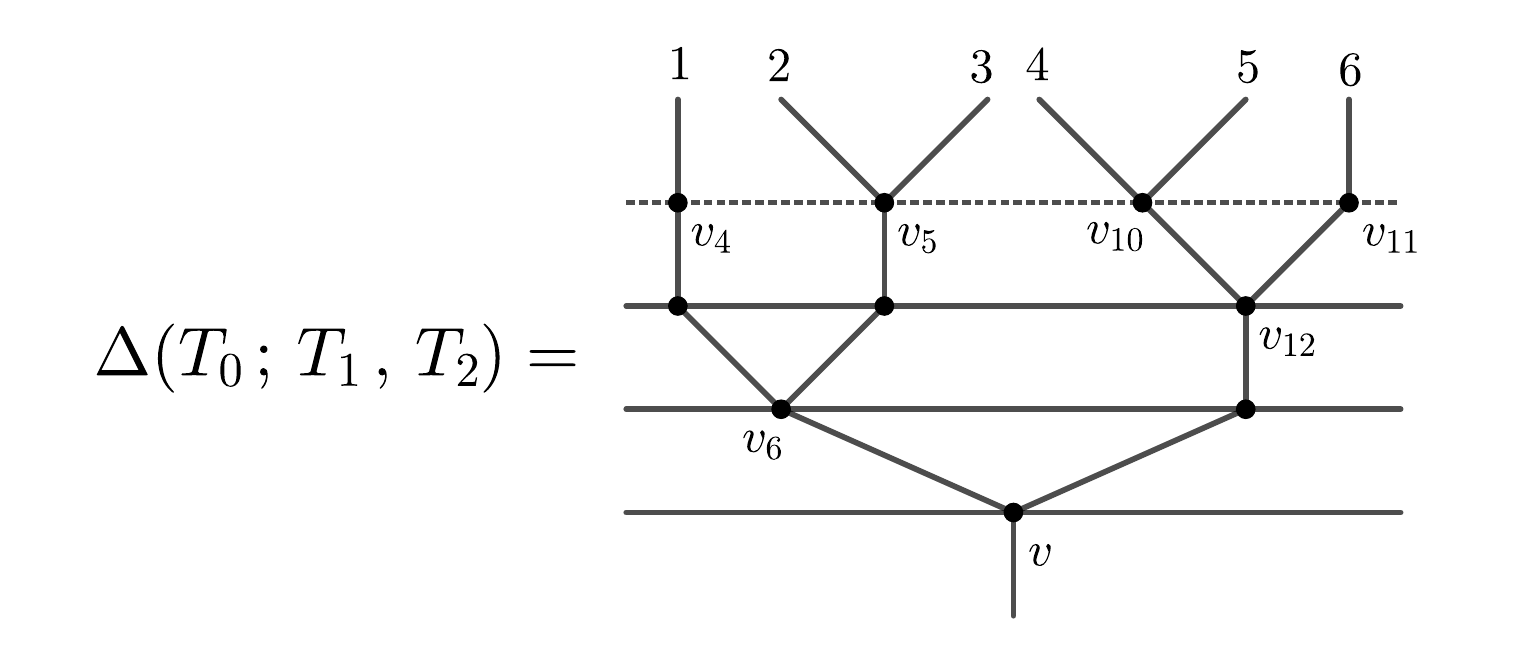}\vspace{-10pt}
  \caption{The leveled tree $\Delta(T_{0},\{T_{i}\})$ associated to  the family represented in Figure \ref{FigOpLeft}.}
  \label{fig:leveled-delta}
\end{figure}

\begin{defi}
  The left operation $\gamma_{L}$ is given by the formula (see Figure~\ref{fig:example-gamma-l}):
  \begin{equation}\label{eq:def-gamma-l}
    \gamma_{L}(T_{0}\,;\,\{T_{i},\iota_{i}\}) \coloneqq \bigl( \gamma(\Delta(T_{0},\{T_{i}\})\,;\,\{T_{i,e}\})\,\,;\,\,\Delta(\{\iota_{i}\})\,\bigr).
  \end{equation}
\end{defi}

\begin{figure}[htbp]
  \centering
  \includegraphics[scale=0.3]{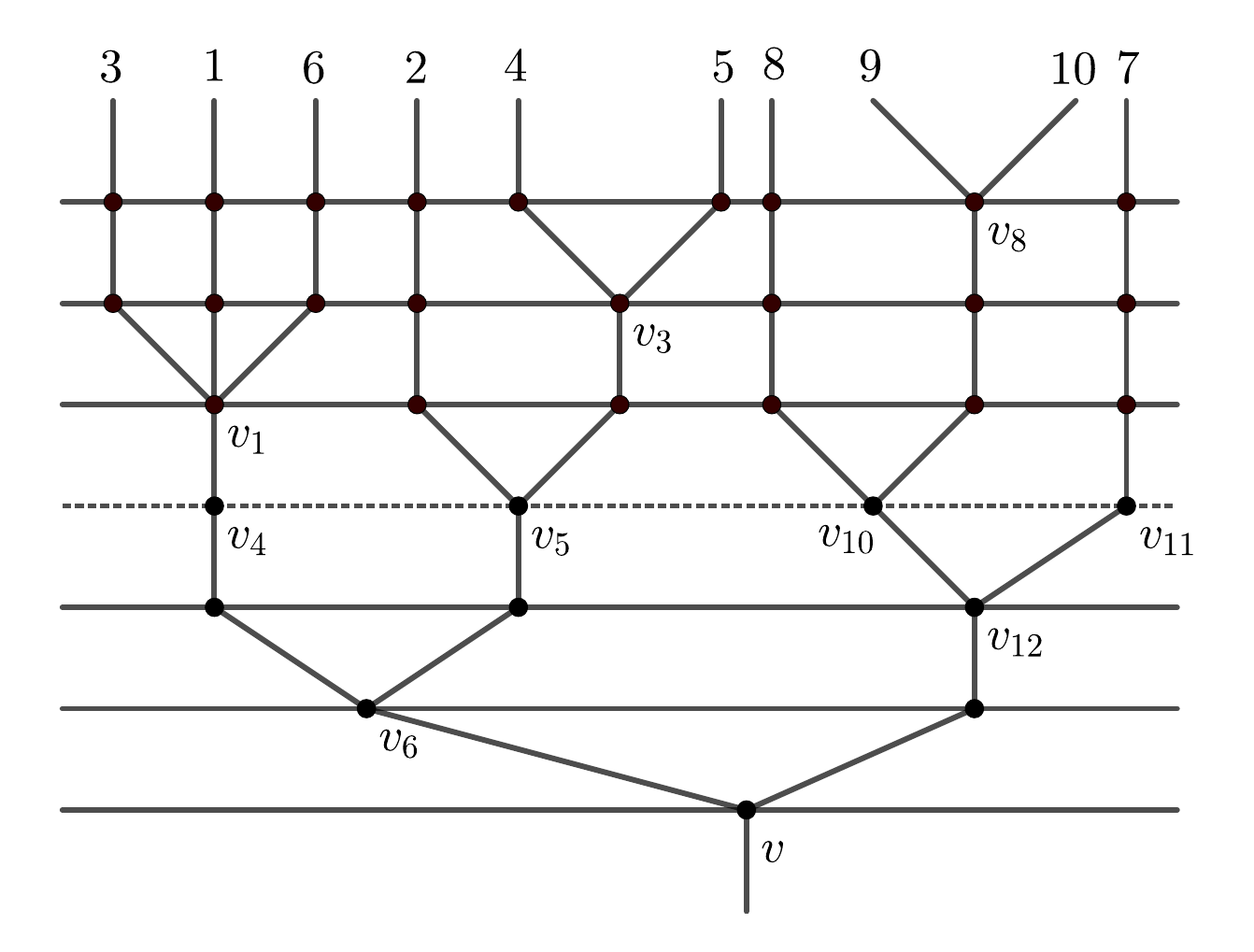}\vspace{-15pt}
  \caption{The image under $\gamma_{L}$ of the family from Figure~\ref{FigOpLeft}.}
  \label{fig:example-gamma-l}
\end{figure}

As in the previous section, the family of sets $\slTree = \{\slTree[n]\}$, equipped with the left and right module operations $\gamma_{L}$ and $\gamma_{R}$, is not a bimodule over $\lTree$ (which is not even an operad). The bimodule axioms are only satisfied up to permutations and contractions of permutable consecutive levels. Nevertheless, this will be enough to define (co)bimodule structures on Boardman--Vogt resolutions or alternative versions of two-sided (co)bar resolutions.

\subsubsection{Comparison between planar trees and leveled trees}\label{sec:comp-planar-leveled}

There are two functors $\alpha:\slTree[n] \rightarrow \sTree^{\geq 2}[n]$ and $\alpha_{\iso}:\slTree_{\iso}[n] \rightarrow \sTree_{\iso}^{\geq 2}[n]$ sending a leveled $n$-tree with section to the planar $n$-tree with section obtained by removing the bivalent vertices other than the pearls and taking the underlying level map. In particular, contractions and permutations of permutable levels are sent to identity morphisms. So, $\alpha$ and $\alpha_{\iso}$ are neither faithful nor injective on objects.

Nevertheless,  $\alpha$ and $\alpha_{\iso}$ are full and surjective on objects. Indeed, for each rooted planar tree $T$, we fix $T_{l}$ to be the unique leveled tree for which  each level other than the main section has exactly one non-bivalent vertex appearing from bottom to top according to the total order on the sets of vertices $V_{d}(T)$ and  $V_{u}(T)$. We set $\beta:\Tree^{\geq 2}[n] \rightarrow \lTree[n]$ and $\beta_{\iso}:\Tree_{\iso}^{\geq 2}[n] \rightarrow \lTree_{\iso}[n]$ the two functors sending a planar $n$-tree with section $T$ to $T_{l}$.
These two functors are faithful and injective on the sets of objects. However there are neither full, nor surjective on objects, nor essentially surjective: unlike permutations of permutable levels, contractions of permutable levels are not isomorphisms.

\begin{figure}[htbp]
  \begin{center}
    \includegraphics[scale=0.32]{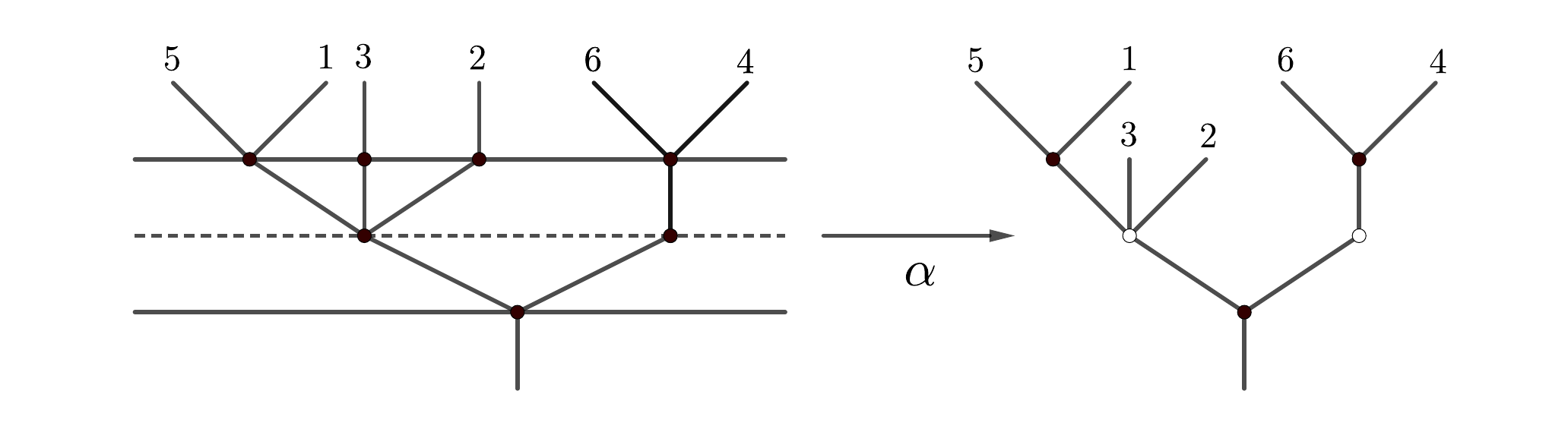}
    \includegraphics[scale=0.32]{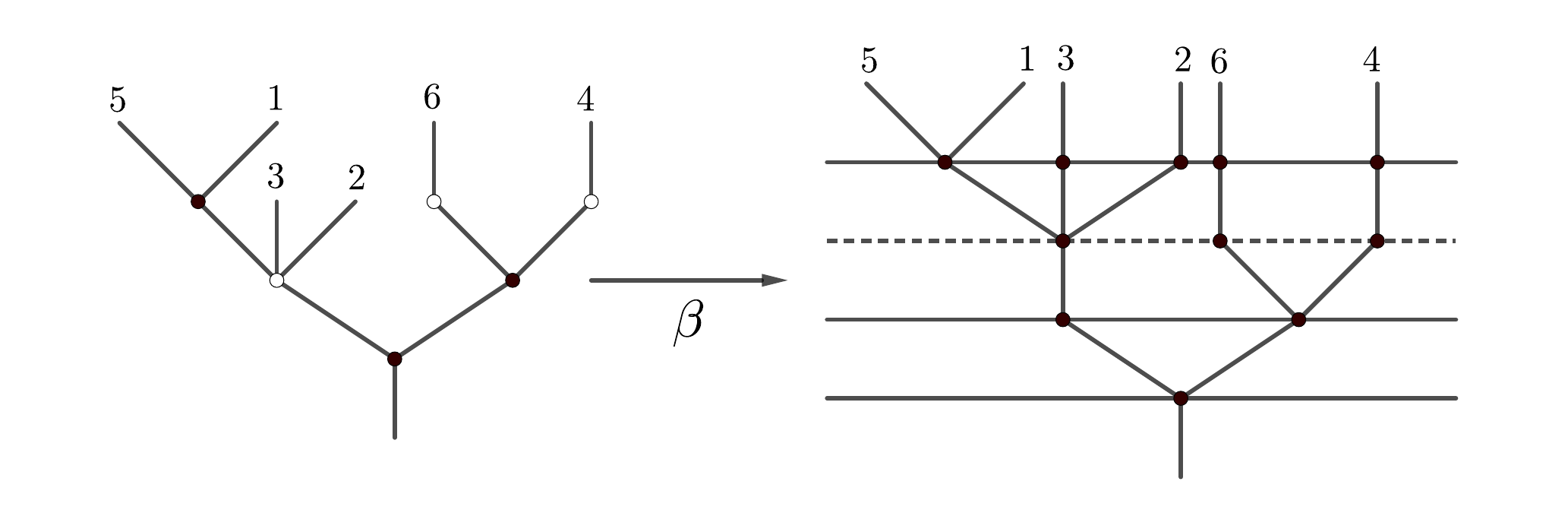}\vspace{-15pt}
    \caption{The functors $\alpha$ and $\beta$.}
    \label{MorphAlpha2}
  \end{center}
\end{figure}

The functors $\beta$ and $\beta_{\iso}$ so defined give rise to sections of the functors $\alpha$ and $\alpha_{\iso}$ in the sense that, for any planar $n$-tree with section $T$, one has $\alpha\circ\beta(T)=T$ and $\alpha_{\iso}\circ\beta_{\iso}(T)=T$. However, for any leveled tree with section $T$, $\beta\circ \alpha(T)$ and $\beta_{\iso}\circ \alpha_{\iso}(T)$ coincide with $T$ only up to contractions and permutations of permutable levels. Furthermore, all the functors considered are compatible with the bimodule operations in the sense that the diagrams
$$
  \begin{tikzcd}
    \sTree^{\geq 2}[k] \times \prod_{i=1}^k \Tree^{\geq 2}[n_{i}] \ar[r, "\gamma_{R}"] \ar[d, shift left, "\prod \beta"]
    & \sTree^{\geq 2}[n_{1}+\cdots +n_{k}] \ar[d, shift left, "\beta"]
    & \Tree^{\geq 2}[k] \times \prod_{i=1}^k \sTree^{\geq 2}[n_{1}] \ar[r, "\gamma_{L}"] \ar[d, shift left, "\prod \beta"]
    & \sTree^{\geq 2}[n_{1}+\cdots +n_{k}] \ar[d, shift left, "\beta"]
    \\
    \slTree[k] \times \prod_{i=1}^k \lTree[n_{i}] \ar[r, "\gamma_{R}"] \ar[u, shift left, "\prod\alpha"]
    & \slTree[n_{1}+\cdots +n_{k}] \ar[u, shift left, "\alpha"]
    & \lTree[k] \times \prod_{i=1}^k \slTree[n_{i}] \ar[r, "\gamma_{L}"] \ar[u, shift left, "\prod\alpha"]
    & \slTree[n_{1}+\cdots +n_{k}] \ar[u, shift left, "\alpha"]
  \end{tikzcd}
$$
commute strictly when we restrict to $\alpha$ and up to contractions and permutations of permutable levels when we restrict to $\beta$.
The same is true for the subcategories $\sTree_{\iso}$ and $\slTree_{\iso}$. We resume the above properties in the following proposition:

\begin{thm}\label{prop:alpha-beta-section}
  The functors $\alpha:\slTree[n] \rightarrow \sTree^{\geq 2}[n]$ and $\alpha_{\iso}:\slTree_{\iso}[n] \rightarrow \sTree_{\iso}^{\geq 2}[n]$, obtained removing the bivalent vertices other than the pearls, are full and surjective on objects. They admit right inverses $\beta:\sTree^{\geq 2}[n] \rightarrow \slTree[n]$ and $\beta_{\iso}:\sTree_{\iso}^{\geq 2}[n] \rightarrow \slTree_{\iso}[n]$, respectively, which are faithful and injective on the set of objects.
\end{thm}

\begin{rmk}\label{rmk:batanin}
  It is possible to adapt our construction to trees without levels as in~\cite{Ducoulombier2018c}.
  However, it is more convenient to use leveled trees in order to construct the bar and cobar constructions for bimodules.
  For instance, in the following tree:
  \begin{center}
    \includegraphics[scale=0.5]{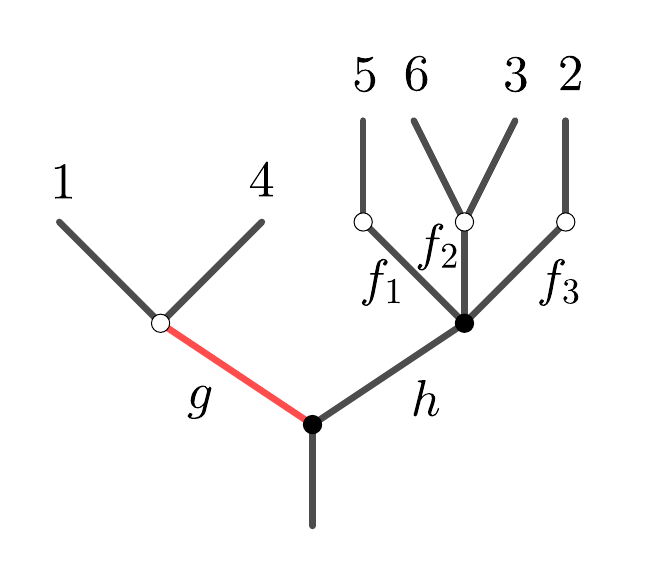}
  \end{center}
  one can contract the edge $h$, or the edges $f_1,f_2,f_3$ together, or the whole tree, but not the edge $g$ alone.
\end{rmk}

\section{Cofibrant resolutions for $\Lambda$-operads in spectra}
\label{sec:cofibr-resol-lambda}

For any  1-reduced operad $\calO$ in $\Sp$ (i.e.\ an operad  $\calO(0) = \calO(1) = *$), we introduce alternative (leveled) definitions of the Boardman--Vogt resolution $W_{l}\calO$ and the bar construction $\mathcal{B}_{l}(\calO)$ of $\calO$.
After that, we prove that the leveled bar resolution of $\calO$ is isomorphic to the cooperad of the indecomposable elements $\Ind(W_{l}\calO)$ .
In the last section, we show that the Boardman--Vogt resolution is also weakly equivalent to the cobar-bar construction related to $\calO$.
Throughout this section all (co)operads will be considered over the category of spectra.

\subsection{The leveled bar construction for operads in spectra}
\label{SectBarOp}

In this section, we introduce an alternative description of the bar construction for operads in spectra using leveled trees.
We show that our construction is isomorphic to the usual one introduced by Salvatore \cite{salvatore1998configuration} and Ching \cite{Chi}.
In what follows, the indices ``B''  emphasize the fact that these functors are used to define the bar construction.
This is to distinguish them from the functors in the next section, which are used to define the W-construction and are decorated by indices ``W''.

Given a 1-reduced operad $\calO$ in spectra, for every $n > 0$ we define the following two functors:

\begin{equation}\label{EqFunctOH}
  \begin{aligned}
    \overline{\calO}_{B} : \lTree[n] & \longrightarrow \Sp \; ,
                                     & T                          & \longmapsto  \underset{v \in V(T)}{\bigwedge} \calO(|v|); \\
    H_{B} : \lTree[n]^{op}           & \longrightarrow \sSet \; ,
                                     & T                          & \longmapsto \begin{cases}
      \Delta[T]/\Delta_{0}[T] & \text{if } n>1, \\
      \hspace{2em}\ast        & \text{if } n=1.
    \end{cases}
  \end{aligned}
\end{equation}
where $\Delta[T]=\prod_{0\leq i \leq h(T)}\, \Delta[1]$ labels the levels by elements in the standard $1$-simplex $\Delta[1]$, while $\Delta_{0}[T]$ is the simplicial subset consisting of faces where either the $0$-th level has value $0$, or any of the other levels has value $1$. By definition, $H_{B}(T)$ is a pointed  simplicial set for any leveled tree $T$ whose basepoint is the equivalence class of $\Delta_{0}[T]$.

On morphisms, the functor $\overline{\calO}_{B}$ is defined using the operadic structure of $\calO$.
For any two consecutive permutable levels $i$ and $i+1$, $H_{B}(\sigma_{i})$ permutes the simplices corresponding to the $i$-th and $(i+1)$-st levels.
For contraction morphisms there are two cases to consider:
\begin{enumerate}
  \item If the levels $i$ and $i+1$ are permutable, then, by using the diagonal map, one has:
        \begin{align*}
          H_{B}(\delta_{\{i+1\}}): H_{B}(T/\{i+1\}) & \longrightarrow H_{B}(T),                                \\
          (t_{0},\dots,t_{h(T)-1})                  & \longmapsto (t_{0},\dots,t_{i},t_{i},\dots, t_{h(T)-1}).
        \end{align*}
  \item If the levels $i$ and $i+1$ are not permutable, then one has instead:
        \begin{align*}
          H_{B}(\delta_{\{i+1\}}): H_{B}(T/\{i+1\}) & \longrightarrow H_{B}(T),                                    \\
          (t_{0},\dots,t_{h(T)-1})                  & \longmapsto (t_{0},\dots,t_{i},0,t_{i+1},\dots, t_{h(T)-1}).
        \end{align*}
\end{enumerate}

\begin{defi}
  The leveled bar construction of a 1-reduced operad $\calO$ in spectra is defined as the simplicial spectrum given by the coend:
  $$
    \calB_{l}(\calO)(n) \coloneqq \int^{T\in \lTree[n]}\overline{\calO}_{B}(T)\wedge H_{B}(T).
  $$
\end{defi}

A point in $\calB_{l}(\calO)(n)$ is the data of a leveled $n$-tree $T$, a family of points in $\calO$ labelling the vertices $\{\theta_{v}\}_{v\in V(T)}$ and a family of elements in the simplicial set $\Delta[1]$ indexing the levels $\{t_{j}\}_{0\leq j \leq h(T)}$.
The equivalence relation induced by the coend is generated by the compatibility with the symmetric group action, permutations of permutable levels, contractions of two consecutive permutable levels indexed by the same simplex, and contractions of consecutive non-permutable levels such that the upper one is indexed by $0$.
Such a point is denoted by $[T\,;\, \{\theta_{v}\}\,;\,\{t_{j}\}]$.

\begin{figure}[!h]
  \hspace{-45pt} \includegraphics[scale=0.37]{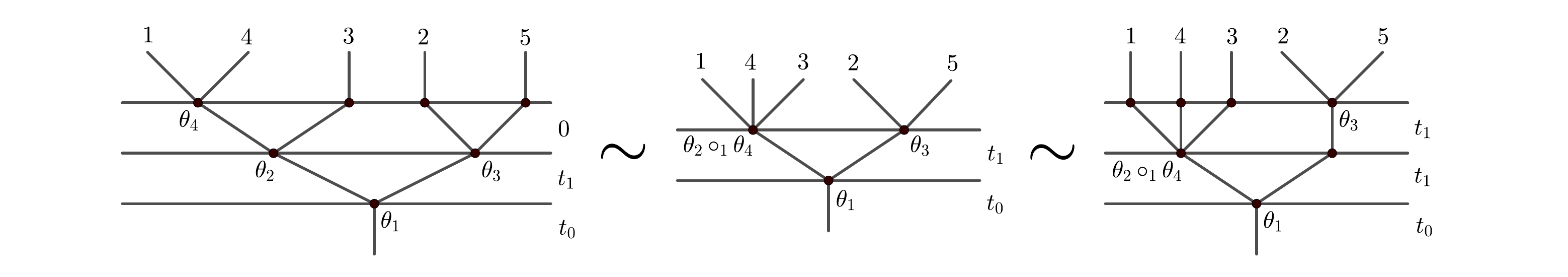}\vspace{-10pt}
  \caption{Illustration of equivalent points in $\calB_{l}(\calO)(5)$.}
  \label{fig:equivalence-bo}
\end{figure}

The sequence $\calB_{l}(\calO)=\{\calB_{l}(\calO)(n)\}$ has a cooperadic structure
\begin{equation}\label{coopstructop}
  \gamma^{c}:\calB_{l}(\calO)(n_{1}+\cdots + n_{k})\longrightarrow \calB_{l}(\calO)(k) \wedge \calB_{l}(\calO)(n_{1})\wedge \cdots \wedge \calB_{l}(\calO)(n_{k}),
\end{equation}
defined as follows. A leveled $n$-tree $T$ is said to be \textit{decomposable} according to the partition $(n_{1},\ldots,n_{k})$, with $n_{1}+\cdots + n_{k}=n$ if there exist leveled trees $T_{0}\in \lTree[k]$ and $T_{i}\in \lTree[n_{i}]$, with $i\leq k$, such that $T$ is of the form $\gamma(T_{0}\,,\,\{T_{i}\})$ up to permutations of permutable levels and contractions of permutable levels where $\gamma$ is the operation \eqref{formulatreecomp}. According to this notation, if $T$ is not decomposable, then $\gamma^{c}([T\,;\, \{\theta_{v}\}\,;\,\{t_{j}\}])$ is sent to the basepoint. Otherwise, let us remark that we have an identification
$$
  [T\,;\, \{\theta_{v}\}\,;\,\{t_{j}\}]= [\gamma(T_{0}\,,\,\{T_{i}\})\,;\, \{\theta_{v}\}\,;\,\{\tilde{t}_{j}\}]
$$
due to the equivalence relation induced by the coend. In that case, we define
$$
  \gamma^{c}([T\,;\, \{\theta_{v}\}\,;\,\{t_{j}\}]) \coloneqq \bigl\{ [T_{i}\,;\, \{\theta_{v}^{i}\}\,;\,\{t_{j}^{i}\}]\bigr\}_{i\in I\sqcup \{0\}}\in \calB_{l}(\calO)(I) \wedge \underset{i\in I}{\bigwedge} \calB_{l}(\calO)(S_{i})
$$
where $\{\theta_{v}^{i}\}$ and $\{t_{j}^{i}\}$ come from the restriction to the parameters corresponding to the sub-tree $T_{i}$ of $\gamma(T_{0}\,,\,\{T_{i}\})$.
The cooperadic operation \eqref{coopstructop} does not depend on the choice of the decomposition of $T$ up to permutations of permutable levels and contractions of permutable levels thanks to the definition of the coend.

\begin{defi}[The usual bar construction for operads in spectra]\label{def:usual-bar}
  For more details, we refer the reader to \cite{Chi2}.  We recall that $\Tree^{\geq 2}[n]$ is the category of planar $n$-trees having vertices with valences $\geq 2$ and whose morphisms are generated by isomorphisms of planar trees and contractions of inner edges. Given a 1-reduced operad $\calO$, we introduce the two functors
  \begin{equation}\label{H'}
    \begin{aligned}
      \overline{\calO}'_{B} : \Tree^{\geq 2}[n] & \longrightarrow \Sp \; ,
                                                & T                          & \longmapsto  \underset{v \in V(T)}{\bigwedge} \calO(|v|); \\
      H_{B}' : \Tree^{\geq 2}[n]^{op}           & \longrightarrow \sSet \; ,
                                                & T                          & \longmapsto
      \begin{cases}
        \Delta'[T]/\Delta_{0}'[T] & \text{if } n>1, \\
        \hspace{2em}\ast          & \text{if } n=1.
      \end{cases}
    \end{aligned}
  \end{equation}
  where $\Delta'[T] \coloneqq \prod_{v\in V(T)}\, \Delta[1]$ labels the vertices by elements in the simplicial $1$-simplex $\Delta[1]$ while $\Delta'_{0}[T]$ is the simplicial subset consisting of faces where, either, the root has value $0$, or, any other vertices has value $1$.
  By definition, $H'(T)$ is a pointed simplicial set for any leveled tree $T$ and the bar construction of $\calO$ is defined as the coend
  $$
    \calB(\calO)(n) \coloneqq \int^{T\in \Tree^{\geq 2}[n]}\, \overline{\calO}'_{B}(T)\wedge H'_{B}(T).
  $$
  A point is denoted by $[T\,;\,\{\theta_{v}\}\,;\,\{t_{v}\}]$ where $T\in \Tree^{\geq 2}[n]$ is a planar tree, $\{\theta_{v}\}_{v\in V(T)}$ is a family of points in $\calO$ labelling the vertices and $\{t_{v}\}_{v\in v\in V(T)}$ is a family of elements in $\Delta[1]$ indexing the vertices.
\end{defi}

\begin{pro}\label{IsoBarOp}
  The leveled bar construction is isomorphic to the usual bar construction denoted by $\calB(\calO)$:
  $$
    \calB_{l}(\calO)\cong \calB(\calO).
  $$
\end{pro}

\begin{proof}
  The proposition is a direct consequence of the comparison morphisms between planar $n$-trees and leveled $n$-trees explained in Section \ref{CompTree}. The isomorphism of operads is given by
  \begin{align*}
    f : \calB_{l}(\calO)(n) & \longrightarrow  \calB(\calO)(n),
                            & \bigl[ T; \{\theta_{v}\}_{v\in V(T)}; \{t_{i}\}_{0\leq i\leq h(T)} \bigr] & \longmapsto \bigl[ \alpha(T); \{\theta_{v}\}_{v\in V(T)}; \{t'_{v}\}_{v\in V(\alpha(T))} \bigr],
  \end{align*}
 where $t'_{e}$ is the maximum of the parameters corresponding to the levels related to the path joining the source vertex of $e$ to its first non-bivalent vertex according to the orientation toward the root. Conversely, one has the continuous map:
  \begin{align*}
    g : \calB(\calO)(n) & \longrightarrow \calB_{l}(\calO)(n),
                        & \bigl[ T; \{\theta_{v}\}_{v\in V(T)}; \{t_{v}\}_{v\in V(T)} \bigr]
                        & \longmapsto \bigl[ \beta(T); \{\theta_{v}\}_{v\in V(T)}; \{t'_{i}\}_{0\leq i\leq h(\beta(T))} \bigr],
  \end{align*}
  where $t'_{i}=t_{v}$ if the unique non-bivalent vertex on the $i$-th level of $\beta(T)$ corresponds to the vertex $v$ in $T$. The reader can easily check that the maps so obtained are well defined, compatible with the cooperadic structures and give rise to isomorphisms between the leveled and usual bar resolutions.
\end{proof}

\subsection{The leveled Boardman--Vogt resolution for 1-reduced $\Lambda$-operads}\label{SectBVop}

This section is split into three parts.
First, we introduce a leveled version of the Boardman--Vogt resolution for 1-reduced operads in spectra.
Then, we compare this alternative construction to the usual Boardman--Vogt resolution introduced by Boardman and Vogt~\cite{BoardmanVogt1973} in the context of topological operads (see also \cite{BM} for a general construction in any symmetric monoidal model category with a notion of interval).
Finally, we extend our resolution to the category of 1-reduced $\Lambda$-operads equipped with the Reedy model category structure.

\paragraph{The leveled Boardman--Vogt resolution for 1-reduced operads}

Let $\calO$ be a 1-reduced operad in spectra. Recall the categories of trees from Section~\ref{sec:invent-categ-trees} and the interval $\Delta[1]_{+}$ from Section \ref{SectModSpOp}.
In the constructions below, the symbols ``$W$'' emphasize the fact that these functors are used to define the Boardman--Vogt resolution. We consider the following two functors:
\begin{align*}
  \overline{\calO}_{W} : \lTree[n] & \longrightarrow \Sp, \qquad T \longmapsto \bigwedge_{v \in V(T)} \calO(|v|);            \\
  H_{W} : \lTree[n]^{op}           & \longrightarrow \sSet, \qquad T \longmapsto \bigwedge_{1 \le i \le h(T)} \Delta[1]_{+}.
\end{align*}

The functor $\overline{\calO}_{W}$ is defined using the operadic structure of $\calO$, the symmetric monoidal structure of spectra, and the unit of the operad $\calO$.
On permutation maps, the functor $H_{W}$ consists in permuting the parameters indexing the levels.
On contraction maps $\delta_{\{i+1\}}:T\rightarrow T/\{i+1\}$ (with $i\in \{0,\ldots,h(T)-1\}$), there are two cases to consider:
\begin{enumerate}
  \item If the levels $i$ and $i+1$ are permutable, then, by using the diagonal map, one has:
        \begin{align*}
          H_{W}(\delta_{\{i+1\}}): H_{W}(T/\{i+1\}) & \longrightarrow H_{W}(T),                                \\
          (t_{1},\dots,t_{h(T)-1})                  & \longmapsto (t_{1},\dots,t_{i},t_{i},\dots, t_{h(T)-1}).
        \end{align*}
  \item If the levels $i$ and $i+1$ are not permutable, then one has instead:
        \begin{align*}
          H_{W}(\delta_{\{i+1\}}): H_{W}(T/\{i\}) & \longrightarrow H_{W}(T),                                    \\
          (t_{1},\dots,t_{h(T)-1})                & \longmapsto (t_{1},\dots,t_{i},0,t_{i+1},\dots, t_{h(T)-1}).
        \end{align*}
\end{enumerate}

\begin{defi}
  The leveled Boardman--Vogt resolution $W_{l}\calO$ is defined in arity $n > 0$ as the coend:
  \[ W_{l}\calO(n) \coloneqq \int^{T\in \lTree[n]} \overline{\calO}_{W}(T) \wedge H_{W}(T). \]
\end{defi}

Roughly speaking, a point of $W_{l}\calO(n)$ is given by a leveled $n$-tree $T$, whose vertices are decorated by points in the operad $\calO$, and whose levels different from $0$ are decorated by elements in $\Delta[1]_{+}$.
Furthermore, we can contract two consecutive levels $i$ and $i-1$ if either the two levels are permutable and they are decorated by the same parameter, or they are not permutable the $i$-th level is decorated by $0$.
Such a point is denoted by $[T\,;\,\{\theta_{v}\}\,;\, \{t_{i}\}]$ where $T$ is a leveled tree, $\{\theta_{v}\}$, with $v\in V(T)$, is the family of points in the operad labelling the vertices and $\{t_{i}\}$, with $1\leq i \leq h(T)$, is the family of real numbers indexing the levels. See Figure~\ref{fig:equivalence-wo} for an example.

\begin{figure}[htbp]
  \hspace{-40pt} \includegraphics[scale=0.37]{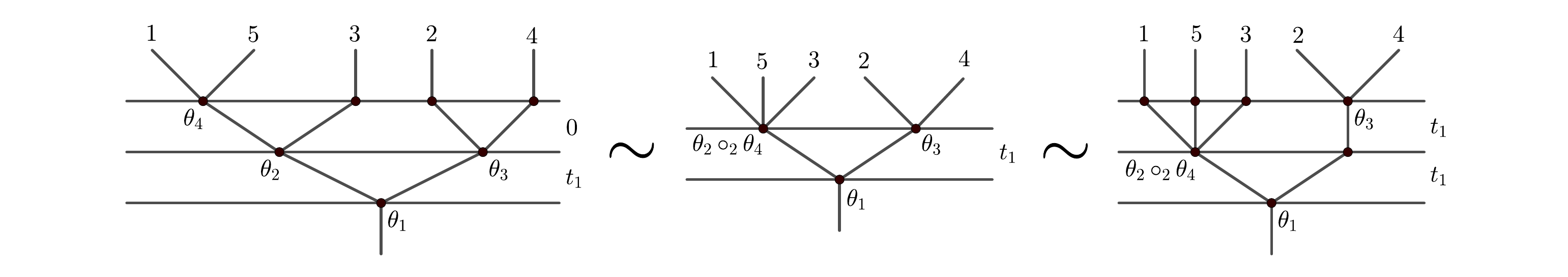}\vspace{-15pt}
  \caption{Illustration of equivalent points in $W_{l}\calO(5)$.}
  \label{fig:equivalence-wo}
\end{figure}

\begin{pro}
  There is an operadic structure on $W_{l}\calO$ defined using the operation $\gamma$ from Equation~\eqref{formulatreecomp} and decorating the new levels by $1:\ast\rightarrow \Delta[1]$.
\end{pro}

\begin{proof}
  While the operation $\gamma$ is not strictly associative, it is associative up to permutation. Thanks to the definition of $W_{l}\calO$ as a coend, it is invariant under permutation of levels. Hence the composition product on $W_{l}\calO$ is strictly associative and defines an operad structure.
\end{proof}

\begin{defi}[The usual Boardman--Vogt resolution for operads]\label{def:usual-W}
  We recall the usual Boardman--Vogt resolution $W\calO$ for any 1-reduced operad in spectra~\cite{BM}.
  For a 1-reduced operad $\calO$, we introduce two functors:
  \begin{align*}
    \overline{\calO}'_{W} : \Tree^{\geq 2}[n] & \longrightarrow \Sp, \qquad T  \longmapsto \bigwedge_{v \in V(T)} \calO(|v|);             \\
    H'_{W} : \Tree^{\geq 2}[n]^{op}           & \longrightarrow \sSet  , \qquad  T  \longmapsto \bigwedge_{e\in E^{in}(T)} \Delta[1]_{+}.
  \end{align*}

  On morphisms, the first functor is is obtained using the operadic structure of $\calO$, the unit in arity $1$ as well as the symmetric monoidal structure on spectra. The usual Boardman--Vogt resolution of $\calO$ is defined as the simplicial spectrum given by the coend:
  $$
    W\calO(n) \coloneqq \int^{T\in \Tree^{\geq 2}[n]}\, \overline{\calO}'_{W}(T) \wedge H'_{W}(T).
  $$
  A point is denoted by $[T\,;\,\{\theta_{v}\}\,;\,\{t_{e}\}]$ where $T\in \Tree^{\geq 2}[n]$ is a planar tree, $\{\theta_{v}\}_{v\in V(T)}$ is a family of points in $\calO$ labelling the vertices and $\{t_{v}\}_{v\in v\in V(T)}$ is a family of elements in $\Delta[1]$ indexing the inner edges. The equivalence relation coming from the coend is generated by contracting inner edges decorated by $0$ and by the compatibility with the action of the symmetric group.
\end{defi}

\begin{pro}\label{ProWWu}
  The usual and leveled Boardman--Vogt resolutions $W\calO$ and $W_{l}\calO$ are isomorphic.
\end{pro}

\begin{proof}
  We build an explicit isomorphism between the two constructions.  According to this notation introduced at the end of  Section \ref{CompTree}, there is an operadic map:
  \begin{align*}
    f : W_{l}\calO(n) & \longrightarrow  W\calO(n),
                      & \bigl[ T; \{\theta_{v}\}_{v\in V(T)}; \{t_{i}\}_{1\leq i\leq h(T)} \bigr] & \longmapsto \bigl[ \alpha(T); \{\theta_{v}\}_{v\in V(T)}; \{t'_{e}\}_{e\in E^{in}(\alpha(T))} \bigr],
  \end{align*}
  where $t'_{e}$ is the maximum of the parameters corresponding to the levels related to the path joining the source vertex of $e$ to its first non-bivalent vertex according to the orientation toward the root. Conversely, one has the continuous map
  \begin{align*}
    g : W\calO(n) & \longrightarrow W_{l}\calO(n),
                  & \bigl[ T; \{\theta_{v}\}_{v\in V(T)}; \{t_{e}\}_{e\in E^{in}(T)} \bigr]
                  & \longmapsto \bigl[ \beta(T); \{\theta_{v}\}_{v\in V(T)}; \{t'_{i}\}_{1\leq i\leq h(\beta(T))} \bigr],
  \end{align*}
  where $t'_{i}=t_{e}$ if the unique non-bivalent vertex on the $i$-th level is the source vertex of $e$. The reader can easily check that the maps so obtained are well defined, compatible with the operadic structures and give rise to an isomorphism between the leveled and usual Boardman--Vogt resolutions.
\end{proof}

\begin{cor}
  Let $\calO$ be a $\Sigma$-cofibrant $1$-reduced operad in spectra. The map $\mu:W_{l}\calO\rightarrow \calO$, sending the parameters indexing the levels to $0$, is a weak equivalence of operads. The operad $W_{l}\calO$ is a cofibrant resolution of $\calO$ in the category of 1-reduced operads equipped with the projective model category structure.
\end{cor}

\begin{proof}
  The second part of the statement follows from the results of~\cite{BM}.
  The two authors show that if $\calO$ is a well pointed (i.e.\ $* \to \calO(1)$ is a cofibration) and $\Sigma$-cofibrant operad, then the usual Boardman--Vogt resolution is cofibrant replacement of $\calO$ in the projective model category of operad.
  Furthermore, they prove that the map $\mu:W\calO\rightarrow \calO$, sending the parameters indexing the inner edges to $0$, is a weak equivalence of operads.
  By using the isomorphism introduced in Proposition~\ref{ProWWu}, the same is true for the leveled Boardman--Vogt resolution.
\end{proof}

\paragraph{The leveled Boardman--Vogt resolution for 1-reduced $\Lambda$-operads}

Let $\calO$ be a 1-reduced $\Lambda$-operad in spectra.
In order to get a cofibrant resolution of $\calO$ in the Reedy model category $\Lambda\Operad$, we provide a $\Lambda$-structure to the construction introduced in Section \ref{SectBVop}.
As a symmetric sequence, we set
$$
  W_{\Lambda}\calO(n) \coloneqq W_{l}\calO_{>0}(n), \hspace{15pt} \text{for all } n>0,
$$
where $\calO_{>0}$ is the underlying 1-reduced operad of $\calO$. The subscript $\Lambda$ is to emphasize that we work in the category of 1-reduced $\Lambda$-operads. By restriction, $W_{\Lambda}\calO$ inherits operadic compositions
$$
  \gamma:W_{\Lambda}\calO(k)\wedge W_{\Lambda}\calO(n_{1})\wedge\cdots\wedge  W_{\Lambda}\calO(n_{k})\longrightarrow W_{\Lambda}\calO(n_{1}+\cdots+n_{k}).
$$
The $\Lambda$-structure in $W_{\Lambda}\calO$ is defined in the obvious way using the $\Lambda$-structure on the (first non-bivalent) vertex connected to the leaf labeled by $i$. If by doing so, the new point so obtained has a level which consists of bivalent vertices, then we remove it.

\begin{figure}[htbp]
  \hspace{-20pt}\includegraphics[scale=0.4]{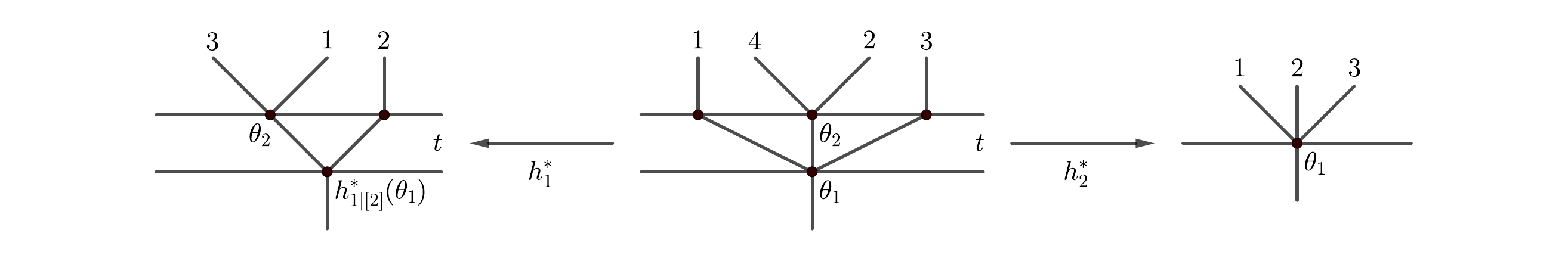}\vspace{-10pt}
  \caption{Illustrations of the $\Lambda$-structure associated to $h_{1},h_{2}:[3]\to [4]$ with $h_{1}(i)=i+1$ and $h_{2}(i)=i$.}
\end{figure}

\begin{pro}\label{prop:w-o-resolution}
  If $\calO$ is a $\Sigma$-cofibrant 1-reduced operad, then $W_{\Lambda}\calO$ is a cofibrant resolution of $\calO$ in the category of 1-reduced $\Lambda$-operads equipped with the Reedy model category structure.
  In particular, the map $\mu:W_{\Lambda}\calO\rightarrow \calO$, sending the parameters indexing the levels to $0$, is a weak equivalence of operads.
\end{pro}

\begin{proof}
  The map $0 : * \to \Delta^{1}$ is a weak equivalence which implies that the operadic map $\mu$ is a weak equivalence too.
  Moreover, we know from the results of~\cite[§8.5.5.2]{Fre} that a 1-reduced $\Lambda$-operad is Reedy cofibrant if and only if the corresponding 1-reduced operad is cofibrant in the projective model category.
  The 1-reduced operad associated to $W_{\Lambda}\calO$ is $W\calO_{>0}$ which is cofibrant in the projective model category.
\end{proof}

\subsection{The cooperad of indecomposable elements}

In the previous section, we built a cofibrant resolution $W_{l}\calO$ for any $\Sigma$-cofibrant 1-reduced operad $\calO$ in spectra. In what follows we show that the leveled bar construction of the operad $\calO$ can be expressed as the suspension of a cooperad $\Ind(W_{l}\calO)$. Unfortunately, we cannot extend this result to 1-reduced $\Lambda$-operads since the leveled bar construction of a 1-reduced $\Lambda$-operad is not necessarily a 1-reduced $\Lambda$-cooperad.

A point in $W_{l}\calO$ is said to be \textit{indecomposable} if no elements indexing the levels is equal to $1:\ast\rightarrow \Delta[1]$. The indecomposable cooperad $\Ind(W_{l}\calO)$ is obtained by identifying any decomposable element in $W_{l}\calO$ with the basepoint. In other words, if we modify slightly the functor $H$ as follows:
\begin{align*}
  H''_{W} : \lTree[n]^{op} & \longrightarrow \sSet,
  \\
  T                        & \longmapsto
  \begin{cases}
    \tilde{\Delta}[T]/ \tilde{\Delta}_{0}[T] & \text{if } n>1, \\
    \hspace{2em}\ast                         & \text{if } n=1,
  \end{cases}
\end{align*}
where $\tilde{\Delta}[T]=\prod_{1\leq i\leq h(T)} \Delta[1]$ and  $\tilde{\Delta}_{0}[T]$ is the simplicial subset consisting of faces where at least one of the levels has value $1$. By construction, $H''(T)$ is already a pointed simplicial set.

\begin{defi}
  the cooperad of indecomposable points is defined as simplicial spectrum given by the coend:
  $$
    \Ind(W_{l}\calO)(n) \coloneqq \int^{T\in \lTree[n]}\overline{\calO}_{W}(T)\wedge H''_{W}(T).
  $$
\end{defi}

For any partition $n = n_{1} + \dots + n_{k}$, the cooperadic operation
$$
  \gamma^{c}:\Ind(W_{l}\calO)(n_{1}+\cdots+n_{k})\longrightarrow \Ind(W_{l}\calO)(k)\wedge \Ind(W_{l}\calO)(n_{1})\wedge \cdots \wedge \Ind(W_{l}\calO)(n_{k}),
$$
is defined as follows. Consider an element $[T\,;\,\{x_{v}\}\,;\,\{t_{i}\}]$. If $T$, up to permutations of permutable levels and contractions of permutable levels, is not of the form $\gamma(T_{0}\,;\,\{T_{i}\})$ with $T_{0}\in \lTree[k]$ and $T_{i}\in \lTree[n_{i}]$ then $\gamma^{c}([T\,;\,\{x_{v}\}\,;\,\{t_{i}\}])$ is the basepoint. Otherwise, the element is sent to the family
$$
  [T_{0}\,;\,\{x_{v}^{0}\}\,;\,\{t_{j}^{0}\}]\,\,\,\,;\,\,\,\,\big\{ \, [T_{i}\,;\,\{x_{v}^{i}\}\,;\,\{t_{j}^{i}\}]\,\big\}_{i\in I}\in \Ind(W_{l}\calO)(k)\wedge \bigwedge_{1\leq i \leq k} \Ind(W_{l}\calO)(n_{i}),
$$
where the parameters indexing the vertices and the levels of the leveled trees $T_{0}$ and $T_{i}$ are induced by the parameters indexing the leveled tree $T$. This structure is similar to the cooperadic structure introduced on the leveled bar construction introduced in Section \ref{SectBarOp}. Actually, one has the following connection between the Boardman--Vogt resolution and the bar construction:

\begin{pro}\label{ProOpFr}
  The leveled bar construction of the operad $\calO$ is isomorphic to the suspension of the cooperad of indecomposable elements:
  $$
    \mathcal{B}_{l}\calO\cong \Sigma \Ind(W_{l}\calO).
  $$
\end{pro}

\begin{proof}
  Taking the indecomposables of $W_{l}\calO$ identifies to the base point all points whose underlying tree has a level of length 1. The suspension coordinate gives us a length for the $0$-th level in the bar construction. To complete the proof, we recall quickly the cooperadic structure on  $\Sigma \Ind(W_{l}\calO))$. We denote by $[T\,;\,\{\theta_{v}\}\,;\,\{t_{j}\}\,;\,x]$ a point in $\Sigma \Ind(W_{l}\calO))$ where $x$ is the suspension coordinate. The cooperadic composition is defined as follows:
  $$
    \begin{aligned}
      \Sigma \Ind(W_{l}\calO))(n_{1}+\cdots+n_{k})           & \longrightarrow \Sigma \Ind(W_{l}\calO))(k)\wedge \Sigma \Ind(W_{l}\calO))(n_{1})\wedge\cdots \wedge \Sigma \Ind(W_{l}\calO))(n_{k}), \\
      \left[T\,;\,\{\theta_{v}\}\,;\,\{t_{j}\}\,;\,x \right] & \longmapsto
      \begin{cases}
        \{\,[T_{i}\,;\,\{\theta_{v}^{i}\}\,;\,\{t_{j}^{i}\}\,;\,x_{i}]\,\}_{i\in I\sqcup \{0\}} & \text{if } T\sim \gamma(T_{0}\,,\,\{T_{i}\}) \text{ is decomposable}, \\
        \hspace{4em}\ast                                                                        & \text{otherwise},
      \end{cases}
    \end{aligned}
  $$
  where $x_{0}=x$ and $x_{i}$, with $i\in I$, is the element indexing the level in $\gamma(T_{0}\,,\,\{T_{i}\})$ corresponding to the root of $T_{i}$. The reader can easily check that the structure so obtained is well defined and compatible with the isomorphism.
\end{proof}

\subsection{The Boardman--Vogt resolution and the cobar-bar construction}\label{SecBVtoCobBar}

In what follows, we adapt the definition of the cobar construction for 1-reduced cooperads in spectra from \cite{Chi}, but using the notion of leveled trees instead of planar trees.
Then we show that this construction is isomorphic to the usual one. After that, we prove that the leveled Boardman--Vogt resolution of a 1-reduced operad $\calO$ in spectra is weakly equivalent to its leveled cobar-bar construction.

\paragraph{The leveled cobar construction for 1-reduced cooperads in spectra}

From \cite{Chi2}, we recall that the simplicial indexing category $\Delta$ has an automorphism $\calR$ that sends a totally ordered set to the same set with the opposite order. For a simplicial set $X$, the reverse of $X$, denoted by $X^{rev}$, is the simplicial set $X\circ \calR$. Let $\calC$ be a 1-reduced cooperad in spectra. We introduce the functor
\begin{equation*}
  \overline{\calC}_{\Omega} : \lTree[n]^{op} \longrightarrow \Sp, \qquad T \longmapsto \bigwedge_{v \in V(T)} \calC(|v|),
\end{equation*}
defined on morphisms using the cooperadic structure of $\calC$.

\begin{defi}
  The leveled cobar construction associated to a 1-reduced cooperad $\calC$ in spectra is the end
  \begin{equation}\label{CobarCoop}
    \Omega_{l}\calC(n) \coloneqq \int_{T\in \lTree[n]}\Map\big(\, H_{B}(T)^{rev}; \overline{\calC}_{\Omega}(T)\,\big),
  \end{equation}
  where $H_{B}$ is the functor given by the formula \eqref{EqFunctOH}.
  By $\Map(-;-)$ we understand the cotensoring of $\Sp$ over pointed simplicial sets. Concretely, a point in $\Omega_{l}\calC(n)$ is a family of maps $\Phi=\{\Phi_{T}:H_{B}(T)\rightarrow \overline{\calC}_{\Omega}(T), \, T\in \lTree[n]\}$ satisfying the following relations: for each permutation $\sigma$ and each contraction morphism $\delta_{N}$, one has the commutative diagrams
  \begin{equation}\label{RelCobar}
    \begin{tikzcd}[column sep = large]
      H_{B}(T\cdot\sigma) \ar[r, "H(\sigma)"] \ar[d, "\Phi_{T\cdot\sigma}"] & H_{B}(T) \ar[d, "\Phi_{T}"]
      & H_{B}(\delta_{N}(T)) \ar[r, "H(\delta_{N})"] \ar[d, "\Phi_{\delta_{N}(T)}"] & H_{B}(T) \ar[d, "\Phi_{T}"] \\
      \overline{\calC}_{\Omega}(T\cdot \sigma) \ar[r, "\overline{\calC}(\sigma)"] & \overline{\calC}_{\Omega}(T)
      & \overline{\calC}_{\Omega}(\delta_{N}(T)) \ar[r, "\overline{\calC}(\delta_{N})"] & \overline{\calC}_{\Omega}(T)
    \end{tikzcd}
  \end{equation}

  The sequence $\Omega_{l}\calC=\{\Omega_{l}\calC(n)\}$ forms an operad in spectra whose operadic composition
  \begin{align*}
    \gamma: \Omega_{l}\calC(k) \wedge \Omega\calC(n_{1})\wedge \cdots \wedge \Omega\calC(n_{k}) & \longrightarrow  \Omega_{l}\calC(n_{1}+\cdots +n_{k}),                                                                                                        \\
    \Phi_{0}\,;\,\{\Phi_{i}\}                                                                   & \longmapsto \gamma_{\Omega}(\Phi_{0}\,;\,\{\Phi_{i}\})=\bigl\{\,\gamma_{\Omega}(\Phi_{0}\,;\,\{\Phi_{i}\})_{T}, \, T\in \lTree[n_{1}+\cdots +n_{k}]\,\bigr\},
  \end{align*}
  is defined as follows. If, up to permutations  of permutable levels and contractions of permutable levels, the leveled tree $T$ is not of the form $\gamma(T_{0}\,;\,\{T_{i}\})$, with $T_{0}\in \lTree[k]$ and $T_{i}\in \lTree[n_{i}]$, then $\gamma_{\Omega}(\Phi_{0}\,;\,\{\Phi_{i}\})_{T}$ sends any decoration of the levels to the basepoint. Otherwise, we define $\gamma_{\Omega}(\Phi_{0}\,;\,\{\Phi_{i}\})_{T}$ to be the composition
  $$
    \gamma_{\Omega}(\Phi_{0}\,;\,\{\Phi_{i}\})_{T}:H_{B}(T)\longrightarrow H_{B}(\gamma(T_{0}\,;\,\{T_{i}\})) \cong H(T_{0}) \wedge \bigwedge_{1\leq i \leq  k} H(T_{i}) \longrightarrow \overline{\calC}_{\Omega}(T_{0}) \wedge \bigwedge_{1\leq i \leq  k} \overline{\calC}_{\Omega}(T_{i}) \cong \overline{\calC}_{\Omega}(T).
  $$
\end{defi}

\begin{defi}[The usual cobar construction for $1$-reduced cooperads]\label{def:usual-cobar}
  For more details, we refer the reader to \cite{Chi2}.  We recall that $\Tree^{\geq 2}[n]$ is the category of planar $n$-trees having vertices with valences $\geq 2$ and whose morphisms are generated by isomorphisms of planar trees and contractions of inner edges. Given a 1-reduced cooperad $\calC$, we introduce the two functors
  $$
    \begin{aligned}
      \overline{\calC}'_{\Omega} : \Tree^{\geq 2}[n] & \longrightarrow \Sp \; ,
                                                     & T                        & \longmapsto  \underset{v \in V(T)}{\bigwedge} \calC(|v|);
    \end{aligned}
  $$
  defined on morphisms using the cooperadic structure of $\calC$. The usual cobar construction of $\calO$ is defined as the end
  $$
    \Omega(\calC)(n) \coloneqq \int_{T\in \Tree^{\geq 2}[n]}\, \overline{\calC}'_{\Omega}(T)\wedge H'_{B}(T),
  $$
  where $H'_{B}$ is given by the formula \eqref{H'}. A point in $\Omega\calC(n)$ is a family of maps $\Phi=\{\Phi_{T}:H_{B}(T)\rightarrow \overline{\calC}_{\Omega}(T), \, T\in \Tree^{\geq 2}[n]\}$ some relations induced by the end.
\end{defi}

\begin{pro}\label{pro:iso-omega-usual}
  The leveled cobar construction is isomorphic to the usual cobar construction:
  $$
    \Omega_{l}\calC\cong \Omega\calC.
  $$
\end{pro}

\begin{proof}
  As a consequence of the comparison morphisms between planar $n$-trees and leveled $n$-trees introduced in Section \ref{CompTree}, one can build an explicit isomorphism
  $$
    L_{n}:\Omega\calC(n)\leftrightarrows \Omega_{l}\calC(n):R_{n}.
  $$
  Let $\Phi$ be an element in $\Omega\calC(n)$ and $T$ be a leveled $n$-tree. Then, $L_{n}(\Phi)_{T}$, the map associated to the leveled tree $T$, is given by $\Phi_{\alpha (T)}$.
  Conversely, let $\Phi'$ be an element in $\Omega_{l}\calC(n)$ and $T'$ be a rooted planar tree.
  Then, $R_{n}(\Phi')_{T}$, the map associated to $T'$, is given by $\Phi'_{\beta (T')}$.
  The map $R_{n}$ does not depend on the fix point $T_{l}\in \alpha^{-1}(T)$ due to the relations \eqref{RelCobar}.
  So, the maps $L_{n}$ and $R_{n}$ are well defined and provide an isomorphism preserving the cooperadic structures.
\end{proof}

\paragraph{Connection between Boardman--Vogt resolutions and cobar-bar constructions}

Let $\calO$ be a 1-reduced operad in spectra.
In Section \ref{SectBVop}, we built a cofibrant resolution of $\calO$ through the leveled Boardman--Vogt resolution $W_{l}\calO$. In Section \ref{SectBarOp}, we introduced a leveled version of the bar construction, denoted $\calB_{l}(\calO)$, which is isomorphic to the usual bar construction.
According to our definition of the leveled cobar construction in the previous section, we apply the strategy used by \cite{Chi} in order to build a map
$$
  \Gamma_{S}:W_{l}\calO(n)\longrightarrow \Omega_{l}\calB_{l}(\calO)(n).
$$

A point in the cobar-bar construction $\Omega_{l}\calB_{l}(\calO)(n)$ is the data of a family of maps $\Phi=\{\Phi_{T}:H_{B}(T)\rightarrow \overline{\calB_{l}(\calO)}_{\Omega}(T),\,\,T\in \lTree[n]\}$ satisfying the relations \eqref{RelCobar}. A point in $\overline{\calB_{l}(\calO)}_{\Omega}(T)$ is a family of elements in $\calB_{l}(\calO)$ indexed the vertices of the leveled tree $T$.

\begin{notat}
  Let $T_{1}$ and $T_{2}$ be two leveled $n$-trees. We say that $T_{1}\geq T_{2}$ if, up to permutations of permutable levels and contractions of permutable levels, $T_{2}$ can be obtained from $T_{1}$ by contracting levels. Given two such trees $T_{1}\geq T_{2}$, we fix the following notation.

  To each vertex $v\in V(T_{2})$ we associate a leveled sub-tree $T_{1}[v]$ of the leveled tree $T_{1}$ in such a way that $T_{1}$ is obtained (up to permutations and contractions of permutable levels) by grafting all the trees $T_{1}[v]$ together.
  For instance, from the two leveled trees $T_{1}\geq T_{2}$:

  \hspace{-30pt}\includegraphics[scale=0.32]{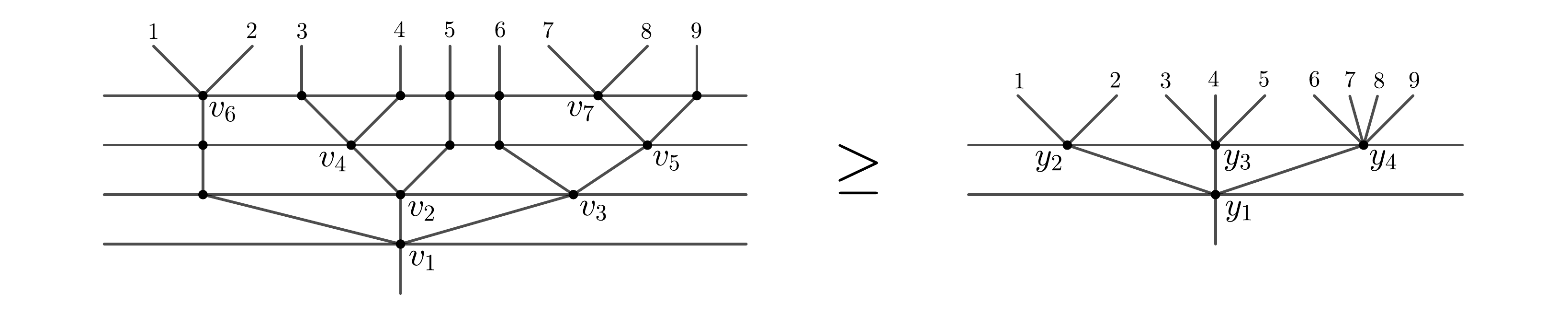}

  \noindent the sub-leveled trees associated to the vertices $y_{1},\ldots,y_{4}$ are the following ones:

  \hspace{-40pt}\includegraphics[scale=0.33]{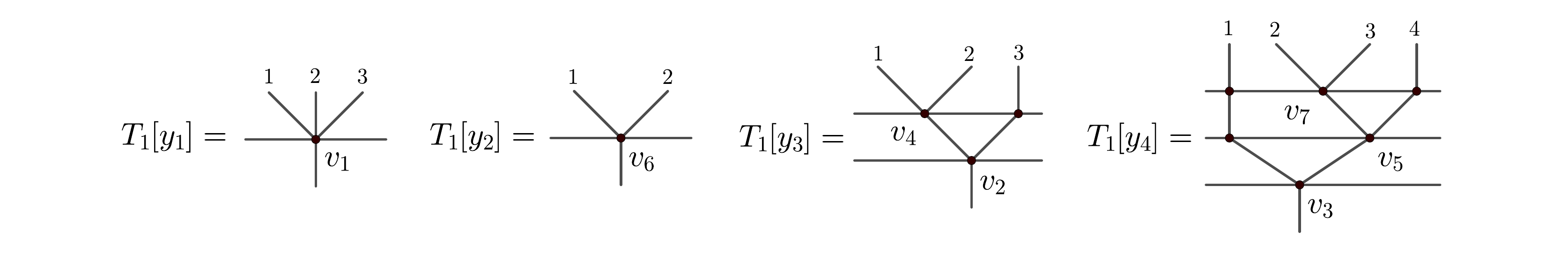}

  For any vertex $v\in V(T_{2})$, we denote by $or_{v}:V(T_{1}[v])\rightarrow V(T_{1})$ the map assigning to a vertex in $T_{1}[v]$ the corresponding vertex in $T_{1}$. Similarly, let $ol_{v}:\{0,\ldots,h(T_{1}[v])\}\rightarrow \{0,\ldots,h(T_{1})\}$ be the map assigning to a level in $T_{1}[v]$ the corresponding level in $T_{1}$. For instance, in the above example one has $ol_{y_{2}}(0)=3=ol_{y_{4}}(2)$ and $ol_{y_{4}}(1)=2$.
\end{notat}

Finally, the map between the leveled Boardman--Vogt resolution and the leveled cobar-bar construction
\begin{equation}\label{eq:BWtoCobarbar}
  \begin{aligned}
    \Gamma_{n} : W_{l}\calO(n)                 & \longrightarrow \Omega_{l}\calB_{l}(\calO)(n),                                 \\
    x=[T_{1}\,;\,\{\theta_{v}\}\,;\,\{t_{l}\}] & \longmapsto \Phi_{x}=\big\{ \Phi_{x;T_{2}}\,\,,\,\,T_{2}\in \lTree[n]\,\big\},
  \end{aligned}
\end{equation}
is defined as follows.
If $T_{1}\ngeq T_{2}$, then $\Phi_{x;T_{2}}$ is the basepoint in $\overline{\calB_{l}(\calO)}_{\Omega}(T_{2})$.
Otherwise, one has
\begin{align}
  \Phi_{x;T_{2}} : H_{B}(T)               & \longrightarrow \overline{\calB_{l}(\calO)}_{\Omega}(T_{2}),                                     \\
  \{\tilde{t}_{i}\}_{0\leq i\leq h(T_{2}} & \longmapsto \bigl\{\,[T_{1}[v]\,;\,\{\theta_{y}[v]\}\,;\,\{t_{j}[v]\}]\,\bigr\}_{v\in V(T_{2})},
\end{align}
where
$$
  \theta_{y}[v]=\theta_{or_{v}(y)}
  \hspace{15pt} \text{and} \hspace{15pt}
  t_{j}[v]
  =
  \begin{cases}
    t_{ol_{v}(j)}                                               & \text{if } j>0,                                                  \\
    1-t_{ol_{v}(j)}                                             & \text{if } j=0 \text{ and } v \text{ is the root of } T_{2},     \\
    \max\big(\, 0\,;\, t_{ol_{v}(j)}-\tilde{t}_{\lev(v)}\,\big) & \text{if } j=0 \text{ and } v \text{ is not the root of } T_{2}.
  \end{cases}
$$

\begin{pro}\label{prop:bw-cobarbar}
  The map \eqref{eq:BWtoCobarbar} induces a weak equivalence of 1-reduced operads.
\end{pro}

\begin{proof}
  It is a direct consequence of \cite[Theorem 2.15]{Chi2} and the fact that the leveled Boardman--Vogt resolution as well as the leveled cobar-bar construction are both isomorphic to the usual constructions.
\end{proof}

\begin{rmk}\label{rmk:lambda-fresse}
  As far as we know, the bar construction of a $\Lambda$-operad does not inherit a $\Lambda$-cooperad structure, which prevents us from extending our results to $\Lambda$-operads -- one would need a $B\Lambda$ structure on $\mathcal{P}$ for $B\mathcal{P}$ to be a $\Lambda$-cooperad~\cite[Proposition~2.4]{FresseTurchinWillwacher2017}.

  Fresse showed in the algebraic setting that the cobar-bar resolution of a dg-$\Lambda$-operad inherits a $\Lambda$-structure~\cite[Proposition~C.2.18]{Fre}.
  His constructions are in some sense dual to ours.
  While our bar construction is defined by a coend and our cobar construction by an end (similarly to Ching's work~\cite{Chi2}), in Fresse's work the bar construction is an end and the cobar construction is a coend.
  Fresse's result is thus more closely related to our results on cooperads and cobimodules (see Sections~\ref{sec:fibr-resol-hopf} and~\ref{sec:fibr-resol-hopf-2}).
\end{rmk}

\section{Cofibrant resolutions for $\Lambda$-bimodules in spectra}
\label{sec:cofibr-resol-lambda-1}

Let $\calP$ and $\calQ$ be two 1-reduced operads and $M$ be a ($\calP$-$\calQ$)-bimodule in spectra. The aim of this section is to introduce a kind of Boardman--Vogt $W_{l}$ resolution for any ($W\calP$-$W\calQ$)-bimodule $M$ and to prove that the leveled two-sided bar construction of $M$ can be expressed as the suspension of the ($\Ind(W\calP)$-$\Ind(W\calQ)$)-cobimodule of indecomposable elements $\Ind(W_{l}M)$. Similarly to the operadic case, we also prove that the Boardman--Vogt resolution is weakly equivalent to the leveled two-sided cobar-bar construction.

\subsection{The two-sided leveled bar construction in spectra}\label{SectBarBi}

Given an operad $\calP$, a right $\calP$-module $M$, and a left $\calP$-module $N$, recall that the two-sided bar construction $\calB(M,\calP,N)$ is obtained as the realization of the simplicial object $M \circ \calP^{\circ \bullet} \circ N$, where faces and degeneracies are defined using the operad/module structure maps.
In particular, $\calB(M,\calP,\calP)$ (resp.\ $\calB(\calP,\calP,N)$) is a cofibrant resolution of the right module $M$ (resp.\ the left module $N$).

Now, if $\calP$ and $\calQ$ are operads and $M$ is a $(\calP,\calQ)$-bimodule, we can thus define a cofibrant resolution of $M$ as the pullback:
\begin{equation}\label{eq:usual-two-sided}
  \calB[\calP,\calQ](M) \coloneqq \calB(\calP,\calP,M) \circ_{M} \calB(M,\calQ,\calQ) = \bigl| \calP^{\circ(1+\bullet)} \circ M \circ \calQ^{\circ(\bullet+1)} \bigr|.
\end{equation}
Unfortunately, this simplicial resolution does not define a cobimodule: there is no way to define cobimodule structure maps that strictly satisfy associativity, because of the total composition (just like simplicial bar construction $\calP^{\circ(\bullet+1)}$ is not a cooperad).
To solve this problem, we introduce an alternative version of this construction using our notion of leveled trees with section which is naturally endowed with a structure of cobimodule.

Let $\calP$ and $\calQ$ be two 1-reduced operads in spectra.
From a ($\calP$-$\calQ$)-bimodule $M$, we define the following two functors:
\begin{align*}
  \overline{M}_{B} : \slTree[n] & \longrightarrow \Sp,
                                & (T,\iota)              & \longmapsto \bigwedge_{v \in V_{d}(T)} \calP(|v|) \wedge \bigwedge_{v \in V_{\iota}(T)} M(|v|) \wedge \bigwedge_{v \in V_{u}(T)} \calQ(|v|); \\
  sH_{B} : \slTree[n]^{op}      & \longrightarrow \sSet,
                                & (T\,,\,\iota)          & \longmapsto
  \begin{cases}
    \Delta[T]/\Delta_{0}[T\,,\,\iota] & \text{if } n>1, \\
    \hspace{2em}\ast                  & \text{if } n=1,
  \end{cases}
\end{align*}
where $\Delta[T]=\prod_{0\leq i \leq h(T)}\, \Delta[1]$ labels the levels by elements in the standard $1$-simplex $\Delta[1]$ while $\Delta_{0}[T\,,\,\iota]$ is the simplicial subset consisting of faces where, either the $\iota$-th level has value $0$, or any of the other levels has value $1$. By definition, $H(T, \iota)$ is a pointed simplicial set for any leveled tree with section $(T,\iota)$, whose basepoint is the equivalence class of $\Delta_{0}[T,\iota]$.

On morphisms, the functor $\overline{M}_{B}$ is defined using the bimodule structure of $M$.
For any two consecutive permutable levels $i$ and $i+1$, $sH_{B}(\sigma_{i})$ permutes the simplices corresponding to the $i$-th and $(i+1)$-st levels.
For contraction morphisms there are three cases to consider:
\begin{enumerate}
  \item If the levels $i$ and $i+1$ are permutable, then, by using the diagonal map:
        \begin{align*}
          sH_{B}(\delta_{\{i+1\}}): sH_{B}(T/\{i+1\}) & \longrightarrow sH_{B}(T),                               \\
          (t_{0},\dots,t_{h(T)-1})                    & \longmapsto (t_{0},\dots,t_{i},t_{i},\dots, t_{h(T)-1}).
        \end{align*}
  \item If the levels $i$ and $i+1$ are not permutable and $i \ge \iota$, then:
        \begin{align*}
          sH_{B}(\delta_{\{i+1\}}): sH_{B}(T/\{i+1\}) & \longrightarrow sH_{B}(T),                                   \\
          (t_{0},\dots,t_{h(T)-1})                    & \longmapsto (t_{0},\dots,t_{i},0,t_{i+1},\dots, t_{h(T)-1}).
        \end{align*}
  \item If the levels $i$ and $i+1$ are not permutable and $i < \iota$, then:
        \begin{align*}
          sH_{B}(\delta_{\{i+1\}}): sH_{B}(T/\{i+1\}) & \longrightarrow sH_{B}(T),                                   \\
          (t_{0},\dots,t_{h(T)-1})                    & \longmapsto (t_{0},\dots,t_{i-1},0,t_{i},\dots, t_{h(T)-1}).
        \end{align*}
\end{enumerate}

\begin{defi}
  The leveled two-sided bar construction is defined as the coend:
  $$\calB_{l}[\calP,\calQ](M)(n) \coloneqq \int^{T\in \slTree[n]} \overline{M}_{B}(T)\wedge sH_{B}(T).$$
\end{defi}

A point in $\calB_{l}(\calP,M,\calQ)(n)$ is the data of a leveled $n$-tree with section $T=(T\,,\,\iota)$, a family of points $\{\theta_{v}\}_{v\in V(T)}$ labelling the vertices on the main section (resp. below and above the main section) by points in $M$ (resp. points in $\calP$ and $\calQ$) and a family of elements in the simplicial set $\Delta[1]$ indexing the levels $\{t_{j}\}_{0\leq j \leq h(T)}$.
The equivalence relation induced by the coend is generated by the compatibility with the symmetric group action, permutations of permutable levels, contractions of two consecutive permutable levels indexed by the same simplex, and contractions of non-permutable levels such that the upper or lower level (depending on whether we are above or below the section) is indexed by $0$.
If there is no ambiguity with the operadic case, such a point is denoted by $[T\,;\, \{\theta_{v}\}\,;\,\{t_{j}\}]$.

\begin{figure}[htbp]
  \center \includegraphics[scale=0.3]{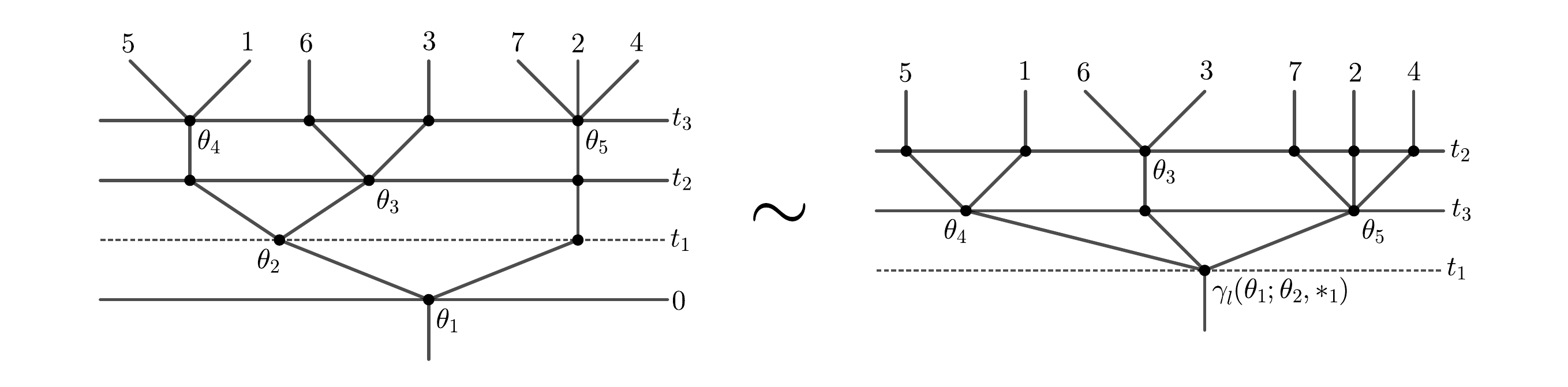}\vspace{-10pt}
  \caption{Illustration of equivalent points in $\calB_{l}[\calP,\calQ](M)(7)$.}
\end{figure}

The sequence $\calB_{l}(\calP,M,\calQ)=\{\calB_{l}(\calP,M,\calQ)(n)\}$ inherits a ($\calB_{l}(\calP)$-$\calB_{l}(\calQ)$)-cobimodule structure, with right and left module maps denoted by:
\begin{equation}\label{comodstructop}
  \begin{aligned}
    \gamma^{c}_{R}:\calB_{l}[\calP,\calQ](M)(n_{1}+\cdots+n_{k}) & \longrightarrow \calB_{l}[\calP,\calQ](M)(k) \wedge \calB_{l}(\calQ)(n_{1})\wedge \cdots \wedge \calB_{l}(\calQ)(n_{k});         \\
    \gamma^{c}_{L}:\calB_{l}[\calP,\calQ](M)(n_{1}+\cdots+n_{k}) & \longrightarrow \calB_{l}(\calP)(k) \wedge  \calB_{l}[\calP,\calQ](M)(n_{1})\wedge\cdots\wedge \calB_{l}[\calP,\calQ](M)(n_{k}).
  \end{aligned}
\end{equation}

A leveled $n$-tree with section $T$ is said to be \emph{right decomposable} according to the partition $(n_{1},\ldots,n_{k})$ if there exist a leveled tree with section $T_{0}\in \slTree[k]$ and leveled trees $T_{i}\in \lTree[n_{i}]$, with $i\leq k$, such that $T$ is of the form $\gamma_{R}(T_{0}\,,\,\{T_{i}\})$ up to permutations and contractions of permutable levels (where $\gamma_{R}$ is the operation \eqref{eq:def-gamma-r}).
According to this notation, if $T$ is not right decomposable, then $\gamma_{R}^{c}([T\,;\, \{\theta_{v}\}\,;\,\{t_{j}\}])$ is sent to the basepoint. Otherwise, let us remark that there is an identification
$$
  [T\,;\, \{\theta_{v}\}\,;\,\{t_{j}\}]= [\gamma_{R}(T_{0}\,,\,\{T_{i}\})\,;\, \{\theta_{v}\}\,;\,\{\tilde{t}_{j}\}]
$$
due to the equivalence relation induced by the coend.
In that case, we define:
$$
  \gamma_{R}^{c}([T\,;\, \{\theta_{v}\}\,;\,\{t_{j}\}]) \coloneqq \bigl\{ [T_{i}\,;\, \{\theta_{v}^{i}\}\,;\,\{t_{j}^{i}\}] \bigr\}_{0\leq i\leq k} \in \calB_{l}[\calP,\calQ](M)(k) \wedge \bigwedge_{1\leq i \leq k} \calB_{l}(\calQ)(n_{i})
$$
where $\{\theta_{v}^{i}\}$ and $\{t_{j}^{i}\}$ come from the parameters corresponding to the sub-tree $T_{i}$ of $\gamma_{R}(T_{0}\,,\,\{T_{i}\})$.

Similarly, a leveled $n$-tree with section $T$ is said to be \emph{left decomposable} according to the partition $(n_{1},\ldots,n_{k})$ if there exist a leveled tree  $T_{0}\in \lTree[k]$ and leveled trees with section $T_{i}\in \slTree[n_{i}]$, with $i\leq k$, such that $T$ is of the form $\gamma_{L}(T_{0}\,,\,\{T_{i}\})$ up to permutations and contractions of permutable levels (where $\gamma_{L}$ is the operation \eqref{eq:def-gamma-l}).
According to this notation, if $T$ is not left decomposable, then $\gamma_{L}^{c}([T\,;\, \{\theta_{v}\}\,;\,\{t_{j}\}])$ is sent to the basepoint. Otherwise, one has
$$
  \gamma_{L}^{c}([T\,;\, \{\theta_{v}\}\,;\,\{t_{j}\}])=\big\{ \,\, [T_{i}\,;\, \{\theta_{v}^{i}\}\,;\,\{t_{j}^{i}\}]\,\,\big\}_{0\leq i\leq k}\in \calB_{l}(\calP)(k) \wedge \underset{1\leq i\leq k}{\bigwedge} \calB_{l}(\calP,M,\calQ)(n_{i})
$$
where $\{\theta_{v}^{i}\}$ and $\{t_{j}^{i}\}$ come from the restriction to the parameters corresponding to the sub-tree $T_{i}$ of $\gamma_{L}(T_{0}\,,\,\{T_{i}\})$.
These operations do not depend on the choice of the decomposition of $T$ up to permutations and contractions of permutable levels thanks to the definition of the coend.

\subsection{The leveled Boardman--Vogt resolution for bimodules}\label{SectBVBi}

We split this subsection into two parts.
First, we construct the Boardman--Vogt resolution for bimodules, which is an isomorphic variant of the one in~\cite{Duc}.
After that, we extend this construction in order to get cofibrant resolutions for the Reedy model category of  $\Lambda$-bimodules over a pair of 1-reduced $\Lambda$-operads.

\paragraph{The leveled Boardman--Vogt resolution for bimodules}

Let $\calP$ and $\calQ$ be two 1-reduced operads and let $M$ be a $(\calP \text{-} \calQ)$-bimodule in spectra. Adapting the notation introduced in Section \ref{SectBVop}, we consider the following two functors:
\begin{align*}
  \overline{M}_{W} : \slTree[n] & \longrightarrow \Sp,
                                & (T,\iota)              & \longmapsto \bigwedge_{v\in V_{d}(T)} \calP(|v|) \wedge \bigwedge_{v\in V_{\iota}(T)} M(|v|) \wedge \bigwedge_{v\in V_{u}(T)} \calQ(|v|); \\
  sH_{W} : \slTree[n]^{op}      & \longrightarrow \sSet,
                                & (T,\iota)              & \longmapsto \bigwedge_{0\leq i\neq \iota\leq h(T)} \Delta[1]_{+}.
\end{align*}

By convention, in $sH_{W}(T)$, the main section is indexed by $t_{\iota}=0$. On morphisms, the functor $\overline{M}_{W}$ is defined using the operadic structures of $\calP$ and $\calQ$, the bimodule structure of $M$ or the symmetric monoidal structure of $\Sp$. On permutation maps, the functor $sH_{W}$ consists in permuting the parameters indexing the levels.
On contraction maps, $\delta_{\{i+1\}} : T\rightarrow T/\{i+1\}$, with $i\in \{0,\ldots,h(T)-1\}$,
there are three cases to consider:
\begin{enumerate}[label={Case \arabic*:}, leftmargin=38pt]
  \item If the levels $i$ and $i+1$ are permutable (in particular $\iota\notin \{i,i+1\}$), then one has :
        \begin{align*}
          sH_{W}(\delta_{\{i+1\}}): sH_{W}(T/\{i+1\}) & \longrightarrow sH_{W}(T),                                 \\
          (t_{0}, \dots, t_{h(T)-1})              & \longmapsto (t_{0},\ldots,t_{i},t_{i},\ldots, t_{h(T)-1}).
        \end{align*}
  \item If the levels $i$ and $i+1$ are not permutable and $i$ is above the main section, then one has:
        \begin{align*}
          sH_{W}(\delta_{\{i+1\}}): sH_{W}(T/\{i+1\}) & \longrightarrow sH_{W}(T),                                     \\
          (t_{0},\ldots,t_{h(T)-1})               & \longmapsto (t_{0},\ldots,t_{i},0,t_{i+1},\ldots, t_{h(T)-1}).
        \end{align*}
  \item If the levels $i$ and $i+1$ are not permutable and $i+1$ is below the main section, then one has:
        \begin{align*}
          sH_{W}(\delta_{\{i+1\}}): sH_{W}(T/\{i+1\}) & \longrightarrow sH_{W}(T),                                     \\
          (t_{0},\ldots,t_{h(T)-1})               & \longmapsto (t_{0},\ldots,t_{i-1},0,t_{i},\ldots, t_{h(T)-1}).
        \end{align*}
\end{enumerate}

\begin{defi}
  Let $M$ be a ($\calP$-$\calQ$)-bimodule.
  Its leveled Boardman--Vogt resolution is:
  $$
    W_{l}M(n) \coloneqq \int^{T\in \slTree[n]} \overline{M}_{W}(T) \wedge sH_{W}(T).
  $$
\end{defi}

\begin{figure}[htbp]
  \hspace{30pt}\includegraphics[scale=0.3]{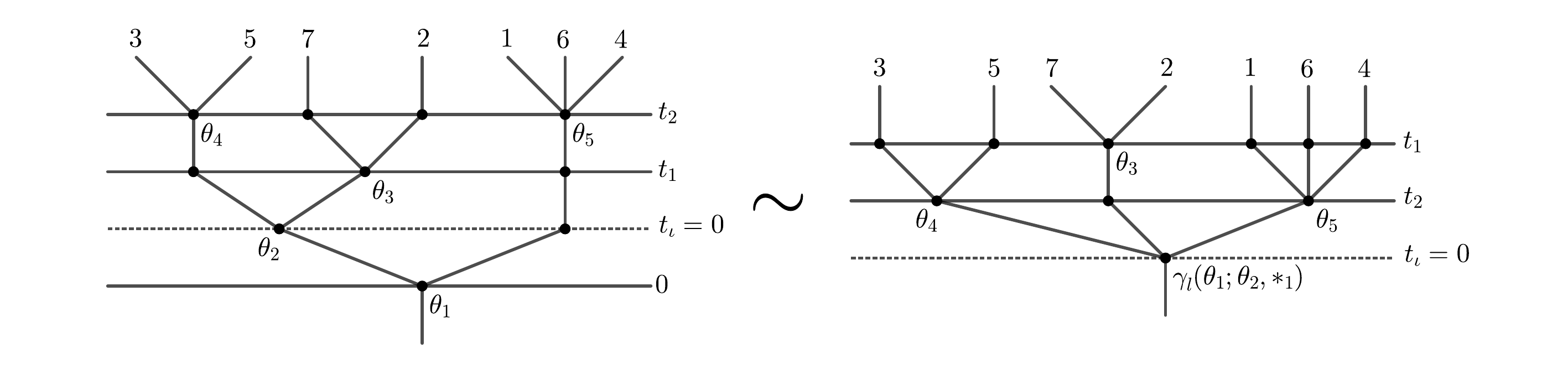}\vspace{-10pt}
  \caption{Illustration of equivalent points in $W_{l}M(7)$.}
\end{figure}

Roughly speaking, a point $[T,\{\theta_{v}\}_{v\in V(T)},\{t_{i}\}_{0\leq i \leq h(T)}]$ in $W_{l}M$ is given by an leveled $n$-tree with section $T=(T,\iota)$ whose vertices above (resp.\ below) the main section are indexed by points in the operads $\calQ$ (resp.\ $\calP$) while the vertices on the main section are labelled by elements in $M$.
The levels other than the main section are indexed by elements in the simplicial set $\Delta[1]$.
Moreover, the equivalence relation, induced by the coend, consists in contracting two consecutive levels $i$ and $i+1$ if we are in one of the following situations:
\begin{enumerate*}
  \item the two levels are permutable and they are indexed by the same parameter in the interval;
  \item the two levels are not permutable, below (resp.\ above) the main section, and the $i$-th (resp.\ $(i+1)$-st) level is indexed by $0$.
\end{enumerate*}

The sequence $W_{l}M$ inherits a ($W_{l}\calP$-$W_{l}\calQ$)-bimodule structure using the left and right operations $\gamma_{L}$ and $\gamma_{R}$ on leveled trees (introduced in Section \ref{SectTreeLev}) and by indexing the new levels by $1$.
This structure is well defined thanks to the definition of the coend.
For instance, the left operation sends the family of elements
\begin{center}
  \includegraphics[scale=0.4]{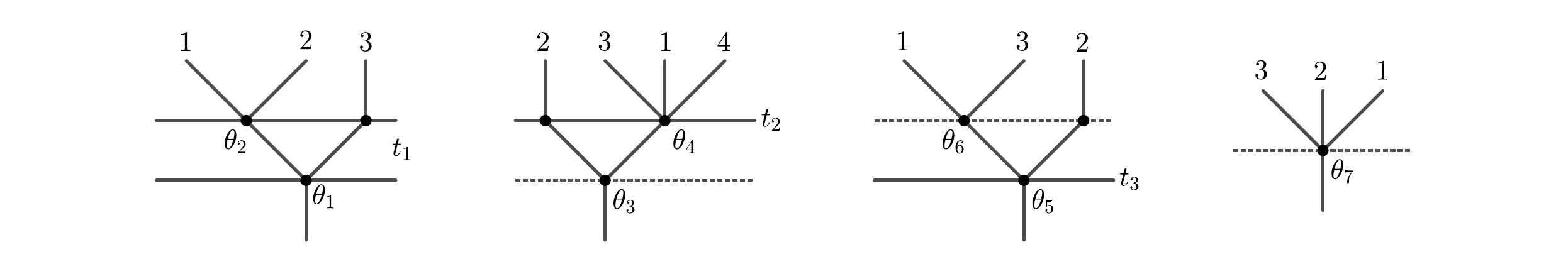}
\end{center}
to the following point
\begin{center}
  \includegraphics[scale=0.37]{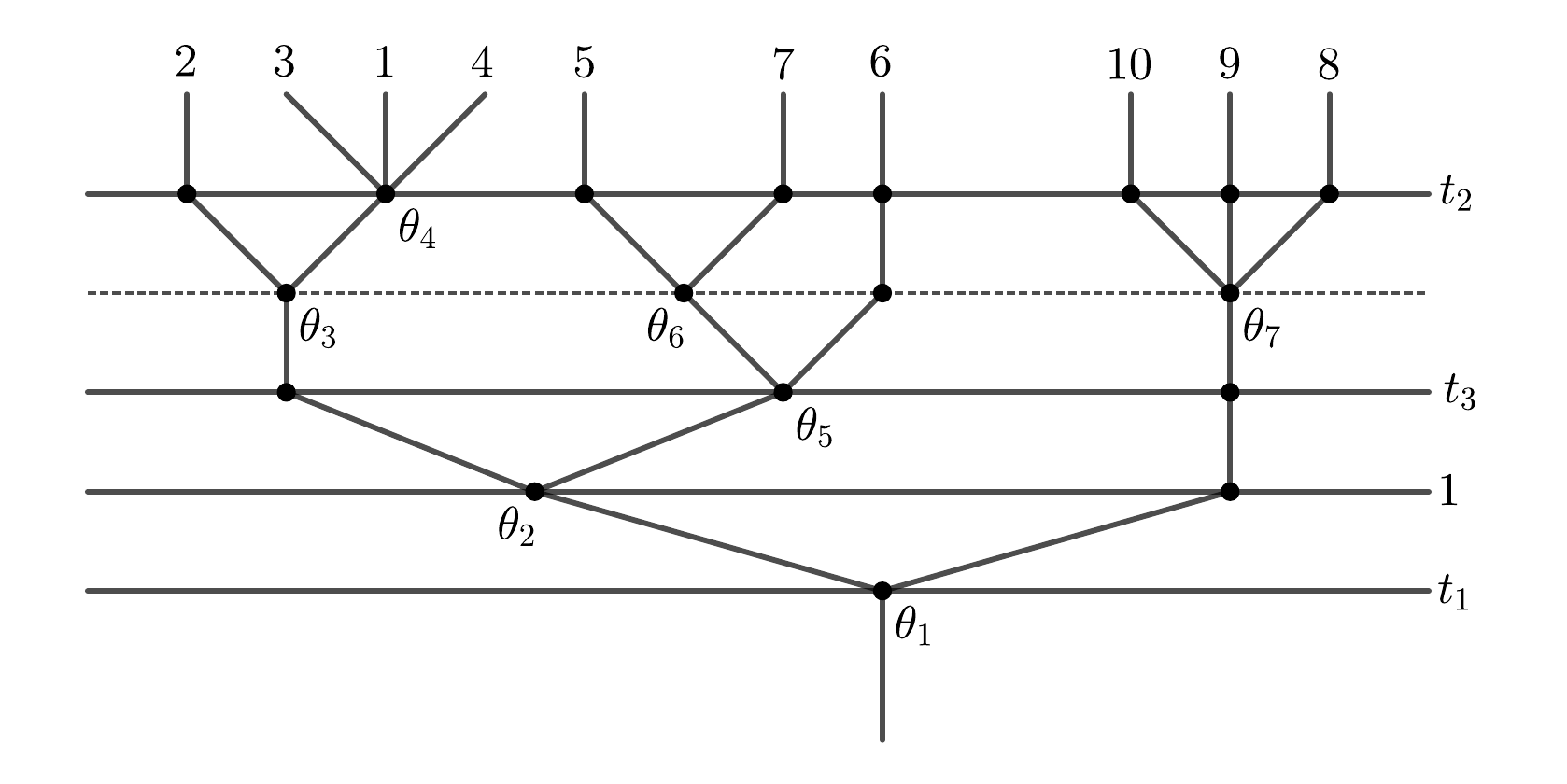}
\end{center}

\begin{pro}\label{pro:w-l-cofibrant}
  If $\calP$ and $\calQ$ are $\Sigma$-cofibrant and $1$-reduced operads and $M$ is $\Sigma$-cofibrant, then $W_{l}M$ is a cofibrant resolution of $M$ in the projective model category of $(W_{l}\calP\text{-}W_{l}\calQ)$-bimodules. In particular, the bimodule map $W_{l}M\rightarrow M$, sending the parameters indexing the levels to $0$, is a weak equivalence of bimodules.
\end{pro}

\begin{proof}
  The map $W_{l}M\rightarrow M$ is a weak equivalence of $(W_{l}\calP\text{-}W_{l}\calQ)$-bimodules.
  In fact, more precisely, it is a retract of the map of symmetric sequences $M\rightarrow W_{l}M$ that sends a point $x\in M$ to the (leveled) corolla indexed by $x$.
  The homotopy consists in bringing the parameters indexing the levels to $0$.

  In order to show that $W_{l}M$ is a cofibrant bimodule, we introduce a filtration of $W_{l}M$ according to the number of leaves.
  First, let us note that an element of $W_{l}M$ is said to be prime if the parameters indexing the levels other than the main section are not equal to $1$.
  Otherwise, it is said to be composite.
  Any element can be decomposed into prime components.
  As illustrated in Figure~\ref{FigFinal}, the prime components are obtained by removing the edges and vertices above (resp.\ below) the sections indexed by $1$ above (resp. below) the main section.

  We now define the filtration of $W_l M$ by using the prime decomposition.
  A prime element is in the $k$-th filtration level $W_{l}M_{k}$ if it has at most $k$ leaves.
  A composite element is in the $k$-filtration term if its prime components are in $W_{l}M_{k}$.
  For instance, the point in Figure~\ref{FigFinal} is in the third filtration level.
  We see that by construction, each $W_{l}M_{k}$ is a ($W_{l}\calP$-$W_{l}\calQ$)-bimodule and one has the following inverse tower of bimodules:
  $$
    W_{l}M_{0}=\emptyset \longrightarrow W_{l}M_{0} \longrightarrow \cdots \longrightarrow W_{l}M_{k} \longrightarrow W_{l}M_{k+1} \longrightarrow \cdots \longrightarrow W_{l}M.
  $$

\hspace{-45pt}\includegraphics[scale=0.31]{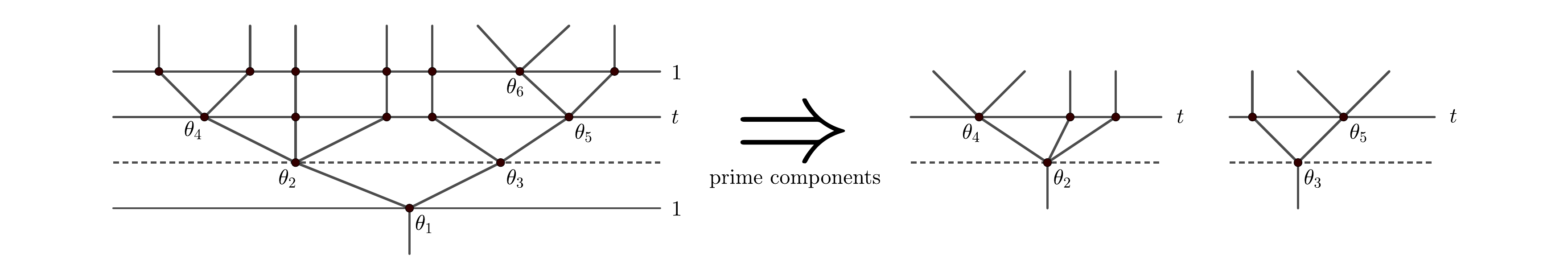}\vspace{-15pt}

  \begin{figure}[!h]
    \begin{center}
      \caption{Illustration of a composite point and its prime components.}\label{FigFinal}
    \end{center}
  \end{figure}\vspace{-15pt}

  Let us show that each map $W_{l}M_{k}\rightarrow W_{l}M_{k+1}$ (with $k\geq 0$) is a cofibration of bimodules.
  We first remark that $W_{l}M_{k}(n)=W_{l}M(n)$ for $n\leq k$.
  We now introduce the following two symmetric sequences concentrated in arity $k+1$:
  $$
    X_{k+1}(n) =
    \begin{cases}
      W_{l}M(k+1)    & \text{if } n=k+1, \\
      \quad\emptyset & \text{otherwise},
    \end{cases}
    \text{\qquad and \qquad}
    \partial X_{k+1}(n) =
    \begin{cases}
      W_{l}M_{k}(k+1) & \text{if } n=k+1, \\
      \quad\emptyset  & \text{otherwise}.
    \end{cases}
  $$
  According to this notation, the bimodule $W_{l}M_{k+1}$ can be obtained from $W_{l}M_{k}$ as a pushout of $(W_{l}\calP\text{-}W_{l}\calQ)$-bimodules:
  $$
    \begin{tikzcd}
      \mathcal{F}_{B}(\partial X_{k+1}) \ar[r] \ar[d] & \mathcal{F}_{B}(X_{k+1}) \ar[d] \\
      W_{l}M_{k} \ar[r] & W_{l}M_{k+1}
    \end{tikzcd}
  $$
  where $\mathcal{F}_{B}:\Sigma Seq_{>0}\rightarrow \Sigma\Bimod_{W_{l}\calP\,,\,W_{l}\calQ}$ is the free bimodule functor.
  Consequently, the map $W_{l}M_{k}\rightarrow W_{l}M_{k+1}$ is a cofibration of bimodules if the inclusion $W_{l}M_{k}(k+1)\rightarrow W_{l}M(k+1)$ is a $\Sigma_{k+1}$-cofibration.
  To prove this statement, we consider another filtration according to the number of levels:
  $$
    W_{l}M_{k}(k+1) \eqqcolon Y_{0} \longrightarrow Y_{1} \longrightarrow \cdots \longrightarrow Y_{i} \longrightarrow Y_{i+1} \longrightarrow \cdots \longrightarrow W_{l}M(k+1).
  $$
  We build the spaces $Y_{i}$ by induction.
  As indicated, we start by setting $Y_0 \coloneqq W_l M_k(k+1)$.
  For any leveled tree $T$, we denote by $sH^{0}_{W}(T)$ the set of elements in $sH_{W}(T)$ that have at least one level over than the main section indexed by $0$ or $1$.
  We also define the set $\slTree[k+1]_{i}$ of leveled trees with section with height $i+1$.
  Finally, let $[T]$ is the isotopy class of $T$ after forgetting the decoration of the leaves by the symmetric group and $\mathrm{Aut}(T)$ is the automorphism group of $T$.
  We then define $Y_{i+1}$ from $Y_i$ by the following pushout diagram:
  \begin{equation}\label{DiagFin}
    \begin{tikzcd}
      \bigvee_{[T] \in \slTree[k+1]_{i} / {\sim}} \bigl( \overline{M}_{W}(T)\wedge sH^{0}_{W}(T) \bigr) \wedge_{\mathrm{Aut}(T)} \Sigma_{k+1} \ar[r] \ar[d]
      & \bigvee_{[T] \in \slTree[k+1]_{i} / {\sim}} \bigl( \overline{M}_{W}(T) \wedge sH_{W}(T) \bigr) \wedge_{\mathrm{Aut}(T)} \Sigma_{k+1} \ar[d] \\
      Y_{i}\ar[r] & Y_{i+1}
    \end{tikzcd}
  \end{equation}
  The $\mathrm{Aut}(T)$-module $\overline{M}_{W}(T)$ is $\mathrm{Aut}(T)$-cofibrant because the operads $\calP$ and $\calQ$ as well as the bimodule $M$ are all $\Sigma$-cofibrant.
  Moreover, the map $sH^{0}_{W}(T) \rightarrow sH_{W}(T)$ is an $\mathrm{Aut}(T)$-cofibration.
  As a consequence of an alternative version of the pushout product axiom~\cite{BM2}, we can thus establish that the upper horizontal map in \eqref{DiagFin} is a $\Sigma_{k+1}$-cofibration.
  Consequently, $Y_{i}\rightarrow Y_{i+1}$ is a $\Sigma_{k+1}$-cofibrations.
  By induction, we thus find that $W_{l}M_{k}(k+1)\rightarrow W_{l}M(k+1)$ is also a $\Sigma_{k+1}$-cofibration, which ends the proof.
\end{proof}

\begin{rmk}\label{rmk:not iso}
  Contrary to the operadic case, we cannot compare directly the leveled Boardman--Vogt resolution introduced in this section with the usual resolution considered by the second author in~\cite{Duc}.
  Indeed, the usual one is a resolution of ($\calP$-$\calQ$)-bimodules in the category of ($\calP$-$\calQ$)-bimodules where vertices are indexed by glued-up simplices.
  By contrast, our construction takes a ($\calP$-$\calQ$)-bimodule and produces a resolution in the category of ($W_{l}\calP$-$W_{l}\calQ$)-bimodules.
\end{rmk}

\paragraph{The leveled Boardman--Vogt resolution for $\Lambda$-bimodules}

Let $\calP$ and $\calQ$ be two 1-reduced $\Lambda$-operads in spectra and let $M$ be a ($\calP$-$\calQ$)-bimodule. In order to get a cofibrant resolution of $M$ in the Reedy model category of ($W_{\Lambda}\calP$-$W_{\Lambda}\calQ$)-bimodules, we make describe the $\Lambda$-structure on the construction introduced in Section \ref{SectBVBi}.
As a symmetric sequence, we set
$$
  W_{\Lambda}M(n) \coloneqq W_{l}M_{>0}(n),
$$
where $M_{>0}$ is the bimodule obtained from $M$ by forgetting the $\Lambda$-structure.
The subscript $\Lambda$ is to emphasize that we work in the category of 1-reduced $\Lambda$-operads.
By definition, $W_{\Lambda}M$ inherits left and right module operations over $W_{\Lambda}\calP$ and $W_{\Lambda}\calQ$, respectively, from $W_{l}M$.

The $\Lambda$-structure map  $h[i]^{*} : W_{\Lambda}\calP(n+1) \to W_{\Lambda}\calP(n)$ (induced by the unique injective increasing map $h[i] : [n] \to [n+1]$ that misses $i$) is defined in the obvious way using the $\Lambda$-structure on the vertex connected to the leaf labelled by $i$.
If the new point so obtained has a level which consists of bivalent vertices, then we remove it.

\begin{figure}[htbp]
  \hspace{-45pt}\includegraphics[scale=0.45]{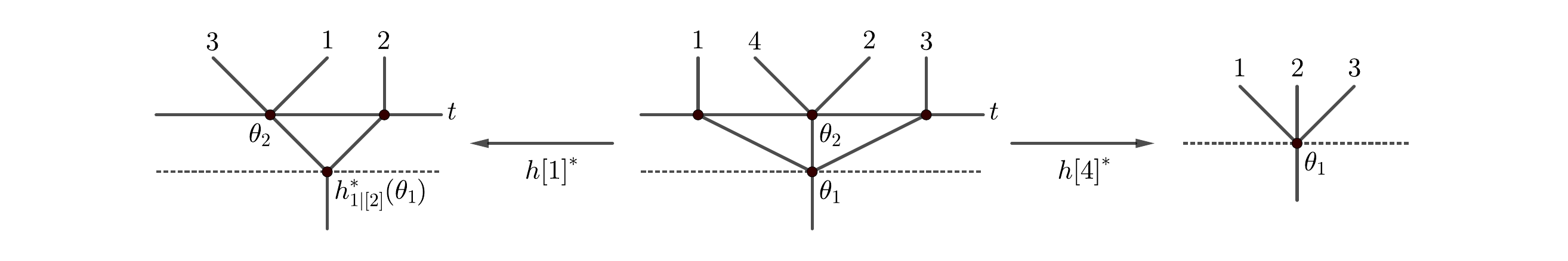}
  \caption{Illustrations of the $\Lambda$-structure associated to $h[1],h[4]:[3]\to [4]$ with $h[1](i)=i+1$ and $h[4](i)=i$.}
\end{figure}

\begin{pro}\label{pro:w-bimod}
  Suppose that $\calP$ and $\calQ$ are $\Sigma$-cofibrant and 1-reduced $\Lambda$-operads, and $M$ is a $\Sigma$-cofibrant $(\calP$-$\calQ)$-bimodule.
  Then $W_{\Lambda}M$ is a cofibrant resolution of $M$ in the Reedy model category of ($W_{\Lambda}\calP$-$W_{\Lambda}\calQ$)-bimodules.
  In particular, the map $\mu : W_{\Lambda}M\rightarrow M$, sending the parameters indexing the levels to $0$, is a weak equivalence of bimodules.
\end{pro}

\begin{proof}
  The map $0 : * \to \Delta^{1}$ is a weak equivalence.
  This implies that the operadic map $\mu$ is a weak equivalence too.
  Moreover, we know from \cite{DucoulombierFresseTurchin2019} that a $\Lambda$-bimodule is Reedy cofibrant if and only if the corresponding $\Sigma$-bimodule is cofibrant in the projective model category.
  The $\Sigma$-bimodule associated to the resolution $W_{\Lambda}M$ is $W_{l}M_{>0}$ which is cofibrant in the projective model category.
\end{proof}

\subsection{The cobimodule of indecomposable elements}

In the previous section, we built a cofibrant resolution $W_{l}M$ for any  $\Sigma$-cofibrant bimodule $M$ in spectra. In what follows we show that the leveled two-sided bar construction of $M$ can be expressed as the suspension of a cobimodule $\Ind(W_{l}M)$. Unfortunately, this identification cannot be extended to $\Lambda$-bimodules since the two-sided leveled bar construction of a $\Lambda$-bimodule in spectra is not necessarily a $\Lambda$-cobimodule.

A point in $W_{l}M$ is said to be \textit{indecomposable} if the elements indexing the levels are different from $1\colon \ast \to \Delta[1]$. The indecomposable bimodule $\Ind(W_{l}M)$ is obtained by modding out the decomposable elements in $W_{l}M$. More precisely, let us introduce a slight variation of the functor $H$ as follows:
\begin{equation}\label{EqFunctBH}
  sH''_{W} : \slTree[n]^{op}  \longrightarrow \sSet,
  \qquad
  (T,\iota) \longmapsto \begin{cases}
    \tilde{\Delta}[T]/ \tilde{\Delta}_{0}[T] & \text{if } n>1, \\
    \hspace{2em}\ast                         & \text{if } n=1,
  \end{cases}
\end{equation}
where $\tilde{\Delta}[T]=\prod_{1\leq i\neq \iota\leq h(T)} \Delta[1]$ and  $\tilde{\Delta}_{0}[T]$ is the simplicial subset consisting of faces where at least one of the levels has value $1$. By construction, $sH''_{W}(T)$ is already a pointed simplicial set.
Taking the functor $\overline{M}_{W}$ (see Section~\ref{SectBVBi}) that sends a tree to the ``vertex-wise'' smash product

\begin{defi}
  The cooperad of indecomposable points is the coend:
  $$
    \Ind(W_{l}M)(n) \coloneqq \int^{T\in \slTree[n]}\overline{M}(T)\wedge sH''_{W}(T).
  $$
\end{defi}

The sequence $\Ind(W_{l}M)=\{\Ind(W_{l}M)(n)\}$ has a ($\Ind(W_{l}\calP)$-$\Ind(W_{l}\calQ)$)-cobimodule structure given by coaction maps:
\begin{align*}
  \gamma^{c}_{R}: \Ind(W_{l}M)(n_{1}+\cdots+n_{k}) & \longrightarrow \Ind(W_{l}M)(k)\wedge \Ind(W_{l}\calQ)(n_{1})\wedge\cdots \wedge \Ind(W_{l}\calQ)(n_{k}); \\
  \gamma^{c}_{L}:\Ind(W_{l}M)(n_{1}+\cdots+n_{k})  & \longrightarrow \Ind(W_{l}\calP)(k)\wedge \Ind(W_{l}M)(n_{1})\wedge\cdots \wedge \Ind(W_{l}M)(n_{k}).
\end{align*}

The right structure sends an element $[T\,;\,\{x_{v}\}\,;\,\{t_{i}\}]$ to the base point if $T$, up to permutations of permutable levels and contractions of permutable levels, is not of the form $\gamma_{R}(T_{0}\,;\,\{T_{i}\})$ with $T_{0}\in \lTree[k]$ and $T_{i}\in \slTree[n_{i}]$, with $i\leq k$. Otherwise, the element is sent to the family

$$
  [T_{0}\,;\,\{x_{v}^{0}\}\,;\,\{t_{j}^{0}\}]\,\,\,\,;\,\,\,\,\big\{ \, [T_{i}\,;\,\{x_{v}^{i}\}\,;\,\{t_{j}^{i}\}]\,\big\}_{1\leq i\leq k}\in \Ind(W_{l}M)(k)\wedge \underset{1\leq i \leq k}{\bigwedge} \Ind(W_{l}\calQ)(n_{i}),
$$
where the parameters indexing the vertices and the levels of the leveled trees with section $T_{0}$ and  the leveled trees $T_{i}$ are induced by the parameters indexing the leveled tree with section $T$. The left structure is defined in the same way. Let us remark that this structure is similar to the cobimodule structure introduced on the leveled two-sided bar construction introduced in Section \ref{SectBarBi}.
Indeed, one has the following connection between the leveled  Boardman--Vogt resolution and the leveled two-sided bar construction:

\begin{pro}\label{ProModFr}
  The leveled two-sided bar construction of the bimodule $M$ is isomorphic to the suspension of the indecomposable cobimodule:
  $$
    \mathcal{B}_{l}[\calP,\calQ](M)\cong \Sigma \Ind(W_{l}M).
  $$
\end{pro}

\begin{proof}
  Taking the indecomposables of $W_{l}M$ identifies to the base point all points whose underlying tree has a level of length 1. The suspension coordinate gives us a length for the $\iota$-th level in the bar construction. To complete the proof, we recall quickly the cobimodule structure on  $\Sigma \Ind(W_{l}\calO))$. We denote by $[T\,;\,\{\theta_{v}\}\,;\,\{t_{j}\}\,;\,x]$ a point in $\Sigma \Ind(W_{l}\calO))$ where $x$ is the suspension coordinate. The right module operation is defined as follows:
  \begin{align*}
    \Sigma \Ind(W_{l}M))(n_{1}+\cdots + n_{k})             & \longrightarrow \Sigma \Ind(W_{l}M))(k)\wedge \Sigma \Ind(W_{l}\calQ))(n_{1})\wedge\cdots \wedge \Sigma \Ind(W_{l}\calQ))(n_{k}), \\
    \left[T\,;\,\{\theta_{v}\}\,;\,\{t_{j}\}\,;\,x \right] & \longmapsto
    \begin{cases}
      \{\,[T_{i}\,;\,\{\theta_{v}^{i}\}\,;\,\{t_{j}^{i}\}\,;\,x_{i}]\,\}_{0\leq i\leq k} & \text{if } T\sim \gamma_{R}(T_{0}\,,\,\{T_{i}\}) \text{ is right decomposable}, \\
      \hspace{4em}\ast                                                                   & \text{otherwise},
    \end{cases}
  \end{align*}
  where $x_{0}=x$ and $x_{i}$, with $1\leq i\leq k$, is the element indexing the level in $\gamma_{R}(T_{0}\,,\,\{T_{i}\})$ corresponding to the root of $T_{i}$.
  The left module structure is defined similarly.
  The reader can easily check that the structure so obtained is well defined and compatible with the isomorphism.
\end{proof}

\subsection{The Boardman--Vogt resolution and the two-sided cobar-bar construction}\label{SectBarCobi}

Similarly to Section \ref{SecBVtoCobBar}, we adapt the definition of the two-sided cobar construction for cobimodules in spectra but using the notion of leveled trees instead of planar trees.
Then we extend it to cobimodules.
Afterwards, we prove that the leveled Boardman--Vogt resolution of a bimodule (or $\Lambda$-bimodule) $M$ in spectra is weakly equivalent to its leveled two-sided cobar-bar construction. Unfortunately, we cannot extend this result to $\Lambda$-bimodules since the two-sided leveled cobar-bar construction for spectra does not inherit a $\Lambda$-structure.

\paragraph{The leveled two-sided cobar construction for cobimodules in spectra}

From \cite{Chi}, we recall that the simplicial indexing category $\Delta$ has an automorphism $\calR$ that sends a total ordered set to the same set with the opposite order. For a simplicial set $X$, the reverse of $X$, denoted by $X^{rev}$, is the simplicial set $X\circ \calR$.  Let $\calP$ and $\calQ$ be two 1-reduced cooperads in spectra.
From a ($\calP$-$\calQ$)-cobimodule $M$, we introduce the functor
\begin{equation*}
  \overline{M}_{\Omega} : \slTree[n]^{op} \longrightarrow \Sp , \qquad T \longmapsto \bigwedge_{\mathclap{v \in V_{d}(T)}} \calP(|v|) \wedge \bigwedge_{\mathclap{v \in V_{\iota}(T)}} M(|v|) \wedge \bigwedge_{\mathclap{v \in V_{u}(T)}} \calQ(|v|),
\end{equation*}
defined on morphisms using the cobimodule structure of $M$.

\begin{defi}
  The leveled two-sided cobar construction associated to a cobimodule $M$ in spectra is the end
  \begin{equation}\label{CobarComod}
    \Omega_{l}[\calP,\calQ](M)(n) \coloneqq \int_{T\in \slTree[n]} \Map\bigl( sH''_{W}(T)^{rev}, \overline{M}_{\Omega}(T) \bigr),
  \end{equation}
  where $sH''_{W}$ is the functor given by the formula \eqref{EqFunctBH}.
  Concretely, a point in $\Omega_{l}[\calP,\calQ](M)(n)$ is a family of maps $\Phi=\{\Phi_{T}: sH''_{W}(T) \rightarrow \overline{\calM}_{\Omega}(T)\}_{T\in \slTree[n]}$ satisfying the following relations: for each permutation $\sigma$ and each contraction morphism $\delta_{N}$, one has the commutative diagrams
  \begin{equation}\label{RelCobarbi}
    \begin{tikzcd}[column sep = large]
      sH''_{W}(T\cdot\sigma) \ar[r, "H''(\sigma)"] \ar[d, "\Phi_{T\cdot\sigma}"] & sH''_{W}(T) \ar[d, "\Phi_{T}"]
      & sH''_{W}(\delta_{N}(T)) \ar[r, "H''_W(\delta_{N})"] \ar[d, "\Phi_{\delta_{N}(T)}"] & sH''_{W}(T) \ar[d, "\Phi_{T}"] \\
      \overline{M}_{\Omega}(T\cdot \sigma) \ar[r, "\overline{M}_{\Omega}(\sigma)"] & \overline{M}_{\Omega}(T)
      & \overline{M}_{\Omega}(\delta_{N}(T)) \ar[r, "\overline{M}_{\Omega}(\delta_{N})"] & \overline{M}_{\Omega}(T)
    \end{tikzcd}
  \end{equation}

  The sequence $\Omega_{l}[\calP,\calQ](M)=\{\Omega_{l}[\calP,\calQ](M)(n)\}$ forms a ($\Omega_{l}\calP$-$\Omega_{l}\calQ$)-bimodule whose structure maps:
  \begin{align*}
    \gamma_{L} : \Omega_{l}\calP(k)\wedge \Omega_{l}[\calP,\calQ](M)(n_{1})\wedge \cdots \wedge \Omega_{l}[\calP,\calQ](M)(n_{k}) & \longrightarrow \Omega_{l}[\calP,\calQ](M)(n_{1}+\cdots + n_{k}),                                        \\
    \{ \Phi_{0}; \{\Phi_{i}\} \}                                                                                                  & \longmapsto \gamma_{L}(\Phi_{0}; \{\Phi_{i}\}) = \bigl\{ \gamma_{L}(\Phi_{0}; \{\Phi_{i}\})_{T} \bigr\}; \\
    \gamma_{R}: \Omega_{l}[\calP,\calQ](M)(k) \wedge  \Omega_{l}\calQ(n_{1})\wedge \cdots \wedge \Omega_{l}\calQ(n_{k})           & \longrightarrow \Omega_{l}[\calP,\calQ](M)(n_{1}+\cdots + n_{k}),                                        \\
    \{ \Phi_{0};\{\Phi_{i}\} \}                                                                                                   & \longmapsto \gamma_{R}(\Phi_{0}; \{\Phi_{i}\}) = \bigl\{ \gamma_{R}(\Phi_{0};\{\Phi_{i}\})_{T} \bigr\},
  \end{align*}
  are defined as follows. If, up to permutations and contractions of permutable levels, the leveled tree with section $T$ is not of the form $\gamma_{L}(T_{0}\,;\,\{T_{i}\})$, with $T_{0}\in \lTree[k]$ and $T_{i}\in \slTree[n_{i}]$, then $\gamma_{L}(\Phi_{0}\,;\,\{\Phi_{i}\})_{T}$ sends any decoration of the levels to the basepoint. Otherwise, one has the composite map
  $$
    \gamma_{L}(\Phi_{0}; \{\Phi_{i}\})_{T}: sH''_{W}(T)\longrightarrow sH''_{W}(\gamma_{L}(T_{0}; \{T_{i}\})) \cong H_{W}(T_{0}) \wedge \bigwedge_{\mathclap{1\leq i \leq k}} sH''_{W}(T_{i}) \longrightarrow \overline{\calP}_{\Omega}(T_{0}) \wedge \bigwedge_{\mathclap{1\leq i \leq k}} \overline{M}_{\Omega}(T_{i})\cong \overline{M}_{\Omega}(T).
  $$
  Similarly, if, up to permutations and contractions of permutable levels, the leveled tree with section $T$ is not of the form $\gamma_{R}(T_{0}\,;\,\{T_{i}\})$, with $T_{0}\in \slTree[k]$ and $T_{i}\in \lTree[n_{i}]$, then $\gamma_{R}(\Phi_{0}\,;\,\{\Phi_{i}\})_{T}$ sends any decoration of the levels to the basepoint. Otherwise, one has the composite map
  $$
    \gamma_{R}(\Phi_{0};\{\Phi_{i}\})_{T} : sH''_{W}(T) \longrightarrow sH''_{W}(\gamma_{L}(T_{0}; \{T_{i}\})) \cong sH''_{W}(T_{0}) \wedge \bigwedge_{\mathclap{1\leq i \leq k}} H_{W}(T_{i}) \longrightarrow \overline{M}_{\Omega}(T_{0})\wedge \bigwedge_{\mathclap{1\leq i \leq k}} \overline{\calQ}_{\Omega}(T_{i}) \cong \overline{M}_{\Omega}(T).
  $$
\end{defi}

\paragraph{Connection between Boardman--Vogt resolutions and two-sided cobar-bar constructions}

Let $\calP$ and $\calQ$ be two 1-reduced operads in spectra and $M$ be a ($\calP$-$\calQ$)-bimodule.
In Section \ref{SectBVBi}, we built a cofibrant resolution of $M$ using the leveled Boardman--Vogt resolution $W_{l}M$.
In Section \ref{SectBarBi}, we introduced a leveled version of the bar construction, denoted $\calB_{l}[\calP,\calQ](M)$.
Using our definition of the leveled two-sided cobar construction in the previous section, we wish to build a map from the Boardman--Vogt resolution and the bar-cobar construction.
If there is no ambiguity about the operads $\calP$ and $\calQ$, by notation $\Omega_{l}\calB_{l}(M)$ we understand the leveled two-sided cobar-bar resolution $\Omega_{l}[\calB_{l}\calP, \calB_{l}\calQ]\bigl( \calB_{l}[\calP,\calQ](M) \bigr)$.
Using the maps of operads $W_{l}\calP \rightarrow \Omega_{l}\calB_{l}(\calP)$ and $W_{l}\calQ\rightarrow \Omega_{l}\calB_{l}(\calQ)$, we want to show that this maps induces a weak equivalence of ($W_{l}\calP$-$W_{l}\calQ$)-bimodules:

\begin{pro}\label{prop:gamma-bimodule}
  The morphism defined in~\eqref{eq:def-gamma} is a natural weak equivalence of ($W_{l}\calP$-$W_{l}\calQ$)-bimodules
  \[\Gamma : W_{l}M \xrightarrow{\;\sim\;} \Omega_{l}\calB_l(M). \]
\end{pro}

Recall that a point in the leveled two-sided cobar-bar construction is the data of a family of maps
{$\Phi = \left\{\Phi_{T}:sH''_{W}(T)\rightarrow \overline{\calB_{l}(M)}(T), \; T\in \slTree[n] \right\}$}
satisfying the relations \eqref{RelCobarbi}.
A point in $\overline{\calB_{l}(M)}(T)$ is a family of elements such that the vertices on the main section of $T$ (resp.\ below and above the main section of $T$) are indexed by points in $\calB_{l}(M)$ (resp.\ by points in $\calB_{l}(\calP)$ and $\calB_{l}(\calQ)$).

\begin{notat}
  Let $T_{1}$ and $T_{2}$ be two leveled $n$-trees with section. We say that $T_{1}\geq T_{2}$ if, up to permutations and contractions of permutable levels, $T_{2}$ can be obtained from $T_{1}$ by contracting levels. In that case, we fix the following notation:
  \begin{itemize}
    \item Each vertex $v\in V(T_{2})$ corresponds to a sub-leveled tree $T_{1}[v]$ of the leveled tree $T_{1}$ in such a way that $T_{1}$ is obtained (up to permutations and contractions of permutable levels) by grafting all the trees $T_{1}[v]$ together. The sub-trees corresponding to vertices on the main section of $T_{2}$ are also leveled trees with section while the sub-trees corresponding to vertices above or below the main section are just leveled trees. For instance, from the two leveled trees $T_{1}\geq T_{2}$

          \hspace{-30pt}\includegraphics[scale=0.32]{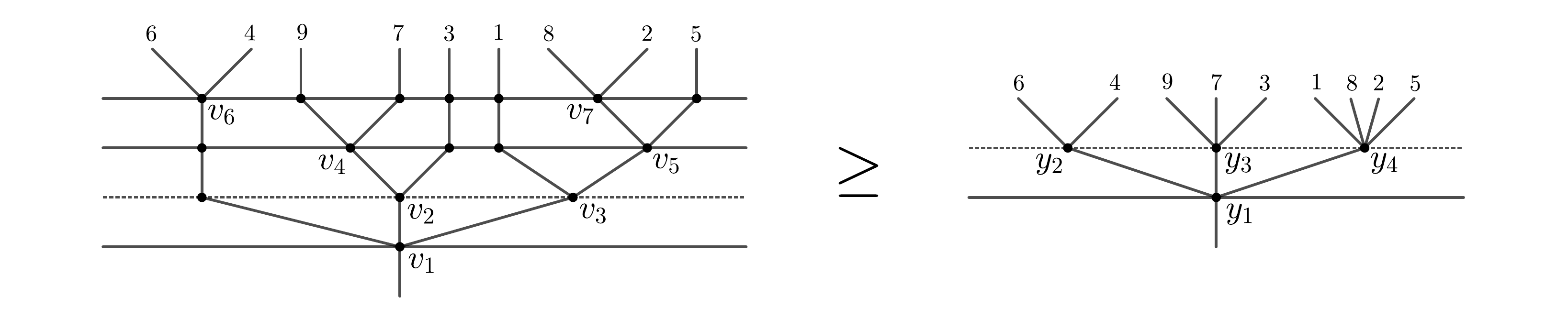}

          \noindent the sub-leveled trees associated to the vertices $y_{1},\ldots,y_{4}$ are the following ones:

          \hspace{-40pt}\includegraphics[scale=0.33]{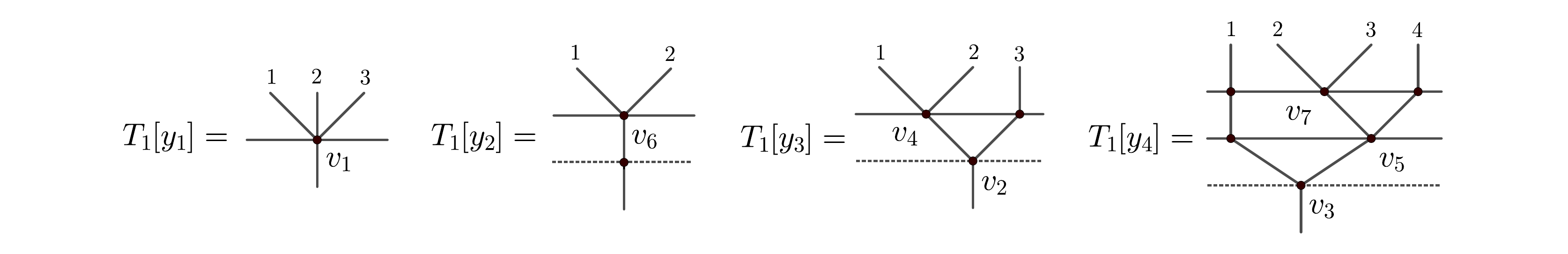}

    \item For any vertex $v\in V(T_{2})$, we denote by $or_{v}:V(T_{1}[v])\rightarrow V(T_{1})$ the map assigning to a vertex in $T_{1}[v]$ the corresponding vertex in $T_{1}$.
          Similarly, let $ol_{v}:\{0,\ldots,h(T_{1}[v])\}\rightarrow \{0,\ldots,h(T_{1})\}$ be the map assigning to a level of $T_{1}[v]$ the corresponding level of $T_{1}$.
          For instance, in the above example, one has $ol_{y_{2}}(1)=3=ol_{y_{4}}(2)$ and $ol_{y_{4}}(1)=2$.
  \end{itemize}
\end{notat}

The map between the leveled Boardman--Vogt resolution and the leveled two-sided cobar-bar construction
\begin{equation}\label{eq:def-gamma}
  \begin{aligned}
    \Gamma_{n}:W_{l}M(n)                                     & \longrightarrow \Omega_{l}\calB_{l}(M)(n),                                                               \\
    \underline{x} = [T_{1}\,;\,\{\theta_{v}\}\,;\,\{t_{l}\}] & \longmapsto \Phi_{\underline{x}}=\bigl\{ \Phi_{\underline{x};T_{2}}\,\,,\,\,T_{2}\in \slTree[n] \bigr\},
  \end{aligned}
\end{equation}
is defined as follows.
If $T_{1} \ngeq T_{2}$, then $\Phi_{\underline{x};T_{2}}$ is the basepoint in $\overline{\calB_{l}(M)}(T_{2})$.
Otherwise, one has:
\begin{align*}
  \Phi_{\underline{x};T_{2}}:H''(T)       & \longrightarrow \overline{\calB_{l}(M)}(T_{2}),                                                 \\
  \{\tilde{t}_{i}\}_{0\leq i\leq h(T_{2}} & \longmapsto \bigl\{ [T_{1}[v] ; \{\theta_{y}[v]\} ; \{t_{j}[v] \}] \bigr\}_{v \in V(T_{2})} \,,
\end{align*}
where
$$
  \theta_{y}[v]=\theta_{or_{v}(y)} \hspace{15pt} \text{and} \hspace{15pt}
  t_{j}[v] =
  \begin{cases}
    t_{ol_{v}(j)}                                              & \text{if } j>0,                                                             \\
    1-t_{ol_{v}(j)}                                            & \text{if } j=0 \text{ and } v \text{ is on the main section of } T_{2},     \\
    \max\bigl(0\,;\, t_{ol_{v}(j)}-\tilde{t}_{\lev(v)}\,\bigr) & \text{if } j=0 \text{ and } v \text{ is not on the main section of } T_{2}.
  \end{cases}
$$

\begin{rmk}
  We cannot extend these results to $\Lambda$-bimodules, see Remark~\ref{rmk:lambda-fresse}.
  However, in the dual case, the $\Lambda$-structure exists and is compatible with our constructions (see Theorem~\ref{thm:finale}).
\end{rmk}

\section{Fibrant resolutions for Hopf $\Lambda$-cooperads}
\label{sec:fibr-resol-hopf}

For any 1-reduced Hopf $\Lambda$-cooperad $\calC$, we build a 1-reduced Hopf $\Lambda$-cooperad $W_{l}\calC$ together with a quasi-isomorphism $\eta:\calC\rightarrow W_{l}\calC$ such that $W_{l}\calC$ is fibrant.
Contrary to \cite{FresseTurchinWillwacher2017}, our construction uses leveled trees.
We show that the cooperad $W_{l}\calC$ so obtained is (as a dg-cooperad) the leveled bar construction of an augmented dg-operad.
Inspired by the methods introduced in \cite{Livernet2012,Chi}, we show that our fibrant replacement $W_{l}\calC$ is quasi-isomorphic to the fibrant replacement introduced in \cite{FresseTurchinWillwacher2017}.

\subsection{The leveled bar construction for 1-reduced cooperads}
\label{SectBarCoop}

We recall from Section \ref{SectTree} that $\lTree_{\iso}[n]$ is the category of leveled trees whose morphisms are generated by isomorphisms of leveled trees, contraction morphisms of permutable levels and permutation morphisms of of permutable levels. Let us define a functor
$$
  \calF^{c}_{l}:\dg\Sigma \Seq_{>1}^{c}\longrightarrow \Sigma\Cooperad.
$$
From such a $1$-reduced symmetric cosequence $X$, we construct the following two functors:
\begin{align*}
  \overline{X}_{F} : \lTree_{\iso}[n]^{op} & \longrightarrow \Ch ,                     & E_{1}: \lTree_{\iso}[n] & \longrightarrow \Ch, \\
  T                                        & \longmapsto \bigotimes_{v\in V(T)}X(|v|); & T                       & \longmapsto \K.
\end{align*}
\begin{defi}
  The leveled cofree cooperad functor $\calF^{c}_{l}$ is defined as the end:
  $$
    \calF^{c}_{l}(X)(n) \coloneqq \int_{T\in \lTree_{\iso}[n]} \overline{X}_{F}(T)\otimes E_{1}(T).
  $$
\end{defi}
Concretely, an element in $\calF^{c}_{l}(X)(n)$ is a map $\Phi$ which maps leveled trees $T$ to elements $\Phi(T)\in \overline{X}(T)$ satisfying the relations:
\begin{enumerate*}[label={(\arabic*)}]
  \item for each permutation $\sigma$ of permutable levels, one has $\Phi(T)=\Phi(\sigma\cdot T)$;
  \item for each morphism $\delta_{i}:T\rightarrow T/\{i\}$ contracting two permutable levels, one has $\Phi(T)=\Phi(T/\{i\})$.
\end{enumerate*}

The cooperadic structure is given by
\begin{align*}
  \gamma'': \calF^{c}_{l}(X)(n_{1}+\cdots+n_{k}) & \longrightarrow \calF^{c}_{l}(X)(k) \otimes \calF^{c}_{l}(X)(n_{1})\otimes \cdots \otimes \calF^{c}_{l}(X)(n_{k}),                     \\
  \Phi                                           & \longmapsto \widehat{\Phi} \coloneqq \bigl\{ \widehat{\Phi}(T_{0},\{T_{i}\}_{1\leq i \leq k}) = \Phi(\gamma(T_{0},\{T_{i}\})) \bigr\},
\end{align*}
where $\gamma$ is the operation \eqref{formulatreecomp}. This formula is well defined: a point $\Phi \in \calF^{c}_{l}(X)$ is equivariant up to permutations of permutable levels and contractions of permutable levels.

\begin{defi}[The usual cofree cooperadic functor]\label{def:usual-cofree}
  We denote the construction of the usual cofree cooperad functor by $\calF^{c}$. We use the category $\Tree_{\iso}^{\geq 2}[n]$, introduced in Section \ref{sec:categ-plan-trees}, of planar trees having $n$ leaves and without univalent or bivalent vertices. In that case morphisms consist of isomorphisms of planar trees. For any $1$-reduced symmetric cosequence $X$, we set
  $$
    \overline{X}_{u}:\Tree_{\iso}^{\geq 2}[n]^{op}\longrightarrow \Ch, \quad T\longmapsto \underset{v\in V(T)}{\bigotimes} X(|v|).
  $$
  Then the cofree cooperad functor is defined as the end
  $$
    \calF^{c}(X)(n)=\int_{T\in \Tree_{\iso}^{\geq 2}[n]}\overline{X}_{u}(T)\otimes E_{1}(T).
  $$
  An element in $\calF^{c}(X)(n)$ is a map $\Phi$ which maps planar trees $T\in \Tree_{\iso}^{\geq 2}[n]$ to elements $\Phi(T)\in \overline{X}(T)$ satisfying the relation: for each isomorphism of planar trees $\sigma$, one has $\Phi(T)=\Phi(\sigma\cdot T)$. The operadic structure is obtained using the operadic composition $\gamma$ of the operad $\Tree_{\iso}^{\geq 2}$:
  \begin{align*}
    \gamma'': \calF^{c}(X)(n_{1}+\cdots+n_{k}) & \longrightarrow \calF^{c}(X)(k) \otimes \calF^{c}(X)(n_{1})\otimes \cdots \otimes \calF^{c}(X)(n_{k}),                                 \\
    \Phi                                       & \longmapsto \widehat{\Phi} \coloneqq \bigl\{ \widehat{\Phi}(T_{0},\{T_{i}\}_{1\leq i \leq k}) = \Phi(\gamma(T_{0},\{T_{i}\})) \bigr\}.
  \end{align*}
\end{defi}

\begin{pro}\label{FreeFunc}
  The functor $\calF_{l}^{c}$ is isomorphic to the usual cofree cooperad functor $\calF^{c}$. In particular, $\calF_{l}^{c}$ is the right adjoint to the forgetful functor from the category of 1-reduced dg-cooperads to $1$-reduced $\Sigma$-cosequences of cochain complexes.
\end{pro}

\begin{proof}
  By using the comparison morphisms $\alpha$ and $\beta$ between planar trees $\Tree_{\iso}^{\geq 2}[n]$ and leveled trees $\lTree_{\iso}[n]$, introduced in Section~\ref{CompTree}, we are now able to give an explicit isomorphism between the leveled and usual versions of the cofree cooperad functors:
  $$
    L_{n}:\calF^{c}(X)(n)\leftrightarrows \calF^{c}_{l}(X)(n):R_{n}.
  $$

  Let $\Phi$ be an element in $\calF^{c}(X)(n)$ and $T$ be an leveled $n$-tree. Then $L_{n}(\Phi)$ evaluated at $T$ is given by $\Phi\circ \alpha (T)$. Conversely, let $\Phi'$ be an element in $\calF^{c}_{l}(X)(n)$ and $T'$ be a rooted planar tree in  $\Tree_{\iso}^{\geq 2}[n]$. Then $R_{n}(\Phi')$ evaluated at $T'$ is given by $\Phi'\circ \beta (T')$.  The map $R_{n}$ does not depend on the fixed point $T_{l}\in \alpha^{-1}(T)$ since the decoration $\Phi'(T_{l})$ does not depend on $T_{l}$ (up to contractions  of permutable levels and permutations of permutable levels). Therefore the maps $L_{n}$ and $R_{n}$ are well defined and provide isomorphisms preserving the cooperadic structures.
\end{proof}

\begin{cor}\label{cor:cofree-as-product}
  Let $X$ be a 1-reduced symmetric cosequence and $n$ an integer.
  There is an isomorphism of cochain complexes compatible with the symmetric group coaction:
  \[ \calF_{l}^{c}(X)(n) \cong \prod_{[T] \in \Tree_{\iso}^{\ge 2}[n] / \cong} \overline{X}_{u}(T). \]
  where the product is over the isomorphism classes of planar trees.
\end{cor}

\begin{proof}
  We use that the functor $E_{1}$ is constant.
  The category $\lTree_{\iso}[n]$ has two kinds of morphisms: those that keep the height constant are isomorphisms (i.e. isomorphisms of planar trees and permutations of permutable levels), and those that strictly decrease the height (i.e. contractions of permutable levels).
  Elements of the end defining $\calF^l_c(X)$ can thus be expressed using only trees of minimal height, and values only depend on the isomorphism class of such trees.

  Moreover, the symmetric cosequence $X$ is 1-reduced, so $\overline{X}_{F}$ is constant on permutations of permutable levels and contractions of permutable levels and is given by $\overline{X}_F(T) = \overline{X}_u(\alpha(T))$.
  The argument in Section~\ref{SectTreeLev} then show that two trees in $\lTree_{\iso}^{\geq 2}[n]$ are connected by morphisms if and only if they define the same isomorphism class in $\Tree_{\iso}^{\geq 2}[n]$.
  Therefore, the product above is a product over isomorphism classes $\Tree_{\iso}^{\geq 2}[n]$.
\end{proof}

\begin{defi}\label{DefBarlev}
  The leveled bar construction of a 1-reduced dg-operad $\calO$ is given by
  $$
    \calB_{l}(\calO) \coloneqq \bigl( \calF_{l}^{c}(\Sigma \calU(\overline{\calO})),\,\, d_{\mathrm{int}}+d_{\mathrm{ext}} \bigr),
  $$
  where $\calU(\overline{\calO})$ is the sequence underlying the augmentation ideal of $\calO$.
  For $\Phi \in \calB_{l}(\calO)$ and $T \in \lTree_{\iso}[n]$ we set:
  $$
    \deg'(\Phi,T) =\sum_{v\in V_{\ge2}(T)} (\deg(\theta_{v}) + 1).
  $$
  An element $\Phi \in \calB_{l}(\calO)$ is then said to be of degree $d$ if $\deg'(\Phi,T) = d$ for all trees $T$.

  The cooperadic structure and the Hopf structure are inherited from the cofree cooperad functor $\calF^{c}_{l}(\mathcal{U}(\overline{\calO}))$.
  The differential is the sum of two terms:
  \begin{itemize}
    \item The differential $d_{\mathrm{int}}$ is the internal differential of the cochain complex $\mathcal{U}(\overline{\calO})$.
    \item The differential $d_{\mathrm{ext}}$ roughly speaking consists in contracting two consecutive levels.
          More precisely, for $\Phi \in \calB_{l}\calO$ and $T \in \lTree_{\iso}[n]$, consider the set of trees
          \[ D_{T} \coloneqq
            \left\{ (T', i) \in \lTree_{\iso}[n] \times \mathbb{N} \; \middle\vert
            \begin{aligned}
              & T = T' / \{ i \} \text{ and  there is a unique edge between} \\
             & \text{levels } i-1 \text{ and } i \text{ that joins two non-bivalent vertices}
            \end{aligned}
            \right\}. \]
          (Note that levels $i-1$ and $i$ cannot be permutable in the previous definition.)
          Then the element $(d_{\mathrm{ext}}\Phi)(T)$ is the sum $\sum_{(T', i) \in D_{T}} \pm \gamma_{i} \Phi(T')$, where $\gamma_{i}$ uses the operadic structure of $\calO$ to contract the levels $i$ and $i+1$.
          See Figure~\ref{fig:d-ext-bar} for an example.
  \end{itemize}
\end{defi}

\begin{figure}[htbp]
  \centering
  \includegraphics[scale=0.4]{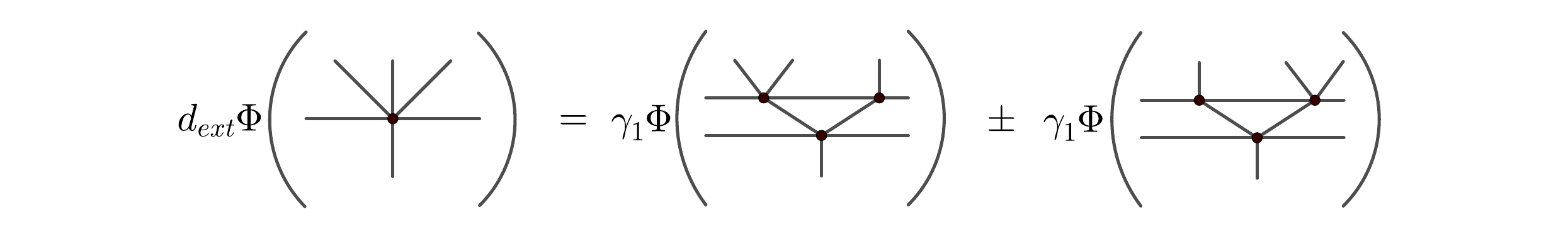}
  \caption{External differential in $\calB_l\calO$.}
  \label{fig:d-ext-bar}
\end{figure}

\begin{pro}\label{prop:bar-well-def}
  The leveled bar construction $\calB_{l}\calO$ of a 1-reduced dg-operad is a well defined 1-reduced dg-cooperad.
\end{pro}

\begin{proof}
  Let us check that $d_{\mathrm{int}} + d_{\mathrm{ext}}$ is a well-defined coderivation that squares to zero.
  It is clear that $d_{\mathrm{int}}$ is well-defined and that it is a coderivation that squares to zero.

  We have to check that if $T_{1}$ and $T_{2}$ define the same planar trees, then $d_{\mathrm{ext}}\Phi(T_{1}) = d_{\mathrm{ext}}\Phi(T_{2})$ in the quotient defining $\calB_{l}\calO$.
  In fact, we can see $d_{\mathrm{ext}}$ is the unique coderivation induced by the following map $\alpha : \calF_{l}^{c}(\Sigma\calU(\overline{\calO})) \to \calU(\overline{\calO})$ and is therefore well-defined.
  Let $c_{n}$ be the corolla with $n$ leaves, then $D_{c_{n}}$ is the set of trees $T \in \lTree[n]$ with exactly two levels and exactly one vertex with $\ge2$ incoming edges on the second level.
  The element $\alpha(\Phi) \in \calO(n)$ is the sum over all $T \in D_{c_{n}}$ of the map of the operad structure maps to $\Phi(T) \in \calO(k) \otimes \calO(l)$.

  We moreover have that $d_{\mathrm{ext}}d_{\mathrm{int}} + d_{\mathrm{int}}d_{\mathrm{ext}} = 0$ because the operad structure of $\calO$ is compatible with the differential.
  Let us finally check that $d_{\mathrm{ext}}^{2} = 0$.
  Since $d_{\mathrm{ext}}$ is a coderivation, it is enough to check this when corestricted to cogenerators.
  We thus have to check that $d_{\mathrm{ext}}^{2}\Phi(T) = 0$ for all trees $T$ with three levels.
  Just like in the case of the standard bar construction, this follows from the associativity of the operad structure of $\calO$ and the signs in the differential.
\end{proof}

\begin{defi}[The usual bar construction for dg-operads]\label{def:usual-bar-dg}
  The usual bar construction $\calB(O)$ is defined as the cofree cooperad generated by the augmentation ideal of $\calO$ (analogously to Definition~\ref{DefBarlev}):
  $$
    \calB(\calO)=\bigl( \calF^{c}(\Sigma \calU(\overline{\calO})),\; d_{\mathrm{int}}+d_{\mathrm{ext}} \bigr).
  $$
  The degree of an element evaluated to a planar tree $T\in \Tree^{\geq 2}_{\iso}$ is the degree of the decorations plus the number of vertices.
  The differential is composed of the internal differential coming from the differential graded algebra $\mathcal{U}(\calO)$ and an external differential which is dual to edge contraction and uses the operadic structure of $\calO$ (compare with the description of $d_{\mathrm{ext}}$ above).
\end{defi}

\begin{pro}\label{ProBlB}
  Let $\calO$ be a 1-reduced dg-operad.
  The leveled bar construction is isomorphic to the usual one:
  $$
    \calB_{l}(\calO)\cong \calB(\calO).
  $$
\end{pro}

\begin{proof}
  There is an isomorphism of graded cooperads between $\calB_{l}(\calO)$ and $\calB(\calO)$ thanks to Proposition~\ref{FreeFunc}.
  We just need to check that it is compatible with the differential.
  Recall that the isomorphism $L : \calF^c(\Sigma\calU(\overline{\calO})) \to \calF^c_l(\Sigma\calU(\overline{\calO}))$ is defined by $\Phi \mapsto \Phi \circ \alpha_{\iso}$, where $\alpha_{\iso} : \lTree_{\iso} \to \Tree^{\ge 2}_{\iso}$ is the functor that forgets levels and bivalent vertices.
  The operad $\calO$ is $1$-reduced, therefore $\calU(\calO)(1) = 0$.
  For some $\Phi \in \calF^c(\Sigma\calU(\overline{\calO}))$, we then see that all the terms in $(d \circ L)(\Phi)(T)$ correspond exactly to the terms in $(L \circ d)(\Phi)(T)$, as the vertices and edges of $T$ are in bijection with those of $\alpha_{\iso}(T)$.
\end{proof}

\subsection{The leveled Boardman--Vogt resolution}\label{SectBVCoop}

This section is split into three parts.
First, we introduce a leveled version of the Boardman--Vogt resolution for $1$-reduced cooperads in chain complexes and we compare this alternative construction to the usual Boardman--Vogt resolution introduced by Fresse--Turchin--Willwacher in \cite{FresseTurchinWillwacher2017}.
The two last parts are devoted to extending this construction to $1$-reduced $\Lambda$-cooperads and $1$-reduced Hopf $\Lambda$-cooperads, respectively.

\paragraph{The leveled Boardman--Vogt resolution for 1-reduced Hopf cooperads}

Let $\calC$ be a 1-reduced Hopf cooperad.
In what follows, we introduce a Boardman--Vogt resolution $W_{l}\calC$ of $\calC$ producing a fibrant resolution in the projective model category of 1-reduced Hopf cooperads.
We describe its cooperadic structure and we prove that there is a natural weak equivalence of cooperads $\eta:\calC\rightarrow W_{l}\calC$. First, we consider the following two functors:
\begin{align}
  \overline{\calC}_{W} :\lTree[n]^{op} & \longrightarrow \CDGA, \quad T \longmapsto \displaystyle\bigotimes_{\mathclap{v\in V(T)}} \calC(|v|);       \\
  E_{W} : \lTree[n]                    & \longrightarrow \CDGA, \quad T  \longmapsto\bigotimes_{\mathclap{1\leq i\leq h(T)}} \K[t,dt].\label{eq:e-w}
\end{align}

The functor $\overline{\calC}_{W}$ consists in indexing the vertices of leveled trees by elements in the cooperad $\calC$ while the functor $E_{W}$ associates to each level bigger than $1$ a polynomial differential form in $\K[t,dt]$.
If $1 \le i, i+1 \le h(T)$ are permutable levels, then the corresponding permutation $\sigma_{i}$ induces operations $\overline{\calC}_{W}(\sigma_{i})$ and $E_{W}(\sigma_{i})$ which are defined using the symmetric monoidal structure of $\CDGA$.
For $i \in \{0, \dots, h(T)-1\}$, the morphism $\delta_{\{i+1\}}:T\rightarrow T/\{i+1\}$ induces an operation $\overline{\calC}_{W}(\delta_{\{i+1\}})$ which is defined using the cooperadic structure and the symmetric monoidal structure.
However, in order to define $E_{W}(\delta_{\{i+1\}})$, there are two cases to consider:
\begin{enumerate}
  \item If the levels $i$ and $i+1$

  are permutable, then the map is obtained by taking the product of the differential polynomial forms indexing the corresponding levels.
  \item If the levels $i$ and $i+1$ are not permutable, then the map consists in evaluating to $0$ the polynomial differential forms indexing the $(i+1)$-st

  level.
\end{enumerate}

\begin{defi}
  The Boardman--Vogt resolution for 1-reduced Hopf cooperad is the end:
  $$
    W_{l}\calC(n) \coloneqq \int_{T\in \lTree[n]}\overline{\calC}_{W}(T)\otimes E_{W}(T).
  $$
\end{defi}

In other words, an element in $W_{l}\calC(n)$ is a map $\Phi$ which associates to each leveled tree $T$ an element $\Phi(T)\in \overline{\calC}_{W}(T)\otimes E_{W}(T)$ satisfying the following relations:
\begin{enumerate*}[label={(\arabic*)}]
  \item for each permutation $\sigma$ of permutable levels, one has $\Phi(T)=\Phi(\sigma\cdot T)$;
  \item for each morphism $\delta_{N}:T\rightarrow T/N$, one has the following identification in the commutative differential graded algebra $\overline{\calC}_{W}(T)\otimes E_{W}(T/N)$:
\end{enumerate*}
\begin{equation}\label{relationcoeq}
  \bigl( \id\otimes E_{W}(\delta_{N}) \bigr)\circ \Phi(T)= \bigl( \overline{\calC}_{W}(\delta_{N})\otimes \id \bigr) \circ \Phi(T/N).
\end{equation}

We recall that $\gamma$ is the operation on leveled trees given by the formula \eqref{formulatreecomp}. The cooperadic structure
\begin{equation}\label{eq:hopf-coop-w-lambda}
  \gamma^{c}: W_{l}\calC(n_{1}+\cdots + n_{k})\longrightarrow W_{l}\calC(k) \otimes  W_{l}\calC(n_{1})\otimes \cdots \otimes W_{l}\calC(n_{k}),
\end{equation}
sends $\Phi\in W_{l}\calC(n)$ to the map $\gamma^{c}(\Phi)$ which associates to each family of leveled trees $T_{0}\in \lTree[k]$ and $T_{i}\in \lTree[n_{i}]$, with $i\leq k$, the decoration
$$
  \gamma^{c}(\Phi)(T_{0}; \{T_{i}\})= \bigl( \id \otimes \ev_{T_{0}; \{T_{i}\}}\bigr) \circ \Phi\bigl(\gamma(T_{0}; \{T_{i}\})\bigr),
$$
where the morphism
\begin{equation}\label{Eva}
  \ev_{T_{0};\{T_{i}\}} : E_{W}(\gamma(T_{0};\{T_{i}\}))\rightarrow E_{W}(T_{0})\otimes \bigotimes_{\mathclap{1\leq i\leq k}} E(T_{i})
\end{equation}
evaluates to $1$ the polynomials associated to the levels of $\gamma(T_{0};\{T_{i}\})$ corresponding to the $0$-th levels of the leveled trees $T_{i}$ with $1\leq i\leq k$.

\begin{pro}\label{OpStr}
  The family $W_{l}\calC=\{W_{l}\calC(n)\}_{n > 0}$ gives rise to a 1-reduced Hopf cooperad.
\end{pro}

\begin{proof}
  We have to check that the following diagram commutes
  \[
    \begin{tikzcd}
      W_{l}\calC(\sum_{i,j}m_{i,j}) \ar[r] \ar[d]
      & W_{l}\calC(\sum_{i} n_{i}) \otimes \underset{i\leq k,\,\, j\leq n_{j} }{\bigotimes} W_{l}\calC(m_{i,j}) \ar[d] \\
      W_{l}\calC(k) \otimes \underset{1\leq i \leq k}{\bigotimes} W_{l}\calC(\sum_{j}m_{i,j}) \ar[r]
      & W_{l}\calC(k) \otimes \underset{1\leq i \leq k}{\bigotimes} \bigl(  W_{l}\calC(n_{i}) \otimes \underset{1\leq j \leq n_{i}}{\bigotimes} W_{l}\calC(m_{i,j}) \bigr)
    \end{tikzcd}
  \]

  Let $T_{0}\in \lTree[k]$, $T_{i}\in \lTree[n_{i}]$ and $T_{i,j}\in \lTree[m_{i,j}]$.
  As explained in Section \ref{SectTreeLev}, the operation $\gamma$ is not strictly associative on leveled trees.
  However, we can easily check that the two total compositions (we refer the reader to the formula (\ref{formulatreecomp}))
  $$
    \gamma\bigl( T_{0};\bigl\{\gamma(T_{i};\{T_{i,j}\})\bigr\}_{i}\bigr)
    \hspace{25pt}\text{and}\hspace{25pt}
    \gamma\bigl( \gamma(T_{0};\{T_{i}\});\{T_{i,j}\}\bigr)
  $$
  coincide up to permutations of permutable levels.
  So, it is not strictly associative at the level of the category $\lTree[\sum_{i,j}m_{i,j}]$ but it is at the level of the resolution $W_{l}\calC$ since $\Phi$ is equivariant along permutations of permutable levels.
\end{proof}

Moreover, there is a morphism of 1-reduced Hopf
cooperads $\eta:\calC\rightarrow W_{l}\calC$ mapping an element $x\in \calC(n)$ to the map $\Phi_{x}$ which when evaluated at a leveled tree $T\in \lTree[n]$ consists in using the cooperadic structure with shape $\alpha T \in \Tree^{\geq 2}[n]$, denoted by $\gamma_{\alpha T}(x)$, and indexing all levels by the constant  polynomial form $1$.
The map so defined preserves the cooperadic structures and gives rise to a resolution of $\calC$ as proved in the next statement.

\begin{thm}\label{ThmQI}
  The morphism of 1-reduced Hopf cooperads $\eta:\calC\rightarrow W_{l}\calC$ is a quasi-isomorphism.
\end{thm}

\begin{proof}
  The proof is similar to \cite[Proposition 5.2]{FresseTurchinWillwacher2017}.  We use the splitting of non-unital CDGAs $\K[t,dt]= \K 1 \oplus \K[t,dt]_{0}$ where $\K[t,dt]_{0} \subset \K[t,dt]$ is the acyclic ideal formed by the polynomial differential forms that vanish at $t=0$. We consider a variation of the functor $E_W$ given by:
  \begin{align*}
    E'_{W} : \lTree[n] & \longrightarrow \Ch,                                  \\
    T                  & \longmapsto \bigotimes_{1\leq i\leq h(T)} \K[t,dt]_{0}.
  \end{align*}
  Let us remark that, if a $k$-th level of a leveled tree $T$ is indexed by $1 \in \K[t,dt]$, then the decoration $\Phi(T)$ is uniquely determined from $\Phi(T/\{k\})$ using the relation (\ref{relationcoeq}).
  Consequently, as chain complexes, there is a quasi-isomorphism:

  $$
    W_{l}\calC(n) \simeq \prod_{[T]\in \pi_{0} \beta\Tree^{\geq 2}[n]} \overline{\calC}_{W}(T)\otimes E'_{W}(T),
  $$
where $\beta\Tree^{\geq 2}[n]$
is the essential image of $\beta$, i.e.\ the subcategory of $\lTree[n]$ which consists of leveled trees having exactly one non-bivalent vertex in each level. The above product is over classes of leveled trees up to isomorphisms of planar trees and permutations of permutable levels. Notice that, thanks to the identity \eqref{relationcoeq}, a point in Boardman--Vogt resolution in  determined by its values on the leveled trees in $\beta\Tree^{\geq 2}[n]$.
Furthermore, if we disregard the term on the right-hand side in which $T$ is not the $n$-corolla $c_{n}$, we obtain a contractible complex. So, the product is quasi-isomorphic to $\overline{\calC}_{W}(c_{n})=\calC(n)$ and the canonical map $\eta:\calC\rightarrow W_{l}\calC$ is given by the identity on this factor.
\end{proof}

\begin{rmk}
  We can easily check that the leveled Boardman--Vogt resolution so obtained is isomorphic to the usual Boardman--Vogt resolution introduced in \cite{FresseTurchinWillwacher2017}. The arguments are the same used in the proof of Proposition \ref{ProBlB}. This gives an alternative proof of the previous theorem.

  For the moment, we do not know that the Boardman--Vogt construction gives rise to a fibrant resolution. It will be proved in Section \ref{SectPrimOp} where this construction is identified with the free operad generated by an explicit cooperad.
\end{rmk}

\paragraph{The leveled Boardman--Vogt resolution for 1-reduced Hopf $\Lambda$-cooperads}

Let $\calC$ be a 1-reduced Hopf $\Lambda$-cooperad.
In order to get a fibrant resolution of $\calC$ in the Reedy model category $\Lambda\Cooperad$, we extend the construction introduced in the previous paragraph to deal with $\Lambda$-structure. As a symmetric cosequence, we set
$$
  W_{\Lambda}\calC(n) \coloneqq W_{l}\calC_{>0}(n), \quad \text{ for all } n>0,
$$
where $\calC_{>0}$ is the underlying 1-reduced Hopf cooperad of $\calC$. The subscript $\Lambda$ is to emphasize that we work in the category of 1-reduced $\Lambda$-cooperads. By definition, $W_{\Lambda}\calC$ inherits cooperadic operations for $k,n_{1},\dots,n_{k} > 0$:
$$
  \gamma^{c}:W_{\Lambda}\calC(n_{1}+\cdots+n_{k})\longrightarrow W_{\Lambda}\calC(k)\otimes W_{\Lambda}\calC(n_{1})\otimes\cdots\otimes  W_{\Lambda}\calC(n_{k}).
$$

It suffices to define the $\Lambda$-costructure on the construction $W_{\Lambda}\calC$: For simplicity, we only build the costructure associated to the order preserving map $h[i]:[n]\rightarrow [n+1]$ skipping the $i$-th term (i.e. $h[i](j)=j$ if $j<i$ and $h[i](j)=j+1$ if $j\geq i$). We need a map of the form
\begin{equation}\label{Lambdacostructure}
  \begin{aligned}
    h[i]_{\ast}:W_{\Lambda}\calC(n) & \longrightarrow  W_{\Lambda}\calC(n+1)                                                       \\
    \Phi                            & \longmapsto h[i]_{\ast}(\Phi) \coloneqq \bigl\{h[i]\circ\Phi(T),\,\,T\in \lTree[n+1]\bigr\}.
  \end{aligned}
\end{equation}
Let $T$ be a leveled $(n+1)$-tree. In what follows, we denote by $v$ the first non-bivalent vertex composing the path from the $i$-th leaf to the root. In order to define $h[i]\circ\Phi(T)\in \overline{\calC}_{W}(T)\otimes E_{W}(T)$, there are different cases to consider:

\begin{itemize}[leftmargin=35pt]
  \item[Case $1$:] If $v$ has at least three incoming edges, then we consider the leveled $n$-tree $T'$ defined from $T$ by removing the branch leading to the $i$-th leaf.
        In that case, $h[i]\circ\Phi(T)$ is given by
        $$
          h[i]\circ\Phi(T)=(h[i]_{|v})_{\ast} \circ \Phi(T'),
        $$
        where the map $(h[i]_{|v})_{\ast}:\calC(|v|-1)\rightarrow \calC(|v|)$ is obtained using the $\Lambda$-costructure of $\calC$ applied to the restriction map $h[i]_{|v}:[|\operatorname{in}(v)|-1]\rightarrow [|\operatorname{in}(v)|]$ to the incoming edges of $T$. For instance, in the next picture, the corresponding map $h[7]_{|v}:[3]\rightarrow [4]$ is given by $h[7]_{|v}(j)=j$.

        \hspace{-22pt}\includegraphics[scale=0.3]{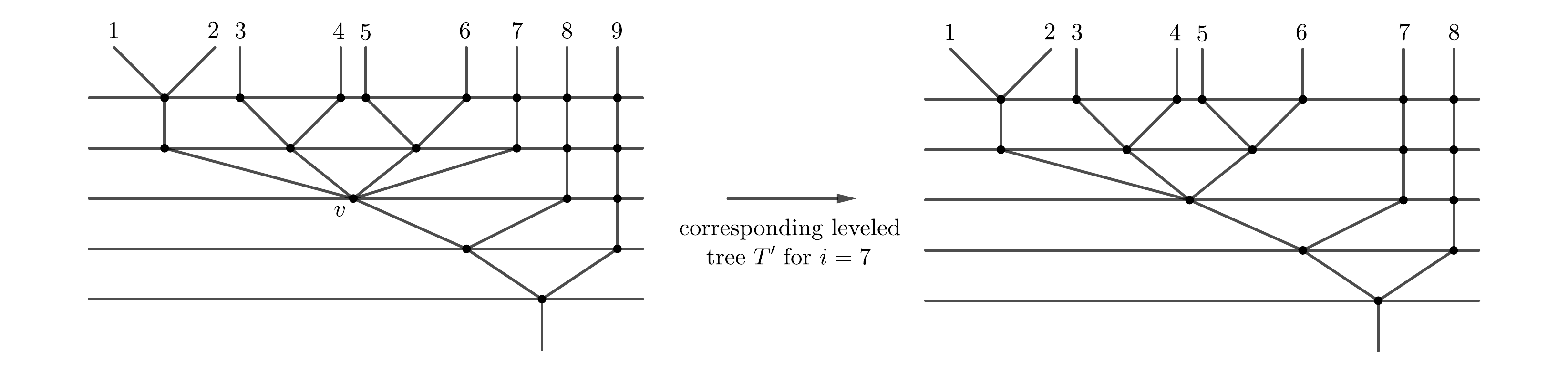}

  \item[Case $2$:] If $v$ has only two incoming edges, then we denote by $e_{i}$ the incoming edge coming from the $i$-th leaf.
        We consider $T'$ obtained from $T$ by removing the edge $e_{i}$.

  \item[Case $2.1$:] If the level $h(v)$ of $T'$ has at least one non-bivalent vertex, then $T'$ is a leveled tree and one has
        $$
          h[i]\circ\Phi(T)= b_{v}\otimes \Phi(T'),
        $$
        where $b_{v}$ is the image of $1$ by the map $\K\rightarrow \calC(2)$ induced by the $\Lambda$-costructure of $\calC$. Roughly speaking, it consists in indexing $v$ by the element $b_{v}$ and keeping the decoration of the other vertices and the levels induced by $\Phi(T')$.

        \hspace{-22pt}\includegraphics[scale=0.3]{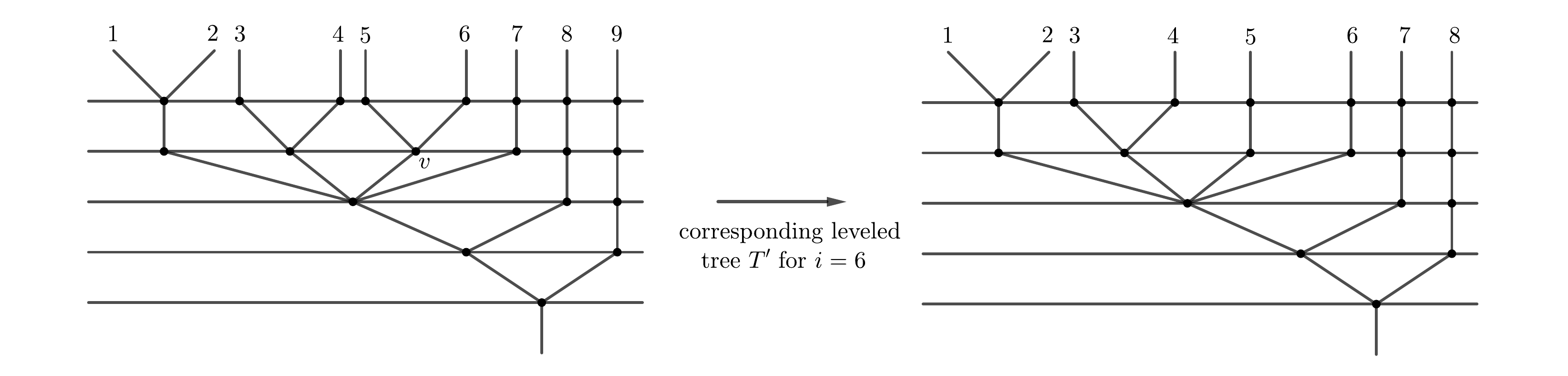}

  \item[Case $2.2$:] If $h(v)=0$ in $T$ and the level consists of a single trivalent vertex $v$, then we consider the leveled tree $T''$ obtained from $T'$ by removing the zeroth level. In that case, one has
        $$
          h[i]\circ\Phi(T)= 1\otimes b_{v}\otimes \Phi(T'').
        $$
        Roughly speaking, it consists in indexing $v$ (which is the root in that case) by the element $b_{v}$, labelling the level $1$ by $1\in \K[t,dt]$ and keeping the decoration of the other vertices and the other levels induced by $\Phi(T'')$.

        \hspace{-22pt}\includegraphics[scale=0.3]{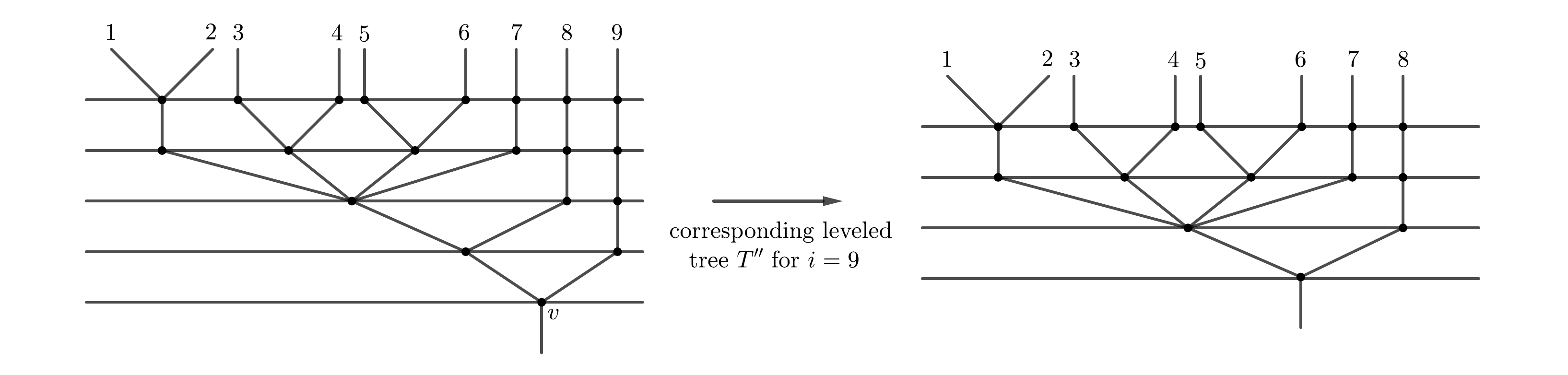}

  \item[Case $2.3$:] If $v$ is a trivalent vertex of $T$ at maximal height $h(v)=h(T)$ and all other vertices at level $h(T)$ are bivalent, then we consider the leveled tree $T''$ obtained from $T'$ by removing the section $h(v)$. In that case, one has
        $$
          h[i]\circ\Phi(T)= 1\otimes b_{v}\otimes \Phi(T'').
        $$
        Roughly speaking, it consists in indexing $v$ by the element $b_{v}$, labelling the top level by $1\in \K[t,dt]$ and keeping the decoration of the other vertices and the other levels induced by $\Phi(T'')$.

        \hspace{-22pt}\includegraphics[scale=0.3]{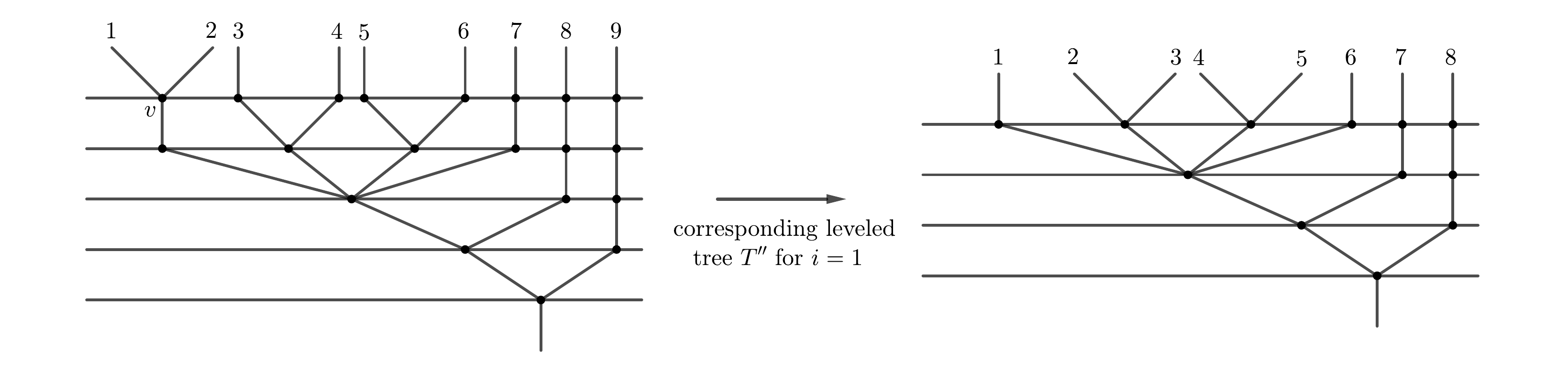}

  \item[Case $2.4$:] If $v$ is the unique non-bivalent vertex at the level $h(v)\notin \{0,h(T)\}$, then we consider the leveled tree $T''$ obtained from $T'$ by removing the section $h(v)$. In that case, one has
        $$
          h[i]\circ\Phi(T)= m^{\ast}_{h(v)}\otimes b_{v}\otimes \Phi(T''),
        $$
        where $m^{\ast}$ is the coassociative coproduct introduced in Section \ref{SectModHopf} and $m^{\ast}_{h(v)}$ is the coproduct applied to the polynomial form associated to the $h(v)$-th level of the leveled tree $T''$.

        \hspace{-22pt}\includegraphics[scale=0.3]{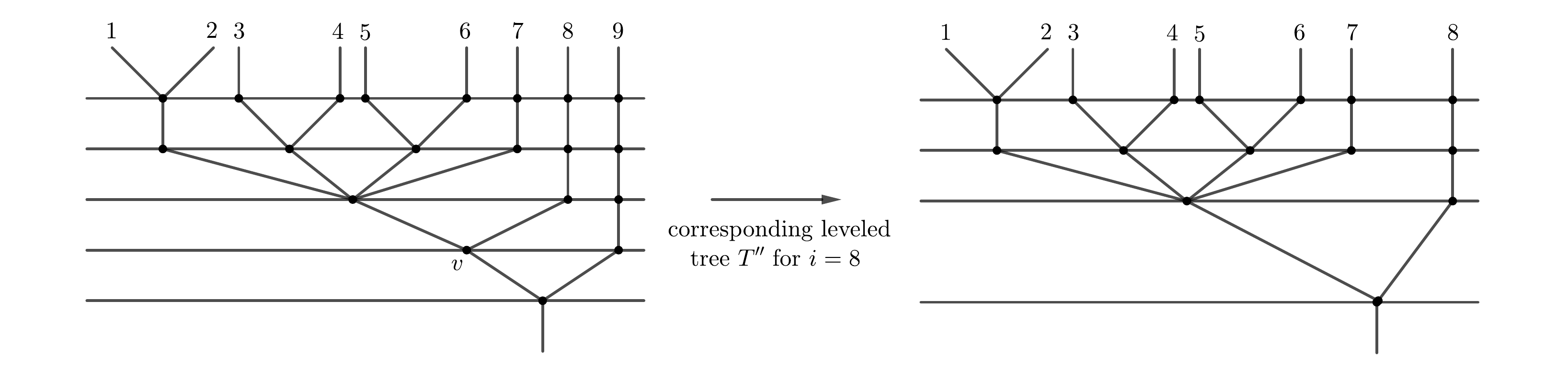}
\end{itemize}

In short, we have
 the following result:

\begin{pro}\label{pro:w-lambda}
  The $\Lambda$-costructure \eqref{Lambdacostructure} makes the 1-reduced Hopf cooperad $W_{\Lambda}\calC$ into a 1-reduced Hopf $\Lambda$-cooperad.
  Furthermore, the morphism $\eta:\calC \rightarrow W_{\Lambda}\calC$ introduced in Theorem \ref{ThmQI} is a quasi-isomorphism of 1-reduced Hopf $\Lambda$-cooperads.
\end{pro}

\begin{proof}

  We simply need to check that the $\Lambda$-costructure is compatible with the Hopf cooperad structure and the morphism $\eta$ (compare with the proof of \cite[Proposition 5.11]{FresseTurchinWillwacher2017}, where the different cases are completely analogous to ours).
  Notice that the cooperadic structure of $W_l \calC$ is defined using the evaluation at $t=1$, which is compatible with the coassociative coproduct $m^*$ as defined in Section~\ref{SectModHopf}.

  The compatibility with $\eta$ follows from $m^*(1) = 1 \otimes 1$.
\end{proof}

\paragraph{Simplicial frame}
Let us now introduce a simplicial frame (see~\cite[Section~3.2]{Fre}) of $W_{\Lambda}\calC$.
If $X \in \bfC$ is an object of some model category, then a simplicial frame of $X$ is a simplicial object $X^{\Delta^{\bullet}} \in s\bfC$ such that
\begin{enumerate}
  \item the zeroth object is $X$, i.e.\ $X^{\Delta^{0}} = X$;
  \item the iterated degeneracy $X^{\Delta^{0}} \to X^{\Delta^{d}}$ is a weak equivalence for all $d \ge 0$;
  \item the product of the vertex maps $X^{\Delta^{d}} \to \prod_{k=0}^{d} X^{\Delta^{0}} = X^{\operatorname{sk}_{0}(\Delta^{d})}$ is a Reedy fibration in $s\bfC$.
\end{enumerate}

\begin{rmk}
  Despite the notation, $X^{\Delta^{\bullet}}$ is not always obtained as the cotensoring of $X$ by $\Delta^{\bullet}$.
\end{rmk}

Our simplicial frame is inspired by the one in~\cite[Section~5.3]{FresseTurchinWillwacher2017}, and we will generalize it to Hopf $\Lambda$-cobimodules in Section~\ref{sec:fibr-resol-hopf-2}.
We consider an extension of the functor $E_{W}$ from Equation~\eqref{eq:e-w}.
Recall that for $d \ge 0$, the CDGA $\Omega^{*}_{\PL}(\Delta^{d})$ of polynomial forms on $\Delta^{d}$ is
\[ \Omega^{*}_{\PL}(\Delta^{d}) = \K[t_{0}, \dots, t_{d}, dt_{0}, \dots, dt_{d}] / (t_{0} + \dots + t_{d} = 1, \; dt_{0} + \dots + dt_{d} = 0). \]
(In particular $\Omega^{*}_{\PL}(\Delta^{1})$ is isomorphic to $\K[t,dt]$.)
For $n > 0$ (the arity) and $d \ge 0$ (the simplicial degree), we define:
\begin{equation*}
  E_{W}^{\Delta^{d}} : \lTree[n] \to \CDGA, \quad
  T \mapsto \bigotimes_{1 \le i \le h(T)} \K[t,dt] \otimes \bigotimes_{0 \le j \le h(T)} \Omega^{*}_{\PL}(\Delta^{d}).
\end{equation*}

Informally, each step between two levels will be decorated by a polynomial from $\K[t,dt]$, and each level will be decorated by a polynomial form on $\Delta^{d}$.
Let us now describe the functor $E_{W}^{\Delta^{d}}$.
Isomorphisms of trees act in the obvious way.
Contractions of consecutive levels act as in Section~\ref{SectBVCoop} on the $\K[t,dt]$ factors, and multiply the corresponding forms on $\Delta^{d}$ together.
Finally, permutations of permutable levels act as before on $\K[t,dt]$ and swap the corresponding $\Omega^{*}_{\PL}(\Delta^{d})$ factors.

It is clear that $E_{W}^{\Delta^{\bullet}}$ inherits a simplicial structure from the one of $\Omega^{*}_{\PL}(\Delta^{\bullet})$.
The simplicial frame is then:
\[ W_{\Lambda}^{\Delta^{\bullet}}\calC(n) \coloneqq \int_{T \in \lTree[n]} \overline{\calC}_{W}(T) \otimes E_{W}^{\Delta^{\bullet}}(T). \]

\begin{pro}\label{prop:simplicial-frame-cooperad}
  The simplicial object $W_{\Lambda}^{\Delta^{\bullet}}\calC$ defines a simplicial frame of $W_{\Lambda}\calC$.
\end{pro}
\begin{proof}
  The cooperad structure on each $W_{\Lambda}^{\Delta^{d}}\calC$ only involves the decorations between the levels and is identical to Equation~\eqref{eq:hopf-coop-w-lambda}; on the decorations between the levels (i.e.\ the PL forms on $\Delta^{d}$) we simply use the identity.
  The $\Lambda$-costructure is also similar to Equation~\eqref{Lambdacostructure}, and we just take the decoration $1 \in \Omega^{*}_{\PL}(\Delta^{d})$ for the decorations between the new levels.
  It is then straightforward to adapt the previous proofs to show that $W^{\Delta^{d}}_{\Lambda}\calC$ is a 1-reduced Hopf $\Lambda$-cooperad.

  Let us now check that it is a simplicial frame for $W_{\Lambda}\calC$.
  Since $\Omega^{*}_{\PL}(\Delta^{0}) = \K$, we clearly have $W_{\Lambda}^{\Delta^{0}}\calC = W_{\Lambda}\calC$.
  To check that the iterated degeneracies $W_{\Lambda}\calC \to W_{\Lambda}^{\Delta^{d}}\calC$ are quasi-isomorphisms, we can define an explicit homotopy (inspired by the contracting homotopy for $\Omega^{*}_{\PL}(\Delta^{d}) \simeq \K$) to contract $W_{\Lambda}^{\Delta^{d}}\calC$ onto $W_{\Lambda}\calC$ (compare with~\cite[Lemma~5.9]{FresseTurchinWillwacher2017}).

  By definition, checking that $W^{\Delta^{d}}_{\Lambda}\calC \to W^{\operatorname{sk}_{0}\Delta^{d}}_{\Lambda}\calC$ is a Reedy fibration is equivalent to checking that $W^{\Delta^{d}}_{\Lambda}\calC \to W^{\partial \Delta^{d}}_{\Lambda}\calC$ is a fibration (where the matching object $W^{\partial \Delta^{d}}_{\Lambda}\calC$ is defined like $W^{\Delta^{d}}_{\Lambda}\calC$ except that we replace $\Delta^{d}$ by its boundary in $E_{W}^{\Delta^{d}}$).
  We can adapt the proofs of the next subsection
  to show that both $W^{\Delta^{d}}_{\Lambda}\calC$ and $W^{\partial \Delta^{d}}_{\Lambda}\calC$ are cofree as graded cooperads, generated by primitive elements.
  Since the restriction map $\Omega^{*}_{\PL}(\Delta^{d}) \to \Omega^{*}_{\PL}(\partial \Delta^{d})$ is surjective, the map $W^{\Delta^{d}}_{\Lambda}\calC \to W^{\partial \Delta^{d}}_{\Lambda}\calC$ is surjective on cogenerators, therefore it is a fibration in Fresse's model structure.
\end{proof}

\subsection{The operad of primitive elements}
\label{SectPrimOp}

Let $\calC$ be a 1-reduced Hopf cooperad.
A point in $W_{l}\calC$ is said to be \textit{primitive} if its image through the cooperadic operations is $0$.
By definition of the cooperadic structure on $W_{l}\calC$, an element $\Phi\in W_{l}\calC(n)$ is primitive if and only if for each leveled $n$-tree $T$ and each level $i\in \{1,\ldots, h(T)\}$, the evaluation of the polynomial differential form $p_{i}(t,dt)$ to $1$ is $0$.
Hence, the decoration that $\Phi$ assigns to a leveled tree must be so that the level decorations belong to the subspace $\K[t,dt]_{1}\subset \K[t,dt]$ of polynomial differential forms that vanish at the endpoint $t=1$.
For this reason, we introduce the functor
\begin{align*}
  E''_{W} : \lTree[n] & \longrightarrow \Ch,                                  \\
  T                   & \longmapsto \bigotimes_{1\leq i\leq h(T)} \K[t,dt]_{1}.
\end{align*}

As we will see in Proposition \ref{prop:primitives are operad}, the sequence of primitive elements of $W_{l}\calC$ forms an operad.

\begin{defi}\label{def:w-l-cofree}
  The subspace of primitive elements associated to the 1-reduced cooperad $\calC$, denoted by $\Prim(W_{l}\calC)$, is the end
  $$
    \Prim(W_{l}\calC)(n) \coloneqq \int_{T\in \lTree[n]} \overline{\calC}_{W}(T) \otimes E''_{W}(T).
  $$
\end{defi}

The proof of the following lemma is similar to~\cite[Lemma~5.3]{FresseTurchinWillwacher2017}.

\begin{lmm}\label{Lmm1}
  As a graded cooperad, $W_{l}\calC$ is the cofree cooperad generated by $\Prim(W_{l}\calC)=\{\Prim(W_{l}\calC)(n)\}$.
  In particular, $W_{l}\calC$ and $W_{\Lambda}\calC$ are fibrant in the model structure of Theorem~\ref{thm:model-struct-cooperads}.
\end{lmm}

\begin{proof}
  Explicitly, we show that $W_{l}\calC$ is isomorphic to $\calF^{c}_{l}(\Prim(W_{l}\calC))$ where $\calF^{c}_{l}$ is the leveled cofree cooperad functor using leveled trees introduced in the Section~\ref{SectBarCoop}.
  An element in the cofree cooperadic object is a map $\omega$ mapping a leveled tree to a decoration of the vertices by elements in $\Prim(W_{l}\calC)$:
  $$
    \omega(T)\in \bigotimes_{v\in V(T)} \Prim(W_{l}\calC)(|v|), \qquad \text{for any leveled tree } T \in \lTree_{\iso}.
  $$

  In order to construct a morphism of sequences $\phi:W_{l}\calC\rightarrow \Prim(W_{l}\calC)$, let us recall that, thanks to the identity \eqref{relationcoeq}, any point in the Boardman--Vogt resolution is determined by its values on the leveled trees in the essential image

  $\beta\Tree$ having exactly one non-bivalent vertex in each level. The same is true for the subspace of primitive elements. Then we introduce a natural transformation $\pi:E\Rightarrow E''$ which carries any polynomial differential form $p(t,dt)\in \K[t\, ,\,dt]$ to the polynomial $\widetilde{p}(t,dt)=p(t,dt)-tp(1,0)$.  In particular, one has $\widetilde{p}(t,dt) \in \K[t,dt]_{1}$, and moreover $\ev_{t=0}\widetilde{p}(t,dt)=\ev_{t=0}p(t,dt)$ (so the new element also satisfies the equations defining $\Prim(W_{l}\calC)$ as an end). For any $\Phi\in W_{l}\calC$ and $T\in \beta\Tree^{\geq 2}$, one has
  $$
  \phi(\Phi)(T)=(\id\otimes E)\circ \Phi(T).
  $$
According to the universal property of the cofree cooperadic object, one has a morphism of graded cooperads
  $$
    \psi:W_{l}\calC\longrightarrow \calF^{c}_{l}(\Prim(W_{l}\calC)).
  $$
  The injectivity of our morphism $\psi$ on the primitive elements on the source and the fact that $W_{l}\calC$ is conilpotent implies that $\psi$ is injective itself.
  We are therefore left with proving that $\psi$ is surjective.

  Let $\omega$ be an element of $\calF^{c}_{l}(\Prim(W_{l}\calC))$.

  Thanks to Corollary~\ref{cor:cofree-as-product}, we can view $\omega = \{ \omega_{T'} \in \overline{\calC}_{W}(T') \otimes E''_{W}(T') \}_{T'}$

  as a collection indexed by isomorphism classes of reduced planar trees $T' \in \Tree_{\iso}^{\ge2}[n] / \cong$.
  We want to define an element $\Phi \in W_{l}\calC$ such that $\psi(\Phi) = \omega$.
  Let $T \in \lTree[n]$ and let us define $\Phi(T)$.
  The leveled tree $T$ defines an isomorphism class of reduced planar trees $[T] \in \Tree_{\iso}^{\ge2}[n] / \cong$.
  We can thus define $\Phi(T) \in \overline{\calC}_{W}(T) \otimes E_{W}(T)$ as follows.
  \begin{itemize}
    \item Let $v \in V(T)$ be a vertex with $|v| \ge 2$.
          Then $v$ corresponds to a unique vertex in $[T]$.
          We decorate it in $\overline{\calC}_{W}(T)$ with the decoration of $v$ in $\omega_{[T]}$.
    \item If a level $1 \le i \le h(T)$ is permutable, then we decorate it with $p(t,dt) = t$.
          If the level $i$ is not permutable, then it corresponds to a unique edge in $[T]$, and we decorate the level with the decoration of the corresponding edge in $\omega_{[T]}$.
  \end{itemize}
  We can then check that $\Phi$ defines an element of $W_{l}\calC$, thanks to the conditions on $\omega \in \calF^{c}_{l}(\Prim(W_{l}\calC))$ and the various $\omega_{T'} \in \Prim(W_{l}\calC)(T')$.
  It is also clear that $\psi(\Phi) = \omega$.
\end{proof}

\begin{pro}\label{prop:primitives are operad}
  The sequence $\Prim(W_{l}\calC)=\{\Prim(W_{l}\calC)(n)\}$ is a 1-reduced dg-operad shifted in degree by one (i.e.\ $\Sigma^{-1} \Prim W_l \calC$ is an operad).
\end{pro}

\begin{proof}
  The operadic composition
  $$
    \gamma: \Sigma^{-1} \Prim(W_{l}\calC)(k) \otimes \Sigma^{-1} \Prim(W_{l}\calC)(n_{1}) \otimes \cdots \otimes \Sigma^{-1} \Prim(W_{l}\calC)(n_{k}) \longrightarrow \Sigma^{-1} \Prim(W_{l}\calC)(n_{1}+\cdots +n_{k})
  $$
  is defined using the operation $\gamma$ introduced in Section \ref{SectTreeLev}.
  Indeed, let $\Phi_{0}$ and $\Phi_{i}$ be elements in $\lTree[k]$ and $\lTree[n_{i}]$, respectively, with $i\leq k$.
  In order to define $\Phi=\gamma(\Phi_{0}, \{\Phi_{i}\})$ there are three cases to consider.
  First, if the leveled tree $T$ is of the form $\gamma(T_{0}, \{T_{i}\})$, then one has
  $$
    \Phi(T)=\Phi_{0}(T_{0}) \otimes \bigotimes_{i\in I} \bigl( \Phi_{i}(T_{i}) \otimes dt \bigr).
  $$
  Secondly, if $T$ is of the form $T' = \gamma(T_{0}, \{T_{i}\})$ up to permutations of permutable levels and contractions of permutable levels, then the decoration of $T$ is given by the decoration of $T'$ composed with the corresponding morphisms of permutations and contractions, with all the new levels decorated by $dt$.
  Finally, if $T$ is not of the form $\gamma(T_{0}, \{T_{i}\})$, then $\Phi(T)=0$.
  The reader can easily check that the operations so obtained are well defined and satisfy the operadic axioms.
\end{proof}

\begin{rmk}
  In general, the component $\Prim(W_{l}\calC)(n)$ does not inherit the structure of a CDGA.
  Indeed, the product of two primitive elements is not necessarily primitive.
\end{rmk}

\begin{thm}\label{thm:W=B for bimods}
  Let $\calC$ be a 1-reduced Hopf cooperad.
  The leveled Boardman--Vogt resolution $W_{l}\calC$ is isomorphic (as a 1-reduced dg-operad) to the leveled bar construction of the shifted operad of its primitive elements:
  $$
    W_{l}\calC \cong \calB_{l}(\Sigma^{-1} \Prim(W_{l}\calC)).
  $$
\end{thm}

\begin{proof}
  Given Lemma \ref{Lmm1}, we simply check that the differentials agree.
  Consider some element $\Phi = \{ \Phi(T) \in \overline{\calC}_W(T) \otimes E_W(T) \}_{T \in \lTree[n]} \in W_l\calC$.
  We must check that $(d_{\mathrm{int}}+d_{\mathrm{ext}})\psi(\Phi) = \psi(d\Phi)$, where $\psi : W_l \calC \to \calF^c_l(\Prim(W_l C))$ is the morphism of cooperads coinduced by $p(t,dt) \mapsto p(t,dt) - t \cdot p(1,0)$ on decorations of levels.

  Since $\psi$ is a morphism and differentials are coderivations, it is sufficient to check that $d\psi = \psi d$ when projected down to cogenerators.
  The internal differential $d_{\mathrm{int}} \psi(\Phi)$ (i.e.\ the action on decorations of vertices and the levels) simply correspond to the differential acting on $\overline{\calC}_W(T)$ and to the $dp$ part of $d(p(t,dt) - p(1,0) \cdot t)$.
  The external differential $d_{\mathrm{ext}} \psi(\Phi)$ contracts consecutive levels using the operadic structure defined in Proposition~\ref{prop:primitives are operad}.
  This corresponds to the $dt$ part of the differential acting decorations of levels $p(t,dt) - t \cdot p(1,0) \in E_W(T) = \bigotimes_{i=1}^{h(T)} \K[t,dt]$ thanks to the description of the operadic structure.
\end{proof}

\subsection{The Boardman--Vogt resolution and the bar-cobar construction}\label{SectCobarCDGA}

This section is split into three parts.
First, we introduce an alternative description of the cobar construction for 1-reduced cooperads.
Then, we show that the bar-cobar construction of a 1-reduced cooperad is quasi-isomorphic to its Boardman--Vogt resolution.
Finally, we extend this result to 1-reduced $\Lambda$-cooperads by introducing a $\Lambda$-costructure on its bar-cobar construction.

\paragraph{The leveled cobar construction for 1-reduced cooperads}

Dually to Section \ref{SectBarCoop}, we introduce alternative versions of the free operad functor and the cobar construction using the category of leveled trees.
Then, we compare these two definitions with the usual ones.
Following the notation introduced in Section \ref{SectBarCoop}, we consider the functor

$$
  \calF_{l}:\dg\Sigma \Seq_{>1}\longrightarrow \Sigma\Operad,
$$
from the category of symmetric sequences of chain complexes to 1-reduced operads. For each sequence $X\in \dg\Sigma \Seq_{>1}$, we consider the two functors
\begin{align*}
  \overline{X}_{F} : \lTree_{\iso}[n] & \longrightarrow \Ch , \qquad T \longmapsto \bigotimes_{v\in V(T)} X(|v|); \\
  E_{1} : \lTree_{\iso}[n]^{op}       & \longrightarrow \Ch, \qquad T \longmapsto \K.
\end{align*}
\begin{defi}
  The leveled free operad functor $\calF_{l}$ is defined as the coend
  $$
    \calF_{l}(X)(n) \coloneqq \int^{T\in \lTree_{\iso}[n]} \overline{X}_{F}(T)\otimes E_{1}(T).
  $$
\end{defi}
An element in $\calF_{l}(X)(n)$ is the data of a leveled tree $T$ together with a family $\{x_{v}\}$, with $v\in V(T)$, of elements in the symmetric sequence $X$.
Such an element is denoted by $[T;\{x_{v}\}]$. The operadic structure defined by
\begin{align*}
  \gamma' : \calF_{l}(X)(k) \otimes \calF_{l}(X)(n_{1})\otimes \cdots \otimes \calF_{l}(X)(n_{k}) & \longrightarrow \calF_{l}(X)(n_{1}+\dots+ n_{k}),                                                    \\
  [T_{0}; \{x_{v}^{0}\}] \otimes \bigl\{ [T_{i}; \{x_{v}^{i}\}] \bigr\}_{1\leq i\leq k}           & \longmapsto \biggl[ \gamma(T_{0} \{T_{i}\}); \{x_{v}^{i}\}_{0\leq i \leq k}^{v\in V(T_{i})} \biggr],
\end{align*}
where $\gamma$ is the operation (\ref{formulatreecomp}), is well defined since a point $\Phi$ is an equivalence class up to permutations and contractions of permutable levels.

\begin{defi}[The usual free functor for operads]\label{deffreeop}
  In order to define the usual free operad functor $\calF$, we use the category $\Tree_{\iso}^{\geq 2}[n]$ introduced in Section \ref{sec:categ-plan-trees}. The morphisms are just isomorphisms of rooted planar trees. For any sequence $X\in \dg\Sigma \Seq_{>1}$, we consider the functor
  $$
    \overline{X}_{u}:\Tree^{\geq 2}_{\iso}[n]\longrightarrow \Ch, \quad T \longmapsto \bigotimes_{v\in V(T)} X(|v|).
  $$
  The free operad functor is defined as the coend
  $$
    \calF(X)(n) \coloneqq \int^{T\in \Tree^{\geq 2}_{\iso}[n]} \overline{X}_{u}(T) \otimes E_{1}(T).
  $$
  A point in the free operad $\calF(X)(n)$ is denoted by $\langle T; \{x_{v}\} \rangle$.
\end{defi}

\begin{pro}\label{FreeFunc2}
  The functor $\calF_{l}$ is isomorphic to the usual free operad functor denoted by $\calF$. In particular, it means that $\calF_{l}$ is the right adjoint to the forgetful functor.
\end{pro}

\begin{proof}
  By using the comparison morphisms $\alpha$ and $\beta$ between the categories $\Tree^{\geq 2}_{\iso}$ and $\lTree_{\iso}$ introduced in Section \ref{CompTree}, we are now able to give an explicit isomorphism between the leveled and usual versions of the free operad functor:
  \begin{align*}
    L_{n}:\calF_{l}(X)(n) & \longrightarrow \calF(X)(n),
                          & R_{n}:\calF(X)(n)                                 & \longrightarrow \calF_{l}(X)(n),  \\
    [T; \{x_{v}\}]        & \longmapsto \langle \alpha(T); \{x_{v}\} \rangle;
                          & \langle T; \{x_{v}\} \rangle                      & \longmapsto [\beta(T);\{x_{v}\}].
  \end{align*}
  The map $R_{n}$ does not depend on the chosen point in $\alpha^{-1}(T)$ since the elements in $\calF_{l}(X)(n)$ are equivalence classes up to contractions of permutable levels and permutations of permutable levels. So, the maps $L_{n}$ and $R_{n}$ are well defined and provide an isomorphism preserving the operadic structures.
\end{proof}

\begin{defi}\label{def:leveled-cobar}
  The leveled cobar construction of a 1-reduced cooperad $\calC$ is the operad given by
  $$
    \Omega_{l}(\calC) \coloneqq \bigl( \calF_l(\Sigma^{-1} \calU(\overline{\calC})), d_{\mathrm{int}} + d_{\mathrm{ext}} \bigr),
  $$
  where $\calU(\overline{\calC})$ is the underlying symmetric sequence of the coaugmentation quotient of $\calC$.
  The degree of an element $[T; \{x_{v}\}]$ is the sum of the degrees of the elements indexing the vertices, with at least two incoming edges, minus $1$:
  $$
    \deg([T; \{x_{v}\}]) = \sum_{\mathclap{v\in V_{\ge2}(T)}} (\deg(x_{v}) - 1),
  $$
  where $V_{\ge 2}(T)$ is the set of vertices which have at least two incoming edges.
  The differential $d_{\mathrm{int}}$ is the differential corresponding to the differential algebra $\mathcal{U}(\overline{\calC})$.
  The differential $d_{\mathrm{ext}}$ consists in splitting a level into two consecutive levels using the cooperadic structure of $\calC$ on one of the vertices that have at least two incoming edges of that level (in all possible ways).
  The operadic structure is induced from the free operad $\calF_{l}(\mathcal{U}(\overline{\calC}))$.
\end{defi}

\begin{pro}\label{pro-cobar-bimod-well-def}
  The leveled cobar construction $\Omega_{l}\calC$ of a 1-reduced dg-cooperad is a well defined 1-reduced dg-operad.
\end{pro}

\begin{proof}
  The proof is essentially dual to Proposition~\ref{prop:bar-well-def}.
  The differential $d_{\mathrm{ext}}$ is the unique derivation induced by the map $\alpha : \Sigma^{-1}\calU(\overline{\calC}) \to \calF_{l}(\Sigma^{-1} \calU(\overline{\calC}))$ which sends an element to the sum of all possible applications of the cooperad structure maps (indexing two-leveled trees with exactly two vertices with $\ge2$ incoming edges).
  The fact that the differential squares to zero is again similar to the case of the classical cobar construction using the coassociativity of the cooperadic structure of $\calC$.
\end{proof}

\begin{defi}[The usual cobar construction for dg-cooperads] \label{def:usual-cobar-dg}
  Thanks to Definition \ref{deffreeop}, the usual cobar construction $\Omega(\calC)$ is defined as the quasi-free operad
  $$
    \Omega(\calC) \coloneqq \bigl( \calF(\Sigma^{-1} \calU(\overline{\calC})), \; d_{\mathrm{int}}+d_{\mathrm{ext}}\bigr)
  $$
  generated by the coaugmentation quotient of $\calC$ in which the degree of an element is the degree of the decorations minus the number of vertices.
  The differential is composed of the internal differential coming from the  differential graded algebra $\calU(\overline{\calC})$ and an external differential splitting a vertex using the cooperadic structure of $\calC$.
\end{defi}

\begin{pro}\label{ProBlBCoop}
  Let $\calC$ be a 1-reduced dg-cooperad. The leveled and usual cobar constructions are isomorphic:
  $$
    \Omega_{l}(\calC)\cong \Omega(\calC).
  $$
\end{pro}

\begin{proof}
  Thanks to Proposition~\ref{FreeFunc2}, we know that $\Omega_l(\calC)$ is isomorphic to $\Omega(\calC)$ as a graded operad, using the comparison morphisms $\alpha$ and $\beta$.
  We thus only need to check that the differentials are compatible with this isomorphism.
  In both cases, the differential is given by the sum $d_{\mathrm{int}} + d_{\mathrm{ext}}$, where the internal differential acts on decorations by elements of $\calU(\overline{\calC})$ and the external differential acts by splitting either a vertex or a level using the cooperadic structure.
  It is clear that the internal differentials match.
  For the external differentials, this can also be seen easily using the representatives given by the trees of the form $\beta(T)$ that are used for the definition of the isomorphism $R$ in Proposition~\ref{FreeFunc2}: in such a representative, there is only one non-bivalent vertex per level, so splitting a level is equivalent to splitting a vertex.
\end{proof}

\paragraph{Comparison with the Boardman--Vogt resolution for 1-reduced ($\Lambda$-)cooperads}

Let $\calC$ be a 1-reduced Hopf cooperad.
In Section~\ref{SectBVCoop} we built a fibrant resolution of $\calC$ using the leveled Boardman--Vogt resolution $W_{l}\calC$.
In Section~\ref{SectBarCoop} and in the previous paragraph, we introduced leveled versions of the bar and cobar constructions, respectively, which are isomorphic to the usual constructions.
In the following, we show that $W_{l}\calC$ is quasi-isomorphic to the bar-cobar construction of $\calC$. Namely, we build an explicit quasi-isomorphism of 1-reduced dg-cooperads:
$$
  \Gamma:\calB_{l}\Omega_{l}(\calC)\longrightarrow W_{l}\calC.
$$
This map will essentially be dual to the map of Proposition~\ref{prop:bw-cobarbar}.

A point in the bar-cobar construction $\calB_{l}\Omega_{l}(\calC)(n)$ is a family of elements $\Phi=\{\Phi(T)\in \overline{\Omega_{l}(\calC)}(T)\}_{T\in \lTree_{\iso}[n]}$, indexed by the set of leveled trees $\lTree_{\iso}[n]$, and satisfying the following relations:
\begin{itemize}
  \item for each permutation $\sigma$ of permutable levels, one has $\Phi(T)=\Phi(\sigma\cdot T)$;
  \item for each morphism $\delta_{i}:T\rightarrow T/\{i\}$ contracting two permutable levels, one has $\Phi(T)=\Phi(T/\{i\})$.
\end{itemize}

For each leveled tree $T$, the element $\Phi(T) \in \overline{\Omega_{l}(\calC)}(T)$ is the data of a family of leveled trees $\{T[v]\}_{v\in V(T)}$ indexing the vertices of the main leveled tree $T$, and a family $\{x_{u}[v]\}_{v \in V(T), u \in V(T[v])}$ of elements in the cooperad $\calC$ labelling the vertices of the sub-leveled trees $T[v]$.
We will explicitly write $\Phi(T)$ as $\{[\{T[v]\},\{x_{u}[v]\}]\}_{v\in V(T)}$.

Let $\Phi \in \calB_{l}\Omega_{l}(\calC)$ be a point in the bar-cobar construction.
In order to define
\[\Gamma(\Phi) \coloneqq \{\Gamma(\Phi)(T)\in H(T)\otimes \overline{\calC}(T), \; T\in \lTree[n]\}, \]
there are two cases to consider.
First, if there is no leveled tree $T'\in \lTree[n]$ such that $T$ is of the form $\gamma_{T'}(\{T'[v]\})$ up to permutations of permutable levels and contractions of permutable levels, then $\Gamma(\Phi)(T)=0$.
Otherwise, the collection $\{x_{u}[v]\}_{v \in V(T'), u \in V(T'[v])}$ defines an element of $\overline{\calC}(T) = \overline{\calC}(\gamma_{T'}(\{T'[v]\}))$, and we can thus define:
$$
  \Gamma(\Phi)(T)=\{p_{i}\} \otimes \{x_{u}[v]\} \in H(T) \otimes \overline{\calC}(T),
$$
where
the $i$-th level in $\gamma_{T'}(\{T'[v]\})$ is indexed by the constant polynomial form $p_{i}(t,dt)=1$ if this level corresponds to the $0$-th level of one of the leveled sub-trees $T'[v]$, and otherwise by the form $p_{i}(t,dt) = dt$.

\begin{pro}\label{ProCompaOp}
  The map $\Gamma:\calB_{l}\Omega_{l}(\calC)\rightarrow W_{l}\calC$ so defined is a weak equivalence of 1-reduced dg-cooperads.
\end{pro}

\begin{proof}
  Recall that $W_{l}\calC \cong \calB_{l}(\Sigma^{-1}\Prim W_{l}\calC)$ (see Section~\ref{SectPrimOp}).
  The map $\Gamma$ defined above is induced by the morphism of operads $\Omega_{l}\calC \to \Sigma^{-1}\Prim W_{l}\calC$ which decorates all (external) levels by $dt$, therefore it is a cooperad morphism.
  The weak equivalence is a consequence of the commutative diagram:
  \[
    \begin{tikzcd}
      \calC \ar[r, "\simeq"] \ar[d, "\simeq" swap] & W_{l}\calC \\
      \calB_{l}\Omega_{l}(\calC) \ar[ru, "\Gamma" swap] &
    \end{tikzcd}
  \]
  \vspace{-3em}
\end{proof}

\begin{thm}\label{thm:lambda-compatible}
  Let $\calC$ be a 1-reduced Hopf $\Lambda$-cooperad.
  There exists a $\Lambda$-costructure on the leveled bar-cobar construction $\calB_{l}\Omega_{l}(\calC)$ making the map $\Gamma:\calB_{l}\Omega_{l}(\calC)\rightarrow W_{\Lambda}\calC$, introduced in Proposition \ref{ProCompaOp}, into a quasi-isomorphism of 1-reduced dg-$\Lambda$-cooperads.
\end{thm}

\begin{proof}
  Let $h[i]:[n]\rightarrow [n+1]$ be the
  order preserving map given by $h[i](j)=j$ if $j<i$ and $h[i](j)=j+1$ if $j\geq i$.
  First, we introduce the following operations:
  \begin{equation}\label{Finaleq}
    h[i]_{\ast}:\Omega_{l}(\calC)(n)\longrightarrow \Omega_{l}(\calC)(n+1).
  \end{equation}
  They do not provide a $\Lambda$-costructure on the leveled cobar construction~\cite[Example 2.6]{FresseTurchinWillwacher2017} but they are useful to define a $\Lambda$-costructure on the bar-cobar construction.

  For $T \in \lTree[n]$, we define the set $T[i]\subset\lTree[n+1]$ which consists of leveled $(n+1)$-trees $T'$ such that $T$ can be obtained from $T'$ by removing the $i$-th leaf and levels composed of bivalent vertices (see the pictures following Equation~\eqref{Lambdacostructure}).
  For $T' \in T[i]$, we denote by $v_{i}\in V(T')$ the first non-bivalent vertex on the path joining the $i$-th leaf to the root.
  Let $\underline{x}=[T;\;\{x_{v}\}]$ be an element in $\Omega_{l}(\calC)(n)$.
  The element $\eta_{i}(\underline{x}, T')\in  \Omega_{l}(\calC)(n+1)$ is defined as follows:
  \begin{enumerate}
    \item If $|v_{i}|=2$, then $\eta_{i}(\underline{x}, T') \coloneqq [T';\; \{x_{v}\}\otimes \{b_{v_{i}}\}]$ is obtained by labelling the vertex $v_{i}$ by $b_{v_{i}}$ (which is the image of $1$ by the map $\K\rightarrow \calC(2)$ induced by the $\Lambda$-costructure on $\calC$) and keeping the decorations from $T$ for the other vertices.

    \item If $|v_{i}|\geq 3$, then $\eta_{i}(\underline{x}, T') \coloneqq [T';\; (h[i]_{|v_{i}})_{\ast}(x_{v_{i}}) \otimes \{x_{v}\}_{v\neq v_{i}}]$ where $h[i]_{|v_{i}}:|v_{i}-1|\rightarrow |v_{i}|$ is the injection induced by $h[i]$ on the incoming edges of $v_{i}$.
  \end{enumerate}

  The map \eqref{Finaleq} is given by
  $$
    h[i]_{\ast}(\underline{x}) \coloneqq \sum_{T'\in T[i]} \eta_{i}(\underline{x}, T').
  $$

  Now we are able to build the $\Lambda$-costructure on the leveled bar-cobar construction
  \begin{align*}
    h[i]:\calB_{l}\Omega_{l}(\calC)(n) & \longrightarrow  \calB_{l}\Omega_{l}(\calC)(n+1);                                                                                     \\
    \Phi                               & \longmapsto h[i]_{\ast}(\Phi) = \bigl\{ h[i]_{\ast}(\Phi)(T)_{v}\in \Omega_{l}(\calC)(|v|) \bigr\}_{T \in \lTree[n+1], \; v\in V(T)}.
  \end{align*}

  Let $T$ be a leveled $(n+1)$-tree and $v_{i}\in V(T)$ be the first non-bivalent vertex on the path joining the $i$-th leaf to the root. Let $\epsilon_{T,i}$ be the integer corresponding to the position of the incoming edge of $v_{i}$ connected to the $i$-th leaf of $T$ according to the planar order. Then, $h[i]_{\ast}(\Phi)(T)_{v}$ is defined as follows:
  \begin{enumerate}
    \item If $v\neq v'$, then $h[i]_{\ast}(\Phi)(T)_{v}=\Phi(T')_{v}$.
    \item If $v=v_{i}$ and $|v_{i}|\geq 3$, then $h[i]_{\ast}(\Phi)(T)_{v_{i}}=h[\epsilon_{T,i}]_{\ast}\big( \Phi(T')_{v_{i}}\big)$ using the operation \eqref{Finaleq}.
    \item If $v=v_{i}$ and $|v_{i}|=2$, then $h[i]_{\ast}(\Phi)(T)_{v_{i}}=b_{v_{i}}$ where $b_{v_{i}}$ is the image of $1$ by the composite map $\K\rightarrow \calC(2)\rightarrow \Omega_{l}(\calC)(2)$.
  \end{enumerate}
  One can then check easily by hand that this defines a $\Lambda$-costructure on $\calB_{l}\Omega_{l}\calC$ and that the morphism $\Gamma$ is compatible with this structure.
\end{proof}

\section{Fibrant resolutions for Hopf $\Lambda$-cobimodules}
\label{sec:fibr-resol-hopf-2}

For any pair of 1-reduced Hopf $\Lambda$-cooperads $\calP$ and $\calQ$ as well as any $(\calP\text{-}\calQ)$-cobimodule $M$, we build a Hopf $(W_{l}\calP\text{-}W_{l}\calQ)$-cobimodule $W_{l}M$ together with a quasi-isomorphism $\eta:M\rightarrow W_{l}M$ where $W_{l}\calP$ and $W_{l}\calQ$ are the Boardman--Vogt resolutions introduced in Section \ref{SectBVCoop} associated to $\calP$ and $\calQ$, respectively.
Furthermore, we show that $W_{l}M$ provides a fibrant resolution of $M$.
Similarly to the previous sections, this Boardman--Vogt resolution is quasi-isomorphic to a leveled version of the two-sided bar construction of the cobimodule of its primitive elements.
Finally, we compare the fibrant resolution with the two-sided leveled bar-cobar construction.

\subsection{The two-sided leveled bar construction for bimodules}
\label{sec:two-sided-leveled}

First, we introduce the two-sided cofree cobimodule functor.
For this purpose we use the category of leveled trees with section $\slTree_{\iso}[n]$, introduced in Section \ref{TreeSect}, whose morphisms are generated by isomorphisms of leveled planar trees with section, permutation morphisms of permutable levels and contraction morphisms of permutable levels. Let $A$ and $B$ be two  1-reduced symmetric cosequences in $\dg\Sigma \Seq_{>1}^{c}$.
We construct the \textit{cofree cobimodule functor}:
$$
  \calF_{l}^{c}[A,B]:\dg\Sigma \Seq_{>0}^{c}\longrightarrow \Sigma\Cobimod_{\calF_{l}^{c}(A)\,,\,\calF_{l}^{c}(B)},
$$
where $\calF_{l}^{c}(A)$ and $\calF_{l}^{c}(B)$ are the leveled cofree cooperads introduced in Section \ref{SectBarCoop}.
Let $C$ be a symmetric cosequence in $\dg\Sigma \Seq_{>0}^{c}$. We consider the two functors:
\begin{align*}
  s\overline{C}_{B} : \slTree_{\iso}[n]^{op} & \longrightarrow \Ch, \quad (T,\iota) \longmapsto \bigotimes_{\mathclap{v\in V_{d}(T)}} A(|v|) \otimes \bigotimes_{\mathclap{v\in V_{\iota}(T)}} C(|v|) \otimes \bigotimes_{\mathclap{v\in V_{u}(T)}} B(|v|) ; \\
  sE'_{1} :\slTree_{\iso}[n]                 & \longrightarrow \Ch, \quad (T,\iota) \longmapsto \bigotimes_{\mathclap{0\leq i\neq \iota\leq h(T)}} \K.
\end{align*}
\begin{defi}
  The two-sided cofree cobimodule functor is given arity-wise by the end:
  $$
    \calF_{l}^{c}[A,B](C)(n) \coloneqq \int_{T\in \slTree_{\iso}[n]} s\overline{C}_{B}(T)\otimes sE'_{1}(T).
  $$
\end{defi}
Roughly speaking, a point of $\calF_{l}^{c}[A,B](C)(n)$ is a map $\Phi$ which assigns to each leveled trees with section $T$ a decoration of the vertices on the main section (resp. below and above the main section) by elements in $C$ (resp. in $A$ and $B$).
See Figure~\ref{fig:example-cofree-cobim} for an example. Furthermore, there is a map from $\calF_{l}^{c}[A,B](C)$ to the sequence $C$ by taking the image of the corollas $c_{n}$, for any $n\geq 0$.
The cobimodule structure is given by the following operations:
\begin{align*}
  \tilde{\gamma}_{L} : \calF_{l}^{c}[A,B](C)(n_{1}+\cdots + n_{k}) & \longrightarrow \calF_{l}^{c}(A)(k) \otimes \calF_{l}^{c}[A,B](C)(n_{1})\otimes \cdots \otimes \calF_{l}^{c}[A,B](C)(n_{k}) \\
  \Phi                                                             & \longmapsto \bigl\{ \tilde{\Phi}(T_{0}, \{T_{i}\}) = \Phi(\gamma_{L}(T_{0}, \{T_{i}\}))\bigr\},                             \\
  \tilde{\gamma}_{R} : \calF_{l}^{c}[A,B](C)(n_{1}+\cdots + n_{k}) & \longrightarrow \calF_{l}^{c}[A,B](C)(k) \otimes  \calF_{l}^{c}(B)(n_{1})\otimes \cdots \otimes \calF_{l}^{c}(B)(n_{k})     \\
  \Phi                                                             & \longmapsto \bigl\{ \tilde{\Phi}(T_{0},\{T_{i}\})=\Phi(\gamma_{R}(T_{0}\,,\,\{T_{i}\}))\bigr\},
\end{align*}
where $\gamma_{L}$ and $\gamma_{R}$ are the operations introduced in Section \ref{SectTreeLev}.

\begin{figure}[htbp]

  \hspace{-30pt}\includegraphics[scale=0.35]{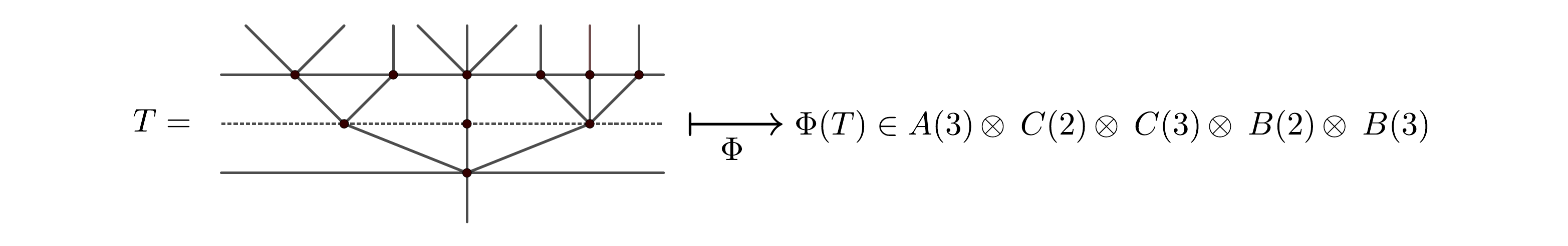}
  \caption{Example of an element in $\calF^c_l[A,B](C)$.}
  \label{fig:example-cofree-cobim}
\end{figure}

\begin{pro}
  The forgetful functor $\calU$ is adjoint to the two-sided cofree cobimodule functor:
  $$
    \calU:\Sigma\Cobimod_{\calF_{l}^{c}(A), \calF_{l}^{c}(B)}\leftrightarrows \dg\Sigma \Seq_{>0}^{c}:\calF^{c}_{l}[A,B].
  $$
  This adjunction is moreover functorial with respect to the sequences $A$ and $B$.
\end{pro}

\begin{proof}
  We need to check that the functor $\calF_{l}^{c}[A,B]$ satisfies the universal property associated to the right adjoint of the forgetful functor. Let $M$ be a $(\calF_{l}^{c}(A)\text{-}\calF_{l}^{c}(B))$-cobimodule and $M'$ be a sequence together with a map of sequences $\phi:M\rightarrow M'$. So, we have to build a  $(\calF_{l}^{c}(A)\text{-}\calF_{l}^{c}(B))$-cobimodule map $\tilde{\phi}:M\rightarrow \calF_{l}^{c}[A,B](M')$ such that the following diagram commutes:
  \[
    \begin{tikzcd}
      M \ar[r] \ar[rd] & M' \\
      {} & \calF_{l}^{c}[A,B](M') \ar[u]
    \end{tikzcd}
  \]

  Let $x$ be a point in $M(n)$. On the corolla $c_{n}$, the map $\tilde{\phi}(x)$ must be defined by $\tilde{\phi}(x)(c_{n})=\phi(x)$.
  More generally, let $T$ be a leveled tree with section.
  In order to get a cobimodule map, $\tilde{\phi}$ has to be defined by the composite map
  $$
    \tilde{\phi}(x)(T)=\phi\circ\Delta_{T}(x),
  $$
  using the cobimodule structure of $M$ and applying $\phi$ to the components corresponding to the vertices on the main section. It is the only way to define $\tilde{\phi}$ in order to get a map of cobimodules.
\end{proof}

\begin{defi}\label{def:leveled-bar-sided}
  Let $\calP$ and $\calQ$ be two 1-reduced dg-operads and let $M$ be a $(\calP\text{-}\calQ)$-bimodule.
  The leveled two-sided bar resolution of $M$ according to $\calP$ and $\calQ$ is given by:
  $$
    \calB_{l}[\calP,\calQ](M) \coloneqq \bigl( \calF_{l}^{c}[\Sigma\calU(\overline{\calP}), \Sigma\calU(\overline{\calQ})](\calU(M)), \; d_{\mathrm{int}}+d_{\mathrm{ext}} \bigr).
  $$
  Let $\Phi \in \calB_{l}[\calP,\calQ](M)$ be an element and $(T, \iota) \in \sTree_{l}[n]$ be a leveled tree with section.
  We define:

  $$
    \deg'(\Phi,(T, \iota)) \coloneqq \sum_{v \in V_{\ge2}(T) \setminus V_{\iota}(T)} (\deg(\theta_{v}) + 1) + \sum_{v \in V_{\iota}(T)} \deg(\theta_{v}).
  $$
  Then we say that $\deg \Phi = d$ if $\deg'(\Phi,T) = d$ for all $T$.

  The ($\calB_{l}\calP$-$\calB_{l}\calQ$)-cobimodule structure and the $\Lambda$-structure arise from the cofree cobimodule functor $\calF_{l}^{c}[\calU(\overline{\calP}),\calU(\overline{\calQ})](\calU(M))$.
  Finally, the differential is the sum of two terms:
  \begin{itemize}
    \item The internal differential $d_{\mathrm{int}}$ is the differential corresponding to the differential algebras $\calU(\overline{\calP})$, $\calU(M)$ and $\calU(\overline{\calQ})$.
    \item The external differential $d_{\mathrm{ext}}$ consists in contracting two consecutive levels.
          More precisely, for $\Phi \in \calB_{l}[\calP,\calQ](M)$ and a leveled tree with section $(T, \iota)$, we consider the set (where $D_T$ was defined in Definition~\ref{DefBarlev})
          \[ D_{T,\iota} \coloneqq \{ ((T', \iota'), i) \in \slTree_{\iso}[n] \times \mathbb{N} \mid (T', i) \in D_T \}. \]
          Then $(d_{\mathrm{ext}}\Phi)(T) = \sum_{((T',\iota'),i)} \gamma_{i}(\Phi(T',\iota'))$, where $\gamma_{i}$ uses either the operadic/module structures of $\calP$, $\calQ$, and $M$ to compose the elements corresponding to the contraction of level $i$.
  \end{itemize}
\end{defi}

\begin{pro}\label{pro:two-sided-bar-well-def}
  The leveled two-sided bar construction $\calB_{l}[\calP,\calQ](M)$ of a dg-($\calP$-$\calQ$)-bimodule $M$ is a well defined dg-($\calB_{l}\calP$-$\calB_{l}\calQ$)-cobimodule.
\end{pro}

\begin{proof}
  The proof is an extension of the proof of Proposition~\ref{prop:bar-well-def}.
  The external differential $d_{\mathrm{ext}}$ is the unique coderivation of ($\calP$-$\calQ$)-cobimodules induced by the map $\alpha : \calF_{l}^{c}[\Sigma\calU(\overline{\calP}), \Sigma\calU(\overline{\calQ})](\calU(M)) \to \calU(M)$ defined as follows.
  For the corolla $c_{n}$, $D_{c_{n},0}$ is the set of trees with exactly two levels and exactly one vertex with $\ge2$ incoming edges on the level which is not the main section.
  For $\Phi \in \calF_{l}^{c}[\Sigma\calU(\overline{\calP}),\Sigma\calU(\overline{\calQ})](\calU(M))$, the element $\alpha(\Phi) \in \calU(M)$ is the sum over all $((T', \iota'), i) \in D_{c_{n},0}$ of the application of the bimodule structure maps of $M$ to the element $\Phi(T) \in M(k) \otimes \calQ(l)$ or $\Phi(T) \in \calP(k) \otimes M(l)$ (depending on whether $\iota' = 0$ or $1$).
  Checking that $d_{\mathrm{ext}}^{2} = 0$ follows again from the associativity of the bimodule structures and the compatibility between the left and the right actions.
\end{proof}

\subsection{The leveled Boardman--Vogt resolution}\label{SectBVCobi}

Similarly to Section \ref{SectBVCoop}, we split this section into three parts. First, we introduce a leveled version of the Boardman--Vogt resolution for Hopf cobimodules.
Then, we extend this construction to Hopf $\Lambda$-cobimodules.
The last part is devoted to a simplicial frame version of our construction.

\paragraph{The leveled Boardman--Vogt resolution for Hopf cobimodules}

Let $\calP$ and $\calQ$ be be two 1-reduced cooperads in chain complexes and let $M$ be a $(\calP \text{-} \calQ)$-cobimodule.
According to the notation introduced in Section~\ref{SectTreeLev}, we consider the following two functors:
\begin{align*}
  s\overline{M}_{W} : \slTree[n]^{op} & \longrightarrow \CDGA, \qquad (T,\iota) \longmapsto \bigotimes_{\mathclap{v\in V_{d}(T)}} \calP(|v|) \otimes \bigotimes_{\mathclap{v\in V_{\iota}(T)}} M(|v|) \otimes \bigotimes_{\mathclap{v\in V_{u}(T)}} \calQ(|v|); \\
  sE_{W} : \slTree[n]                 & \longrightarrow \CDGA, \qquad (T,\iota) \longmapsto \bigotimes_{\mathclap{0\leq i\neq \iota\leq h(T)}} \K[t,dt].
\end{align*}

The functor $s\overline{M}_{W}$ labels the vertices on the main section (resp. below and above the main section) by elements of the cobimodule $M$ (resp. by elements of the cooperads $\calP$ and $\calQ$).
The functor $sE_{W}$ indexes levels other than $\iota$ by polynomial differential forms.
By convention, the level $\iota$ is indexed by the constant form $0$, i.e.\ one has $p_{\iota}(t,dt)=0$.

On morphisms, the functor $s\overline{M}_{W}$ is defined using the cooperadic structures of $\calP$ and $\calQ$, the cobimodule structure of $M$ and the symmetric monoidal structure of the ambient category.
On permutations $\sigma$ of two permutable levels, the functor $sE_{W}$ consists in permuting the parameters indexing the corresponding levels.
On contraction maps $\delta_{\{i+1\}}:T\rightarrow T/\{i+1\}$, with $i\in \{0,\ldots,h(T)-1\}$, there are three cases to consider:
\begin{enumerate}[label={Case \arabic*:}, leftmargin=35pt]
  \item If the levels $i$ and $i+1$ are permutable (in particular $\iota\notin \{i,i+1\}$), then one has
        \begin{align*}
          sE_{W}(\delta_{\{i+1\}}): sE_{W}(T) & \longrightarrow sE_{W}(T/\{i+1\})                                 \\
          (p_{0},\ldots,p_{h(T)})           & \longmapsto (p_{0},\ldots,p_{i}\cdot p_{i+1},\ldots, p_{h(T)}).
        \end{align*}
  \item If the levels $i$ and $i+1$ are not permutable and $i\geq \iota$ is above the main section, then one has
        \begin{align*}
          sE_{W}(\delta_{\{i+1\}}): sE_{W}(T) & \longrightarrow sE_{W}(T/\{i+1\})                                   \\
          (p_{0},\ldots,p_{h(T)})           & \longmapsto (p_{0},\ldots,\ev_{t=0}\circ p_{i+1},\ldots, p_{h(T)}).
        \end{align*}
  \item If the levels $i$ and $i+1$ are not permutable and $i+1\leq \iota$ is below the main section, then one has
        \begin{align*}
          sE_{W}(\delta_{\{i+1\}}): sE_{W}(T) & \longrightarrow sE_{W}(T/\{i+1\})                                     \\
          (p_{0},\ldots,p_{h(T)})           & \longmapsto (p_{0},\ldots,\ev_{t=0}\circ p_{i},\ldots, p_{h(T)}).
        \end{align*}
\end{enumerate}

\begin{defi}
  The leveled Boardman--Vogt resolution of $M$ is defined as the end:
  $$
    W_{l}M(n) \coloneqq \int_{T\in \slTree[n]} s\overline{M}_{W}(T) \otimes sE_{W}(T).
  $$
\end{defi}

Roughly speaking, an element in $W_{l}M(n)$ is a map $\Phi:\slTree[n]\rightarrow \CDGA$ which assigns to each leveled tree with section $T$ a decoration.
More precisely, the vertices on the main section (resp. above and below the main section) are indexed by elements of $M$ (resp. the cooperads $\calQ$ and $\calP$) while the levels other than the main section are indexed by polynomial differential forms.
Such a map needs to satisfy some relations related to morphisms in the category of leveled trees with section.
For each permutation $\sigma$ of permutable levels, one has $\Phi(T)=\Phi(\sigma\cdot T)$ and for each morphism $\delta_{N}:T\rightarrow T/N$, one has the following identification in the commutative differential graded algebra $s\overline{M}_{W}(T)\otimes sE_{W}(T/N)$:
\begin{equation}\label{relationcoeq2}
  \bigl( \id \otimes sE_{W}(\delta_{N})\bigr) \circ \Phi(T)= \bigl( M(\delta_{N})\otimes \id \bigr)\circ \Phi(T/N).
\end{equation}

The cobimodule structure is given by the following operations:
\begin{align*}
  \tilde{\gamma}_{L} : W_{l}M(n_{1}+\cdots + n_{k}) & \longrightarrow  W_{l}\calP(k) \otimes W_{l}M(n_{1})\otimes \cdots \otimes W_{l}M(n_{k}),                                      \\
  \Phi                                              & \longmapsto \bigl\{ \tilde{\Phi}_{L}(T_{0},\{T_{i}\}) = \ev_{T_{0},\{T_{i}\}} \circ \Phi(\gamma_{l}(T_{0},\{T_{i}\}))\bigr\};  \\
  \tilde{\gamma}_{R}: W_{l}M(n_{1}+\cdots + n_{k})  & \longrightarrow W_{l}M(k) \otimes W_{l}\calQ(n_{1})\otimes \cdots \otimes W_{l}\calQ(n_{k}),                                   \\
  \Phi                                              & \longmapsto \bigl\{ \tilde{\Phi}_{R}(T_{0}, \{T_{i}\}) = \ev_{T_{0},\{T_{i}\}} \circ \Phi(\gamma_{r}(T_{0},\{T_{i}\}))\bigr\}.
\end{align*}
where $\gamma_{L}$ and $\gamma_{R}$ are the operations introduced in Section \ref{SectTreeLev} on leveled trees with section.
Furthermore, the evaluation maps $\ev_{T_{0},\{T_{i}\}}$ is given by the formula \eqref{Eva}.
The reader can easily check that the operations so obtained are well defined.
By using the same arguments as in the proof of Proposition \ref{OpStr}, we conclude that $W_{l}M$ has a $(W_{l}\calP\text{-}W_{l}\calQ)$-cobimodule structure.
Finally, there is a map of $(W_{l}\calP\text{-}W_{l}\calQ)$-cobimodules
$$
  \eta:M\longrightarrow W_{l}M
$$
sending an element $x\in M(n)$ to the map $\Phi_{x}$ which is defined by indexing the vertices according to the operation $\Delta_{T}(x)$ using the $(W_{l}\calP\text{-}W_{l}\calQ)$-cobimodule structure of $M$ and indexing the levels other than $\iota$ by the constant polynomial differential form equal to $1$.

\begin{thm}\label{ThmQI2}
  The morphism of cobimodules $\eta:M\longrightarrow W_{l}M$ is a quasi-isomorphism.
\end{thm}

\begin{proof}
  The proof is similar to \cite[Proposition 5.2]{FresseTurchinWillwacher2017} and the proof of Theorem \ref{ThmQI}.
  We use the splitting $\K[t,dt]=\K1 \oplus \K[t,dt]_{0}$ where $\K[t,dt]_{0}\subset \K[t,dt]$ is the acyclic ideal formed by the polynomial differential forms that vanish at $t=0$.
  We consider the following functor:
  \begin{align*}
    sE'_{W} : \slTree[n] & \longrightarrow \Ch,                                             \\
    T                    & \longmapsto \bigotimes_{0\leq i\neq \iota \leq h(T)} \K[t,dt]_{0}.
  \end{align*}
  Let us remark that, if a $k$-level, with $k\neq \iota$, of a leveled tree with section $T$ is indexed by a polynomial of the form $\K1$, then the decoration $\Phi(T)$ can be obtained from $\Phi(T/\{k\})$ due to the relation (\ref{relationcoeq2}). Consequently, as chain complexes, there is a quasi-isomorphism:

  $$
    W_{l}M(n) \simeq \prod_{[T] \in \pi_{0} \beta\sTree^{\geq 2}[n]} s\overline{M}_{W}(T) \otimes sE'_{W}(T),
  $$
where $\beta\sTree^{\geq 2}[n]$ is the essential image of $\beta$, i.e.
the subcategory of $\slTree[n]$ which consists of leveled trees with section having exactly one non-bivalent vertex in each level other than the main section. The above product is over classes of leveled trees up to isomorphisms of planar trees and permutations of permutable levels. Notice that, thanks to the identity \eqref{relationcoeq2}, a point in Boardman--Vogt resolution in  determined by its values on the leveled trees in $\beta\sTree^{\geq 2}[n]$. Furthermore, we disregard the term on the right-hand side in which $T$ is not the $n$-corolla $c_{n}$, we obtain a contractible complex. It follows that the product is quasi-isomorphic to $s\overline{M}_{W}(c_{n})=M(n)$ and the canonical map $\eta:M\rightarrow W_{l}M$ is given by the identity on this factor, thus finishing the proof.
\end{proof}

\paragraph{The leveled Boardman--Vogt resolution for Hopf $\Lambda$-cobimodules}

Let $\calP$ and $\calQ$ be two 1-reduced Hopf $\Lambda$-cooperads and $M$ be a Hopf $\Lambda$-cobimodule over the pair $(\calP,\calQ)$.
In order to get a fibrant resolution of $M$ in the Reedy model category $\Lambda\Bimod_{\calP,\calQ}$, we provide a slight variant of the construction introduced in the previous paragraph.
As a symmetric cosequence, we set
$$
  W_{\Lambda}M(n) \coloneqq W_{l}M_{>0}(n), \hspace{15pt} \text{for all } n>0,
$$
where $M_{>0}$ is the underlying ($\calP_{>0}$-$\calQ_{>0}$)-cobimodule of $M$.
The subscript $\Lambda$ is to emphasize that we work in the category of 1-reduced Hopf $\Lambda$-cooperads.
The symmetric cosequence $W_{\Lambda}M$ inherits a ($W_l\calP_{>0}$-$W_l\calQ_{>0}$)-cobimodule
structure from $W_{l}M_{>0}$.

Let us define the $\Lambda$-costructure on the construction $W_{\Lambda}M$ which is compatible which the cobimodule structure.
For simplicity, we only build the costructure associated to the order preserving map ${h[i]:[n]\rightarrow [n+1]}$ skipping the $i$-th term (i.e. $h[i](j)=j$ if $j<i$ and $h[i](j)=j+1$ if $j\geq i$).
We shall construct a map of the form
\begin{equation}\label{Lambdacostructure2}
  \begin{aligned}
    h[i]_{\ast}:W_{\Lambda}M(n) & \longrightarrow  W_{\Lambda}M(n+1)                                                              \\
    \Phi                        & \longmapsto h[i]_{\ast}(\Phi) \coloneqq \bigl\{ h[i]\circ\Phi(T),\,\,T\in \slTree[n+1] \bigr\}.
  \end{aligned}
\end{equation}
Let $T$ be a leveled $(n+1)$-tree with section. In what follows, we denote by $v$ the first non-bivalent vertex composing the path from the $i$-th leaf to the root.
In order to define $h[i]\circ\Phi(T)\in s\overline{M}_{W}(T)\otimes sE_{W}(T)$, there are different cases to consider:

\begin{itemize}[leftmargin=45pt]
  \item[Case $1$:]
        If $v$ has at least three incoming edges, then we consider the leveled $n$-tree with section $T'$ defined from $T$ by removing the branch coming from the leaf indexed by $i$.
        In that case, $h[i]\circ\Phi(T)$ is given by
        $$
          h[i]\circ\Phi(T)=(h[i]_{|v})_{\ast} \circ \Phi(T'),
        $$
        where the map $(h[i]_{|v})_{\ast}$ is obtained using the $\Lambda$-costructures of $\calP$, $\calQ$ or $M$ applied to the restriction map $h[i]_{|v}:[|\operatorname{in}(v)|-1]\rightarrow [|\operatorname{in}(v)|]$ to the incoming edges of $T$.
        For instance, in the next picture, the corresponding map $h[7]_{|v}:[3]\rightarrow [4]$ is given by $h[7]_{|v}(j)=j$.

        \hspace{-22pt}\includegraphics[scale=0.3]{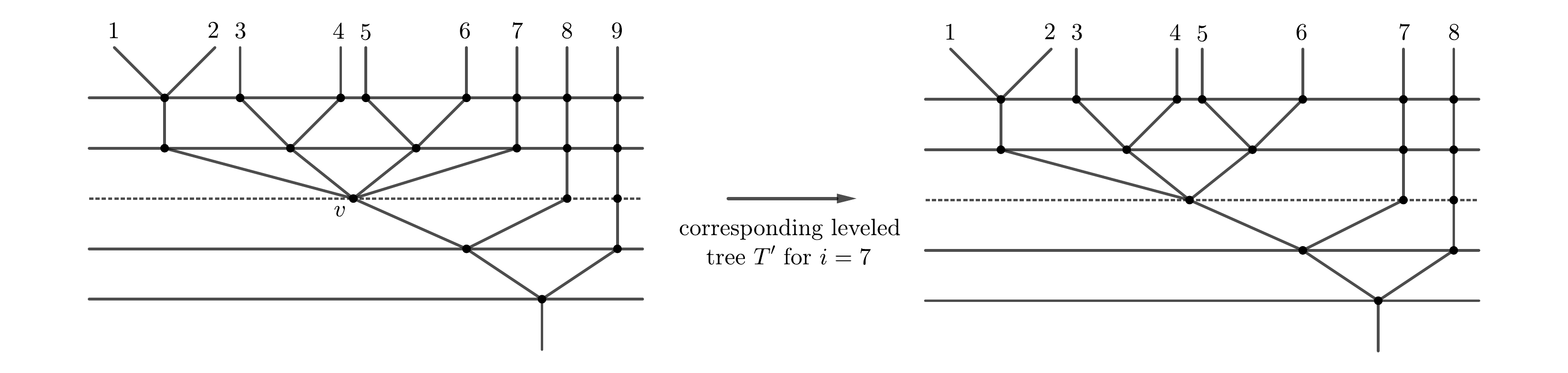}\vspace{-10pt}

  \item[Case $2$:] If $v$ has only two incoming edges, then we denote by $e_{i}$ the incoming edge connected to the $i$-th leaf.
        We consider $T'$ obtained from $T$ by removing the branch connecting the leaf $i$ to $v$.

  \item[Case $2.1$:] If the level $h(v)$ of $T'$ has at least one non-bivalent vertex, then $T'$ is a leveled tree with section and one has
        $$
          h[i]\circ\Phi(T)= b_{v}\otimes \Phi(T'),
        $$
        where $b_{v}$ is the image of $1$ by the map $\K\rightarrow \calP(2)$ or $\K\rightarrow \calQ(2)$ (depending on if $v$ is below or above the main section) induced by the $\Lambda$-costructure of $\calP$ and $\calQ$. Roughly speaking, it consists in indexing $v$ by the element $b_{v}$ and keeping the decoration of the other vertices and the levels induced by $\Phi(T')$.

        \hspace{-22pt}\includegraphics[scale=0.3]{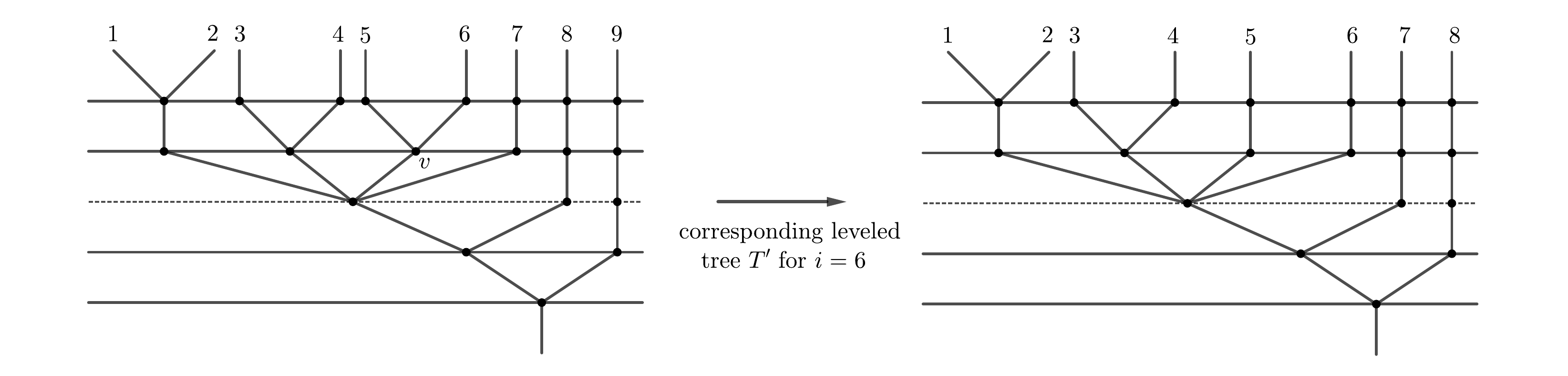}\vspace{-10pt}

  \item[Case $2.2$:] If $v$ is the root, $\iota \neq 0$, and $v$ is trivalent, then we consider the leveled tree $T''$ obtained from $T'$ by removing the zeroth level. In that case, one has
        $$
          h[i]\circ\Phi(T)= 1\otimes b_{v}\otimes \Phi(T'').
        $$
        Roughly speaking, it consists in indexing $v$ (which is the root in that case) by the element $b_{v}$, labelling the level $1$ by $1\in \K[t,dt]$ and keeping the decoration of the other vertices and the other levels induced by $\Phi(T'')$.

        \hspace{-22pt}\includegraphics[scale=0.3]{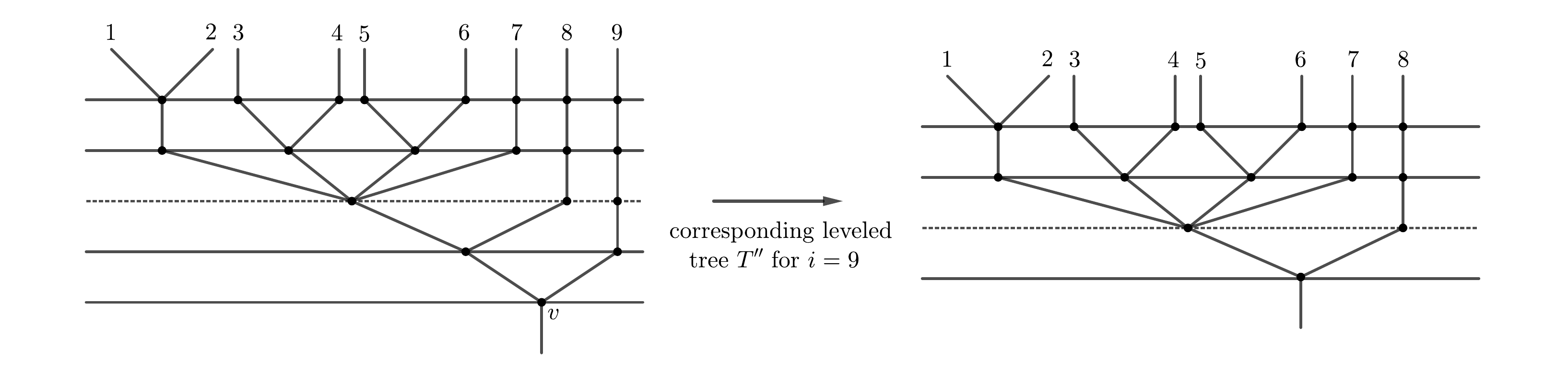}\vspace{-10pt}

  \item[Case $2.3$:]  If $v$ is a trivalent vertex of $T$ at maximal height $h(v)=h(T)\neq \iota$ and all other vertices at $h(T)$ are bivalent, then we consider the leveled tree $T''$ obtained from $T'$ by removing the section $h(v)$.
        In that case, one has
        $$
          h[i]\circ\Phi(T)= 1\otimes b_{v}\otimes \Phi(T'').
        $$
        Roughly speaking, it consists in indexing $v$ by the element $b_{v}$, labelling the top level by $1\in \K[t,dt]$ and keeping the decoration of the other vertices and the other levels induced by $\Phi(T'')$.

        \hspace{-22pt}\includegraphics[scale=0.3]{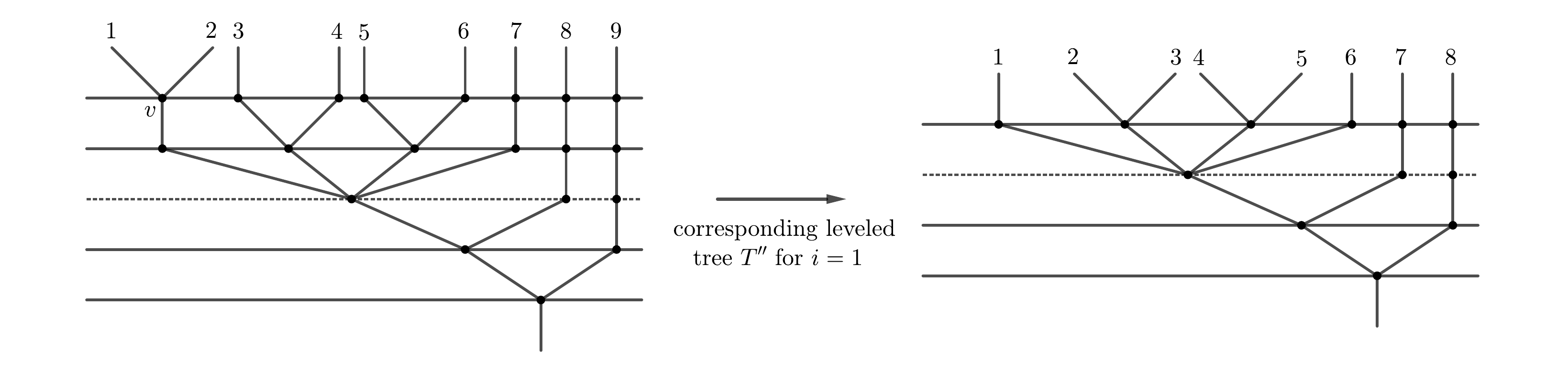}

  \item[Case $2.4$:] If $v$ is the unique non-bivalent vertex at the level $h(v)\notin \{0,h(T),\iota\}$, then we consider the leveled tree $T''$ obtained from $T'$ by removing the section $h(v)$. In that case, one has
        $$
          h[i]\circ\Phi(T)= m^{\ast}_{h(v)}\otimes b_{v}\otimes \Phi(T''),
        $$
        where $m^{\ast}$ is the coassociative coproduct introduced in Section \ref{SectModHopf} and $m^{\ast}_{h(v)}$ is the coproduct applied to the polynomial form associated to the $h(v)$-th level of the leveled tree $T''$.

        \hspace{-22pt}\includegraphics[scale=0.3]{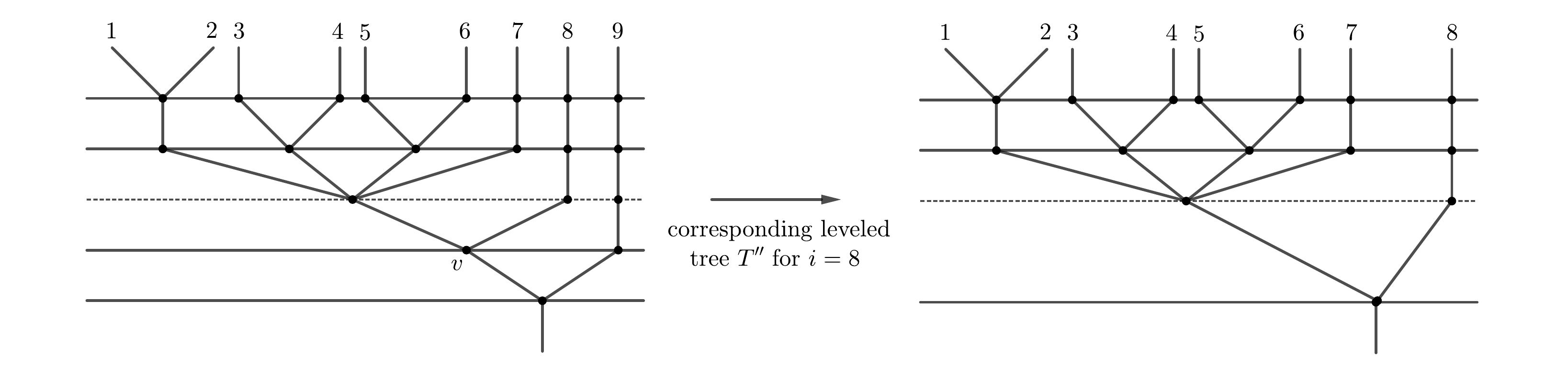}

  \item[Case $2.5$:] If the level $h(v)$ of $T'$ is composed of bivalent vertex and $h(v)=\iota$, then $T'$ is a leveled tree with section and one has one has
        $$
          h[i]\circ\Phi(T)=  b_{v}\otimes \Phi(T'),
        $$
        where $b_{v}$ is the image of $1$ by the map $\K\rightarrow M(2)$ induced by the $\Lambda$-costructure of $M$.
        Roughly speaking, it consists in indexing $v$ (which is on the main section) by the element $b_{v}$, labelling the level $1$ by $1\in \K[t,dt]$ and keeping the decoration of the other vertices and the other levels induced by $\Phi(T')$.

        \hspace{-22pt}\includegraphics[scale=0.3]{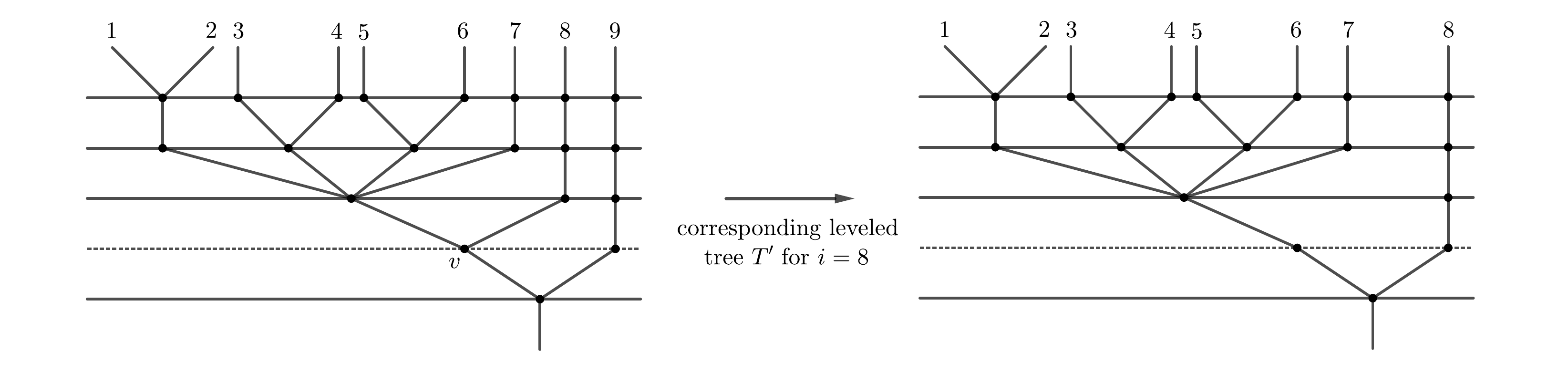}\vspace{-10pt}
\end{itemize}

\begin{thm}\label{thm:w-lambda-cobim}
  The $\Lambda$-costructure \eqref{Lambdacostructure2} makes the cobimodule $W_{\Lambda}M$ into a $\Lambda$-cobimodule over the pair of $\Lambda$-cooperad $W_{\Lambda}\calP$ and $W_{\Lambda}\calQ$. Furthermore, the morphism $\eta:M \rightarrow W_{\Lambda}M$ introduced in Theorem \ref{ThmQI2} is a quasi-isomorphism of $\Lambda$-cobimodules.
\end{thm}

\paragraph{Simplicial frame}

Let us now define a simplicial framing of $W_{\Lambda}M$, similarly to Proposition~\ref{prop:simplicial-frame-cooperad}.
The proofs are completely analogous.
We define a new functor:
\begin{equation*}
  sE^{\Delta^{d}}_{W} : \slTree[n] \to \CDGA, \qquad
  (T, \iota) \mapsto \bigotimes_{\mathclap{0\leq i\neq \iota\leq h(T)}} \K[t,dt] \otimes \bigotimes_{\mathclap{0 \le j \le h(T)}} \Omega^{*}_{\PL}(\Delta^{d}).
\end{equation*}

We then define:
\[ W^{\Delta^{d}}_{\Lambda}M(n) \coloneqq \int_{T\in \slTree[n]} s\overline{M}_{W}(T) \otimes sE^{\Delta^{d}}_{W}(T). \]

\begin{pro}
  Let $M$ be a ($\calP$-$\calQ$)-bimodule.
  The collection $(W^{\Delta^{\bullet}}_{\Lambda}\calP, W^{\Delta^{\bullet}}_{\Lambda}M, W^{\Delta^{\bullet}}_{\Lambda}\calQ)$ defines a simplicial frame for the triple $(\calP, M, \calQ)$ in the category of bimodules.
\end{pro}
\begin{proof}
  One easily checks that $(W^{\Delta^{\bullet}}_{\Lambda}\calP, W^{\Delta^{\bullet}}_{\Lambda}M, W^{\Delta^{\bullet}}_{\Lambda}\calQ)$ defines a simplicial object in the category of bimodules (i.e.\ the category whose objects are triples $(\calP, M, \calQ)$ consisting of two operads and a bimodule over them).
  The proof that this is a simplicial frame for $(\calP, M, \calQ)$ is identical to the one of Proposition~\ref{prop:simplicial-frame-cooperad}.
\end{proof}

\subsection{The bimodule of primitive elements}

Let $\calP$ and $\calQ$ be 1-reduced Hopf cooperads.
A point in a ($\calP$-$\calQ$)-cobimodule $M$ is said to be \textit{primitive} if its image via both cobimodule operations is $0$.
By definition of the cobimodule structure of $W_{l}M$, an element $\Phi\in W_{l}M(n)$ is primitive if and only if for each leveled $n$-tree with section $(T, \iota)$ and for each level $i\in \{0,\ldots, h(T)\} \setminus \iota$, the evaluation of the polynomial differential form $p_{i}(t,dt)$ at $t=1$ vanishes.
Hence, the decoration that $\Phi$ assigns to a leveled tree with section must be so that the decorations of levels belong to the subspace $\K[t,dt]_{1} \coloneqq \ker(\ev_{t=1}) \subset \K[t,dt]$ of polynomial differential forms that vanish at the endpoint $t=1$.
For this reason, we introduce the functor
\begin{equation*}
  sE''_W : \slTree[n] \longrightarrow \Ch, \qquad
  (T,\iota) \longmapsto \bigotimes_{\substack{0\leq i \leq h(T) \\ i \neq \iota}} \K[t,dt]_{1}.
\end{equation*}

We will show (Proposition \ref{prop:primitive cobimod}) that the primitive elements form a $(\Prim(W_{l}\calP)\text{-}\Prim(W_{l}\calQ))$-bimodule.

\begin{defi}
  The space of the primitive elements of a ($\calP$-$\calQ$)-cobimodule $M$, denoted by $\Prim(W_{l}M)$ is the sequence given by the end
  $$
    \Prim(W_{l}M)(n) \coloneqq \int_{T\in \slTree[n]} s\overline{M}_{W}(T)\otimes sE''_{W}(T).
  $$
\end{defi}

\begin{lmm}\label{lem:cofree-wl}
  As a graded cobimodule, $W_{l}M$ is the cofree $(W_{l}\calP, W_{l}\calQ)$-cobimodule
  generated by the sequence $\Prim(W_{l}M)$, i.e.:
  $$
    W_{l}M \cong \calF^{c}_{l}[\Prim(W_{l}\calP),\Prim(W_{l}\calQ)](\Prim(W_{l}M)).
  $$
\end{lmm}

\begin{proof}
 An element in the cofree cobimodule object is a map $\omega$ mapping a leveled tree with section to a decoration of the vertices on the main vertices (resp. below and above the main section) by elements in $\Prim(W_{l}M)$ (resp. by elements in $\Prim(W_{l}\calP)$ and $\Prim(W_{l}\calQ)$):
  $$
    \omega(T, \iota) \in \bigotimes_{v\in V_{d}(T)} \Prim(W_{l}\calP)(|v|) \otimes \bigotimes_{v\in V_{\iota}(T)} \Prim(W_{l}M)(|v|) \otimes \bigotimes_{v\in V_{u}(T)} \Prim(W_{l}\calQ)(|v|).
  $$

  In order to construct a morphism of sequences $W_{l}M\rightarrow \Prim(W_{l}M)$, let us recall that thanks to the identity \eqref{relationcoeq2}, any point in the Boardman--Vogt resolution is determined by its values on the leveled trees in the essential image $\beta \sTree$ having exactly one non-bivalent vertex in each level other than the main section. The same is true for the subspace of primitive elements. Then we introduce the natural transformation $\pi:sE_{W}\Rightarrow sE''_{W}$ which carries any polynomial differential form $p(t,dt)\in \K[t,dt]$ to the polynomial $\tilde{p}(t,dt)=p(t,dt)-tp(1,0)$~\cite[Lemma~5.3]{FresseTurchinWillwacher2017}.
  In particular, one has $\ev_{t=0}\tilde{p}(t,dt)=\ev_{t=0}p(t,dt)$ (so the new element also satisfies the equations $\Prim(W_{l}M)$ as an end). For any $\Phi\in W_{l}M$ and $T\in \beta\sTree^{\geq 2}$, one has
  $$
  \phi(\Phi)(T)=(\id\otimes E)\circ \Phi(T).
  $$
According to the universal property of the cofree cobimodule object, one has a morphism of graded $(W_{l}\calP\text{-}W_{l}\calQ)$-cobimodules
  $$
    \psi:W_{l}M \longrightarrow \calF^{c}_{l}[\Prim(W_{l}\calP),\Prim(W_{l}\calQ)](\Prim(W_{l}M)),
  $$
  due to the identification $W_{l}\calP\cong \calF^{c}_{l}(\Prim(W\calP))$ and $W_{l}\calQ\cong \calF^{c}_{l}(\Prim(W\calQ))$ described in Lemma~\ref{Lmm1}.

  The injectivity of our morphism $\psi$ on the primitive elements in the source and the fact that $W_{l}M$ is conilpotent imply that $\psi$ is injective itself.
  We are therefore left to prove that $\psi$ is surjective.

  Let $\omega \in \calF^{c}_{l}[\Prim(W_{l}\calP),\Prim(W_{l}\calQ)](\Prim(W_{l}M))$ be an element in the cofree construction.
  Just like in Lemma~\ref{Lmm1}, we can view $\omega = \{ \omega_{T'} \}$ as a collection of elements $\omega_{T'} \in s\overline{M}_{W}(T') \otimes sE''_{W}(T')$ indexed by isomorphism classes of planar trees with section (i.e.\ planar trees as in Section~\ref{sec:categ-plan-trees} equipped with marked vertices such that each path from a leaf to the root meets a unique marked vertex).
  Let us now define $\Phi \in W_{l}M$ such that $\psi(\Phi) = \omega$.
  Let $(T, \iota) \in \slTree[n]$ be a tree and $[T]$ be the corresponding isomorphism class of planar trees with section.
  \begin{itemize}
    \item Let $v \in V(T)$ be a vertex such that $|v| \ge 2$ or $v$ is on the section $\iota$.
          Then $v$ corresponds to a unique vertex in $[T]$, and we define the decoration of $v$ in $\Phi(T)$ to be the decoration of $v$ in $\omega_{[T]}$.
    \item If a level $0 \le i \neq \iota \le h(T)$ is permutable, then we decorate it with $p(t,dt) = t$.
          If the level $i$ is not permutable, then it corresponds to a unique edge in $[T]$, and we decorate the level with the decoration of the corresponding edge in $\omega_{[T]}$.
  \end{itemize}
  One can then check (thanks to the fact that $\omega$ satisfies the equation of the end defining the cofree cobimodule) that $\Phi$ is a well-defined element of $W_{l}M$, and $\psi(\Phi) = \omega$ is immediate.
\end{proof}

Recall from Proposition~\ref{prop:primitives are operad} that $\Sigma^{-1} \Prim(W_l \calP)$ and $\Sigma^{-1} \Prim(W_l \cal Q)$ are dg-operads.

\begin{pro}\label{prop:primitive cobimod}
  The sequence $\Sigma^{-1} \Prim(W_{l}M) = \{ \Sigma^{-1} \Prim(W_{l}M)(n) \}$ is a dg-$(\Sigma^{-1} \Prim(W_{l}\calP)\text{-}\Sigma^{-1} \Prim(W_{l}\calQ))$-cobimodule.
\end{pro}

\begin{proof}
  The left module operations
  $$
    \tilde{\gamma}_{L} : \Sigma^{-1} \Prim(W_{l}\calP)(k) \otimes \Sigma^{-1} \Prim(W_{l}M)(n_{1}) \otimes \dots \otimes \Sigma^{-1} \Prim(W_{l}M)(n_{k}) \longrightarrow \Sigma^{-1} \Prim(W_{l}M)(n_{1}+\cdots +n_{k})
  $$
  are defined using the operation $\gamma_{L}$ introduced in Section \ref{SectTreeLev}.
  Let $\Phi_{0}$ and $\Phi_{i}$, with $i \leq k$, be maps in $\Sigma^{-1} \Prim(W_{l}\calP)(k)$ and $\Sigma^{-1} \Prim(W_{l}M)(n_{i})$, respectively.
  In order to define $\Phi=\tilde{\gamma}_{L}(\Phi_{0}\,,\,\{\Phi_{i}\})$ there are three cases to consider. First, if the leveled tree $T$ is of the form $\gamma_{L}(T_{0}\,,\,\{T_{i}\})$, then one has
  $$
    \Phi(T)=\Phi_{0}(T_{0}) \otimes \bigotimes_{1\leq i\leq k} (\Phi_{i}(T_{i}) \otimes dt).
  $$
  Secondly, if $T$ is of the form $T'=\gamma_{L}(T_{0}\,,\,\{T_{i}\})$ up to permutations of permutable levels and contractions of permutable levels, then the decoration of $T$ is given by the decoration of $T'$ composed with the corresponding morphisms of permutations and contractions.
  Finally, if $T$ is not of the form $\gamma_{L}(T_{0}\,,\,\{T_{i}\})$ up to permutations and contractions of permutable levels, then $\Phi(T)=0$.

  Similarly, we define the right module operations
  $$
    \tilde{\gamma}_{R} : \Sigma^{-1} \Prim(W_{l}M)(k) \otimes \Sigma^{-1} \Prim(W_{l}\calQ)(n_{1}) \otimes \cdots \otimes \Sigma^{-1} \Prim(W_{l}\calQ)(n_{k}) \longrightarrow \Sigma^{-1} \Prim(W_{l}M)(n_{1}+\cdots + n_{k})
  $$
  by using the operation $\gamma_{R}$ introduced in Section \ref{SectTreeLev}.
  The reader can easily check that the operations so obtained are well defined and make the sequence $\Sigma^{-1} \Prim(W_{l}M)$ into a $(\Sigma^{-1} \Prim(W_{l}\calP)\text{-}\Sigma^{-1} \Prim(W_{l}\calQ))$-bimodule.
\end{proof}

\begin{rmk}
  The product of two primitive elements of $W_{l}M$ may not necessarily be primitive, therefore $\Prim(W_{l}M)$ is not a Hopf cobimodule.
\end{rmk}

\begin{thm}\label{thm:wm-bar-cobim}
  The Boardman--Vogt resolution $W_{l}M$ is isomorphic (as a dg-cobimodule) to the leveled two-sided bar construction of the bimodule of its primitive elements:
  $$
    W_{l}M \cong \calB_{l}[\Sigma^{-1} \Prim(W_{l}\calP), \Sigma^{-1} \Prim(W_{l}\calQ)](\Sigma^{-1} \Prim(W_{l}M)).
  $$
\end{thm}

\begin{proof}
  Similar to the proof of Theorem \ref{thm:W=B for bimods}.
\end{proof}

\subsection{The Boardman--Vogt resolution and the two-sided bar-cobar construction}\label{SectCobarComod}

This section is split into three parts.
First, we introduce an alternative description of the cobar construction for dg-cobimodules.
Then, we show that the bar-cobar construction of a Hopf cobimodule is quasi-isomorphic (as a dg-cobimodule) to its Boardman--Vogt resolution.
Finally, we extend this result to Hopf $\Lambda$-bimodules by introducing a $\Lambda$-costructure on two-sided bar-cobar constructions.

\paragraph{The leveled two-sided cobar construction for cobimodules}

Dually to Section \ref{SectBarCobi}, we extend the free two-sided bimodule functor and the cobar construction to cobimodules using the category of leveled trees with section.
Let $A$ and $B$ be two 1-reduced symmetric sequences.
According to the notation introduced in Section \ref{SectBarCobi}, we build the functor
$$
  \calF_{l}[A,B]:\dg\Sigma \Seq_{>0}\longrightarrow \Sigma\Bimod_{\calF_{l}(A),\calF_{l}(B)},
$$
from the category of symmetric sequences of chain complexes to bimodules. For each object $X\in \dg\Sigma \Seq_{>0}$, we consider the two functors
\begin{align*}
  s\overline{X}_{F} : \slTree_{\iso}[n]^{op} & \longrightarrow \Ch, \qquad (T,\iota) \longmapsto \bigotimes_{v\in V_{d}(T)} A(|v|) \otimes \bigotimes_{v\in V(T)} X(|v|) \otimes \bigotimes_{v\in V(T)} B(|v|); \\
  sE_{1} : \slTree_{\iso}[n]                 & \longrightarrow \Ch, \qquad (T,\iota) \longmapsto \K.
\end{align*}
The free leveled two-sided bimodule functor $\calF_{l}[A,B]$ is then defined as the coend:
$$
  \calF_{l}[A,B](X)(n) \coloneqq \int^{T\in \slTree_{\iso}[n]} s\overline{X}_{F}(T)\otimes sE_{1}(T).
$$

An element in $\calF_{l}[A,B](X)(n)$ is the data of a leveled tree with section $T=(T,\iota)$ together with a family $\{x_{v}\}$, with $v\in V(T)$, of elements in the symmetric sequence $X$ (resp. the symmetric sequences $A$ and $B$) indexing the vertices on the main section (resp. below and above the main section). Such an element is denoted by $[T;\{x_{v}\}]$. The bimodule structure is defined by
\begin{align*}
  \gamma''_{L} : \calF_{l}(A)(k) \otimes  \calF_{l}[A,B](X)(n_{1}) \otimes \cdots \otimes \calF_{l}[A,B](X)(n_{k}) & \longrightarrow \calF_{l}[A,B](X)(n_{1}+\cdots + n_{k}),                                               \\
  [T_{0}; \{x_{v}^{0}\}], \bigl\{[T_{i}; \{x_{v}^{i}\}] \bigr\}_{i\leq k}                                          & \longmapsto \bigl[ \gamma_{L}(T_{0},\{T_{i}\}); \{x_{v}^{i}\}_{0\leq i \leq k}^{v\in V(T_{i})} \bigr]; \\
  \gamma''_{R} : \calF_{l}[A,B](X)(k) \otimes \calF_{l}(B)(n_{1})\otimes \cdots \otimes \calF_{l}(B)(n_{k})        & \longrightarrow \calF_{l}[A,B](X)(n_{1}+\cdots + n_{k}),                                               \\
  [T_{0}; \{x_{v}^{0}\}], \bigl\{ [T_{i}; \{x_{v}^{i}\}] \bigr\}_{i\leq k}                                         & \longmapsto \bigl[ \gamma_{R}(T_{0}, \{T_{i}\}); \{x_{v}^{i}\}_{0\leq i\leq k}^{v\in V(T_{i})}\bigr],
\end{align*}
where $\gamma_{L}$ and $\gamma_{R}$ are the operations \eqref{eq:def-gamma-l} and \eqref{eq:def-gamma-r}, respectively.
These operations are well defined since a point $\Phi$ is an equivalence class up to permutations of permutable levels and contractions of permutable levels.

\begin{defi}\label{def:leveled-cobar-sided}
  Let $\calP$ and $\calQ$ be two 1-reduced dg-cooperads.
  The leveled two-sided cobar construction of a dg-cobimodule $M$ in the category $\Ch$ is given by:

  $$
    \Omega_{l}[\calP,\calQ](M) \coloneqq \bigl( \calF_{l}[\Sigma^{-1} \calU(\overline{\calP}), \Sigma^{-1} \calU(\overline{\calQ})](\calU(M)), \; d_{\mathrm{int}}+d_{\mathrm{ext}} \bigr).
  $$
  The degree of an element $[T; \{x_{v}\}]$ is the sum of the degrees of the elements indexing the $\ge 3$-valent vertices (not on the main section) of $T$ minus $1$:
  $$
    \deg([T; \{x_{v}\}]) = \sum_{v\in V_{\ge2}(T) \setminus V_{\iota}(T)} (\deg(x_{v}) - 1) + \sum_{v \in V_{\iota}(T)} \deg(x_{v}).
  $$
  The differential $d_{\mathrm{int}}$ is the differential corresponding to the differential algebras $\mathcal{U}(\overline{\calP})$, $\mathcal{U}(M)$ and $\mathcal{U}(\overline{\calQ})$, while $d_{\mathrm{ext}}$ consists in splitting a level into two consecutive levels using the cooperadic structures of $\calP$ and $\calQ$ (on one of the vertices with $\ge2$ incoming edges of the level involved) as well as the cobimodule structure of $M$ (on any of the vertices on the main section). The bimodule structure is induced from the free two-sided bimodule functor $\calF_{l}[\Sigma^{-1}\mathcal{U}(\calP), \Sigma^{-1} \mathcal{U}(\calQ)]$.
\end{defi}

\begin{pro}
  The leveled two-sided cobar construction $\Omega_{l}[\calP,\calQ](M)$ of a dg-($\calP$-$\calQ$)-cobimodule $M$ is a well defined dg-($\Omega_{l}\calP$-$\Omega_{l}\calQ$)-bimodule.
\end{pro}

\begin{proof}
  The proof is an extension of the proof of Proposition~\ref{pro-cobar-bimod-well-def}.
  The differential $d_{\mathrm{ext}}$ is the unique derivation induced by the map which sends an element $x \in \calU(M)(n)$ to all possible ways of decomposing it using either the left or the right comodule structure of $M$ (indexing trees with exactly two levels).
  This differential squares to zero thanks to the coassociativity of the cobimodule structure of $M$ and the compatibility between the left and right coactions.
\end{proof}

\paragraph{Comparison with the Boardman--Vogt resolution for Hopf $\Lambda$-cobimodules}

Let $\calP$ and $\calQ$ be two 1-reduced Hopf cooperads and let $M$ be a Hopf ($\calP$-$\calQ$)-cobimodule.
In Section \ref{SectBVCobi} we built a fibrant resolution of $M$ using the leveled Boardman--Vogt resolution $W_{l}M$.
In Sections~\ref{SectBarCoop}, \ref{SectCobarCDGA}, \ref{sec:two-sided-leveled}, and \ref{SectCobarComod}, we introduced leveled versions of the two-sided bar and cobar constructions, respectively.
In the following, we show that $W_{l}M$ is weakly equivalent to the two-sided bar-cobar construction of $M$.
Namely, we build an explicit weak equivalence of dg-($\calB_{l}\Omega_{l}(\calP)$-$\calB_{l}\Omega_{l}(\calQ)$)-cobimodules:
$$
  \Gamma:\calB_{l}[\Omega_{l}(\calP), \Omega_{l}(\calQ)]\bigl( \Omega_{l}[\calP,\calQ](M) \bigr) \longrightarrow W_{l}M,
$$
where the dg-($\calB_{l}\Omega_{l}(\calP)$-$\calB_{l}\Omega_{l}(\calQ)$)-cobimodule on $W_{l}M$ is induced by the maps of cooperads (see Section~\ref{SectCobarCDGA})
$$
  \calB_{l}\Omega_{l}(\calP)\longrightarrow W_{l}\calP \qquad \text{and} \qquad \calB_{l}\Omega_{l}(\calQ)\longrightarrow W_{l}\calQ.
$$

A point in the two-sided bar-cobar construction is a family of elements $\Phi=\{\Phi(T)\in \overline{\Omega_{l}[\calP,\calQ](M)}(T),\,\,T\in \slTree[n]\}$, indexed by the set of leveled trees with section $\slTree[n]$, and satisfying the following relations:
\begin{itemize}
  \item for each permutation $\sigma$ of permutable levels, one has $\Phi(T)=\Phi(\sigma\cdot T)$;
  \item for each morphism $\delta_{i}:T\rightarrow T/\{i\}$ contracting two permutable levels, one has $\Phi(T)=\Phi(T/\{i\})$.
\end{itemize}
For each leveled tree with section $T$, $\Phi(T)$ is the data of a family of leveled trees and leveled trees with section $\{T[v],\,\,v\in V(T)\}$ in which the vertices of the main section of the leveled tree with section $T$ (resp. below and above the main section) are indexed by leveled trees with section (resp. leveled trees).
Moreover, one has a family $\{X_{u}[v]\}$, with $v\in V(T)$ and $u\in V(T[v])$, of elements in the cooperads $\calP$, $\calQ$ and the cobimodule $M$ labelling the vertices of the sub-leveled trees $T[v]$.
To be explicit, $\Phi(T)$ is denoted by $\{[\{T[v]\}\,,\,\{x_{u}[v]\}]\}_{v\in V(T)}$.

Let $\Phi$ be a point in the two-sided bar-cobar construction. In order to define
\[\Gamma(\Phi) \coloneqq \{\Gamma(\Phi)(T)\in sE_{W}(T)\otimes s\overline{M}_{W}(T)\}_{T\in \slTree[n]}, \]
there are two cases to consider.
\begin{enumerate}
  \item If there is no leveled tree with section $T'\in \lTree[n]$ such that $T$ is of the form $\gamma_{T'}(\{T[v]\})$ up to permutations and contractions of permutable levels, then $\Gamma(\Phi)(T)=0$.
  \item Otherwise, if $T = \gamma_{T'}(\{T[v]\})$, then we set
        $$
          \Gamma(\Phi)(T) = \bigl[ \gamma_{T'}(\{T[v]\}); \{x_{u}[v]\}; \{p_{i}\} \bigr],
        $$
        where $\{x_{u}[v]\}$ is the family of elements labelling the sub-leveled trees $T'[v]$, with $v\in V(T')$.
        The $i$-th level in $\gamma_{T'}(\{T[v]\})$, with $i\neq \iota$, is indexed by the constant polynomial form $p_{i}(t,dt)=1$ if the $i$-th level correspond to the $0$-th level of one of the sub-leveled trees $T'[v]$.
        Otherwise, the $i$-th level is indexed by the polynomial form $p_{i}(t,dt)=dt$.
\end{enumerate}

\begin{pro}\label{ProCompaOp2}
  The map $\Gamma : \calB_{l}[\Omega_{l}(\calP), \Omega_{l}(\calQ)]\bigl( \Omega_{l}[\calP,\calQ](M) \bigr) \longrightarrow W_{l}M$ so defined is a quasi-isomorphism of dg-($\calB_{l}\Omega_{l}(\calP)$-$\calB_{l}\Omega_{l}(\calQ)$)-cobimodules.
\end{pro}

\begin{proof}
  The quasi-isomorphism is a consequence of the commutative diagram:
  \[
    \begin{tikzcd}
      M \ar[r, "\simeq"] \ar[d, "\simeq" swap] & W_{l}M \\
      \calB_{l}[\Omega_{l}(\calP), \Omega_{l}(\calQ)]\bigl( \Omega_{l}[\calP,\calQ](M) \bigr) \ar[ru, "\Gamma" swap]
    \end{tikzcd}
    \qedhere
  \]
\end{proof}

We conclude with the compatibility of $\Gamma$ with the $\Lambda$-structures (compare with Theorem~\ref{thm:lambda-compatible}).

\begin{thm}\label{thm:finale}
  Let $\calP$ and $\calQ$  be two 1-reduced Hopf $\Lambda$-cooperad and $M$ be a Hopf $\Lambda$-cobimodule over the pair $\calP$ and $\calQ$.
  There exists a $\Lambda$-costructure on the two-sided leveled bar-cobar construction
  $$
    \calB_{l}[\Omega_{l}(\calP), \Omega_{l}(\calQ)]\bigl( \Omega_{l}[\calP,\calQ](M) \bigr)
  $$
  making the map

  $$
    \Gamma : \calB_{l}[\Omega_{l}(\calP), \Omega_{l}(\calQ)]\bigl( \Omega_{l}[\calP,\calQ](M) \bigr) \longrightarrow W_{\Lambda}M,
  $$
  introduced in Proposition \ref{ProCompaOp2}, into a quasi-isomorphism of dg-$\Lambda$-cobimodules.
\end{thm}

\begin{proof}
  The proof is similar to the one of Theorem~\ref{thm:lambda-compatible}, with the necessary adjustments.
  Let $h[i] : [n] \to [n+1]$ be the injective nondecreasing map that misses $i$.
  We first define auxiliary operations
  \begin{equation}\label{eq:histar-bim}
    h[i]_{*} : \Omega_{l}[\calP,\calQ](M)(n) \to \Omega_{l}[\calP,\calQ](M)(n+1).
  \end{equation}
  For $(T, \iota) \in \slTree[n]$, we define $(T,\iota)[i] \subset \slTree[n+1]$ to be the set of leveled $(n+1)$-trees with section $(T', \iota')$ such that $T$ can be obtained from $T'$ by removing the branch coming from the $i$-th leaf and levels composed of bivalent vertices (compare with the pictures in the definition of $h[i]_{*}$, Equation~\eqref{Lambdacostructure}).
  For such a tree $(T', \iota')$, we let $v_{i} \in V(T')$ be the first vertex on the path joining the $i$th leaf to the root which has at least two incoming edges.
  Given $\underline{x} = [T; \; \{x_{v}\}] \in \Omega_{l}[\calP,\calQ](M)(n)$, we define $\eta_{i}(\underline{x}, T', \iota') \in \Omega_{l}[\calP,\calQ](M)(n+1)$ as follows.
  The decorations of the vertices other than $v_{i}$ come from $\underline{x}$.
  \begin{enumerate}
    \item If $|v_{i}| = 2$ and $v_{i}$ is not on the main section, then the decoration of $v_{i}$ is the image of $1 \in \calQ(1) = \K$ under one of the $\Lambda$-costructure map $\calQ(1) \to \calQ(2)$, $\calP(1) \to \calP(2)$ depending on whether 1. $v_{i}$ is above, below or on the main section $\iota'$ and 2. the branch is on the left or the right.
    \item If either $v_{i}$ is not on the section $\iota'$ and $|v_{i}| \ge 3$ or $v_{i}$ is on the main section and $|v_{i}| \ge 2$, then the decoration of $v_{i}$ is obtained by applying the $\Lambda$-costructure map of $\calP$, $\calQ$ or $M$ to the decoration of the vertex corresponding to $v_{i}$ in $\underline{x}$.
  \end{enumerate}
  Moreover, if a bivalent vertex on the main section is created, then it is decorated by the element of $M(1)$ defined by the $\Lambda$-costructure map $\varepsilon : \K \to M(1)$.

  We can now define the map \eqref{eq:histar-bim} by:
  \[ h[i]_{*}(\underline{x}) \coloneqq \sum_{(T', \iota') \in (T,\iota)[i]} \eta_{i}(\underline{x}, T', \iota'). \]

  Using these auxiliary maps, we can define the $\Lambda$-costructure on the leveled bar-cobar construction exactly as in Theorem~\ref{thm:lambda-compatible}.
  We can then check easily that this $\Lambda$-costructure is compatible with $\Gamma$.
\end{proof}

\paragraph{Acknowledgments}
The authors would like to thank Clemens Berger, Matteo Felder, Benoit Fresse, Muriel Livernet, Bruno Vallette, and Thomas Willwacher for helpful discussions and comments.
R.C.\ and N.I.\ acknowledge partial support from ERC StG 678156--GRAPHCPX, two CNRS PEPS grants, and hospitality from ETH Zurich.
N.I.\ contributes to the IdEx University of Paris ANR-18-IDEX-0001.
J.D.\ acknowledges support from ERC StG 678156--GRAPHCPX.

\printbibliography
\vspace{20pt}

\noindent Ricardo Campos: IMAG, Université de Montpellier, CNRS, Montpellier, France\\
\textit{E-mail address: } \href{mailto:ricardo.campos@umontpellier.fr}{ricardo.campos@umontpellier.fr}

\vspace{20pt}

\noindent Julien Ducoulombier: Max Planck Institute for Mathematics, Vivatsgasse 7, 53111 Bonn, Germany\\
\textit{E-mail address: } \href{mailto:julien@mpim-bonn.de}{julien@mpim-bonn.de}

\vspace{20pt}

\noindent Najib Idrissi: Université de Paris and Sorbonne Université, IMJ-PRG, CNRS, F-75006 Paris, France\\
\textit{E-mail address: } \href{mailto:najib.idrissi-kaitouni@imj-prg.fr}{najib.idrissi-kaitouni@imj-prg.fr}
\end{document}